%% file: ArxivUpdated_IP.tex
\documentclass[12pt]{article}

\usepackage{amssymb, amsmath, amsfonts, amsthm, comment}

\newtheorem{theorem}{Theorem}
\newtheorem{lemma}{Lemma}

\newtheorem{corollary}{Corollary}

\newtheorem{proposition}{Proposition}
\newtheorem{definition}{Definition}
\newtheorem{example}{Example}

\include{NABdefs}

\usepackage{oupbib}

\begin{document}

\title{Consistency of the posterior distribution in generalised linear inverse problems}

\author{Natalia Bochkina\\ University of Edinburgh, UK}

\maketitle



\input{IntroIP_New}

\input{ModelIP}

\input{KyFanDefs}

\input{ConvergenceKyFan_New}

\input{Example_ExpFamily}


\input{NGrows_IP}

\input{PureBoundary}

\input{Discussion_IP}

\section*{Acknowledgements}

This work has arisen from ongoing joint research with Peter Green (University of Bristol, UK) on ill-posed inverse problems with a non-regular likelihood in a more general setting where some coordinates of $x^\star$ are on the boundary. The author gratefully acknowledges financial support for research visits provided by the EPSRC-funded SuSTaIn programme at Bristol University.


\input{ArxivUpdated_IPbib.bbl}

\newpage
\appendix

\section{Appendix: proofs}

We will use the following notation throughout the proofs:
\begin{eqnarray}
 H &=& A^T V_{y}(x^\star) A +\nu B(x^\star), \label{def:H}\\
 H_{\nu} &=& A^T V_{y\exact}(x^\star) A +\nu B(x^\star), \label{def:Hnu}\\
 x_0 &=& H^{-1} \nabla h_y(x^\star), \label{def:x0}\\
\Phix(B, \Omega) &=&  \frac{[\det(\Omega)]^{1/2}}{(2\pi)^{p/2}} \int_B e^{ - (x - \tau^{-1/2}\Omega^{-1} H x_0 )^T \Omega (x -\tau^{-1/2}\Omega^{-1} H x_0)/2} dx, \label{def:Phix0}\\
\Phi(B) &=&  \frac{1}{(2\pi)^{p/2}} \int_B e^{ - ||x||^2/2} dx, \label{def:Phi}
\end{eqnarray}
where $B \subset \mathbb{R}^p$ and $\Omega$ is a positive definite $p\times p$ matrix. In particular, a simple geometric argument implies
$$
\Phix(B(0, R), \Omega) \geq \Phi(B(0,r))
$$
where $r = \sqrt{\lambda_{\min}(\Omega)}(R -   \tau^{-1/2}||\Omega^{-1} H x_0||)$ and $B(z, R) = \{x: \, ||x- z|| \leq R\}$.



\input{Proof_KyFanConvIP_New}

\input{Aux_KyFanDist}

\input{Proof_DataKyFan}

\input{Proof_PureBoundary}

\input{Aux_LemmasIP}

\end{document}

%% file: NABdefs.tex
\newcommand{\textcite}{\cite}

\newcommand{\G}{\mathcal{G}}
\newcommand{\A}{\mathcal{A}}

\newcommand{\N}{\mathcal{N}}

\renewcommand{\P}{\mathbb{P}}

\newcommand{\X}{\mathcal{X}}
\newcommand{\Y}{\mathcal{Y}}

\newcommand{\W}{\mathcal{W}}

\newcommand{\pr}{\mathbb{P}}
\newcommand{\PP}{\mathbb{P}}
\newcommand{\E}{\mathbb{E}}
\newcommand{\HH}{\mathbb{H}}

\newcommand{\diag}{\mbox{diag}}

\newcommand{\rank}{\mbox{rank}}
\newcommand{\Var}{\text{Var}}

\newcommand{\trace}{{\rm tr} }
\newcommand{\Exp}{\text{Exp}}

\newcommand{\PA}{P_{A^T}}

\newcommand{\UU}{U}

\newcommand{\argmin}{\text{argmin}}

\newcommand{\true}{_{\rm true}}
\newcommand{\exact}{_{\rm exact}}

\newcommand{\Phix}{\Phi_{\omega}}

\newcommand{\ETA}{\eta}

\newcommand{\f}{{\mathsf{f}}}
\newcommand{\g}{{\mathsf{g}}}

\newcommand{\fk}{{x\true} }
\newcommand{\gk}{{y\exact}}

%% file: IntroIP_New.tex
\section{Introduction}

\subsection{Bayesian approach to ill-posed inverse problems}\label{sec:intro1}

Inverse  problems encountered in nature are commonly ill-posed:
their solutions fail to satisfy at least one of the three
desiderata of existing, being unique, and being
stable. Thus, in the case of linear inverse problems, the focus is
not on a unique solution $x$ of
\begin{equation}\label{eq:linip} y=Ax, \end{equation} for given
matrix $A$ and data vector $y$, but rather on the corresponding
space of solutions.

Even when the solution $x$ to (\ref{eq:linip}) exists
and is unique for each possible $y$, lack of stability means that
the solution can be extremely sensitive to small errors, either in
the observed $y$ or in numerical computations for solving the
equations. This has obvious deleterious consequences for the
practical value of solutions. To circumvent this, the inverse
problem is typically regularised, that is, re-formulated to
include additional criteria, such as smoothness of
the solution. If the data is observed with error, i.e. there is a probability distribution of observations has density (with respect to the Lebesgue or  counting measure) $p(y \mid A x)$, then the regularised solution to the inverse problem is
 \begin{equation}\label{eq:linpen}
 x = \argmin [ -2\log p(y \mid Ax)+ \lambda\,\text{pen}(x)]
 \end{equation}
 where $\lambda$ a positive constant
determining the trade-off between accuracy and smoothness. For example, for independent Gaussian random errors with variance $\sigma^2$, the equation above gives a regularised least squares solution since $-2\log p(y \mid Ax) = ||Ax - y||^2/\sigma^2$.
Such solutions make sense, and are commonly used, whether we
regard the error in the data used as deterministic or stochastic
in nature. The best known regularisation for ill-posed inverse problems is Tikhonov regularisation with pen$(x) =  ||x||^2$  \cite{Tikh63}. For reviews of regularisation methods and further details, see \textcite{NychkaCox}, \textcite{IP96} and \textcite{KaipioSomersalo}.

In practice, $x$ and $y$ are finite dimensional vectors that represent discretisation, or finite dimensional approximation, of functions $\f$ and $\g$ respectively of an infinite dimensional problem $\g = \A \f$, and matrix $A$ is the corresponding discretisation of operator $\A$. There are two main discretisation frameworks considered in the literature. One is discretisation on functions and an integral operator $\A$ on a grid: for $\f: \mathbb{R}^{d} \to \mathbb{R}$,  $\g: \mathbb{R}^{d} \to \mathbb{R}$, the discretisation occurs at points $u_j \in \mathbb{R}^{d}$, $j=1,\ldots,p$, and $t_i \in \mathbb{R}^{d}$, $i=1,\ldots,n$, respectively, i.e. we have $x_j = \f(u_j)$, $y_i = \g(t_i)$. An integral operator can be written as
$$
(\A \f)(t) = \int K(t, u) \f(u) du
$$
for some function $K(t, u)$, and then $A_{ij} = K(t_i, u_j)$. For studies of this and other types of discretisation effect, see \textcite{JSDiscr}, \textcite{MathePerDiscr} and \textcite{MPdiscr}. \textcite{KaipioDiscr} studied discretisation in Bayesian models with Gaussian error and Gaussian prior; they also showed that an appropriately chosen discretisation can regularise an ill-posed problem.

Another type of discretisation is achieved by a using finite dimensional approximations of $\f$ and $\g$ in (possibly different) orthonormal bases of $L^2(d\mu)$ $\{\phi_j\}_{j=1}^{\infty}$ and $\{\psi_j\}_{j=1}^{\infty}$ with some measure $\mu$. Suppose that there exist $(x_j)_{j=1}^{\infty}$ and $(y_i)_{i=1}^{\infty}$ such that
$$
\f(u) =\sum_{j=1}^{\infty} x_{j} \phi_j(u),\quad   \g(t)= \sum_{i=1}^{\infty} y_{i} \psi_i(t),
$$
and that the operator $\A$ can be written as
$$
(\A \f)(t) =\sum_{j=1}^{\infty} \sum_{i=1}^{\infty} a_{ij} \langle \f, \phi_j\rangle \psi_i(t)
$$
for some coefficients $(a_{ij})$. Then, given a finite number of observed coefficients of $\g$, $y=(y_1\ldots, y_n)$, the finite dimensional problem becomes
$ y = A x$ where  $x=(x_1,\ldots, x_p)$ and $A \in \mathbb{R}^{n\times p}$ is a matrix with entries $A_{ij} = a_{ij}$, and the recovered coefficients $x$ are used to construct an approximation to $\f$. Given $n$, the number of coefficients $p$ necessary to recover function $\f$ well enough depends on smoothness of $\f$ and on ill-posedness of $\A$.

The latter framework is often used if an infinite-dimensional operator  $\A: \, \HH_{\f} \to \HH_{\g}$ is compact, $\f\in \HH_{\f}$, $\g \in \HH_{\g}$, where $\HH_{\f}$ and $\HH_{\g}$ are separable Hilbert spaces  with orthonormal bases $\{\phi_j\}_{j=1}^{\infty}$ and $\{\psi_j\}_{j=1}^{\infty}$ that are the eigenfunctions of self-adjoint operators $\A^* \A$ and $\A \A^*$, respectively (here $\A^*$ is conjugate to operator $\A$). The corresponding estimator of $\f$ is called a spectral cutoff estimator \cite{Munketal}. This is the framework considered, among others, in \textcite{JSpet},  \textcite{DiggleHall}, \textcite{MairRuym}, \textcite{vdVaart2011}, \textcite{Stuart_Mild}, and \textcite{Ray_Inv}. Other orthonormal bases can be used, e.g. wavelet or wavelet-vaguelette bases (e.g. \textcite{JohnstoneInv}, \textcite{CavalierPET}).


Smoothness, or other `regular' behaviour of the solution to an
inverse problem, is a prior assumption on the
unknown $x$, information about the model parameters known or
assumed before the data are observed. To use such information is
thus to accept that the required solution must combine data with prior
information. In a statistical context
the best-established principle for doing this is the
Bayesian paradigm, in which all sources of variation, uncertainty
and error are quantified using probability.

From this  perspective, the solution to (\ref{eq:linpen}) is
immediately recognisable -- it is the maximum a posteriori (MAP)
estimate of $x$, the mode of its posterior distribution in a
Bayesian model in which the data $y$ are modelled with a
Gaussian distribution with expectation $Ax$, with
constant-variance uncorrelated errors, and in which the prior
distribution of $x$ has negative log-density proportional to
$\text{pen}(x)$.

However, the Bayesian perspective brings more than merely a
different characterisation of a familiar numerical solution.
Formulating  a statistical inverse problem as one of inference in
a Bayesian model has great appeal, notably for what this brings in
terms of coherence, the interpretability of regularisation
penalties, the integration of all uncertainties, and the
principled way in which the set-up can be elaborated to encompass
broader features of the context, such as measurement error,
indirect observation, etc. The Bayesian formulation comes close to
the way that most scientists intuitively regard the inferential
task, and in principle allows the free use of subject knowledge in
probabilistic model building (e.g. \textcite{BIP9}, \textcite{BIP2}, \textcite{BIP5}, \textcite{BIP7} and  \textcite{BIP6}). For an interesting philosophical
view on inverse problems, falsification, and the role of Bayesian
argument, see \textcite{Tarantola}. Various Bayesian methods to solve inverse problems in practice have been proposed, for instance, in \textcite{BIP4}, \textcite{BIP1}, \textcite{BIP3}, \textcite{BIP7}, \textcite{BIP8}.
For a review of Bayesian methods in inverse problems, see \textcite{StuartReview}.

\subsection{Consistency of the posterior distribution}


\textcite{DiaconisFriedman} were one of the first to raise possible  issues with consistency of Bayesian estimators in infinite dimensional problems, drawing attention to importance of choosing the prior distribution appropriately. Furthermore, Bayesian estimators that achieve the optimal rate of convergence in the minimax sense under the Gaussian errors have been constructed and studied by \textcite{vdVaart2011} and \textcite{Stuart_Mild} for mildly ill-posed inverse problems, and \textcite{Stuart_Severe} for severely ill-posed inverse problems; in all these cases a conjugate Gaussian prior was used.  Considered estimators were non-adaptive, in the sense that in order to achieve the optimal rate of convergence, the prior distribution must depend on the smoothness of the unknown function. \textcite{vdVaart2011} have also studied the problem of estimating linear functionals that can achieve up the  parametric rate for some functionals, and coverage of the Bayesian credible intervals. \textcite{Ray_Inv} studied the linear inverse problem under the assumption of Gaussian errors with a non-conjugate prior distribution that can be adaptive.  The minimax rates of convergence for linear inverse problems under the assumption of Gaussian noise were established by \textcite{Cavalier}.


Theoretical properties of inverse problems with non-Gaussian errors have been studied very little (as far as we are aware). In a frequentist approach, consistency and the minimax rate of convergence of estimators in inverse problems with Poisson errors was studied by  \textcite{JSpet} (in the context of density estimation), \textcite{CavalierPET} (as a regression problem) where the minimax rate coincided with the minimax rate of convergence under Gaussian errors, and
\textcite{Munketal} considered the case of an abstract error being a Hilbert-space process (a continuous linear operator from a Hilbert space to the probability space $L^2(\Omega, {\cal F}, \pr)$), with examples of white noise Hilbert processes. However, the rate of convergence in nonparametric estimation in non-Gaussian problems can be different, and in fact they can be faster. This has been shown in the direct observations problem, by \textcite{Bouret} for a nonparametric estimator of the intensity of a Poisson process and by \textcite{Chichignoud} for Bayesian-type estimators in the nonparametric regression problem with multiplicative errors.


The motivating example for this paper is a Bayesian solution to the inverse problem arising in tomography with Poisson variation with a non-Gaussian prior where both observations and the solution are discretised to a regular grid \cite{Green}. This example fits in the framework considered in this paper that is described below.

\subsection{Main results of the paper}

{\bf Summary}. In this paper we prove consistency of Bayesian solutions of an ill-posed linear inverse problem in the Ky Fan metric for a general class of likelihoods and prior distributions, with unknown $x\in \X \subseteq \mathbb{R}^p$ that is to be recovered from noisy observations $y \in \Y \subseteq \mathbb{R}^n$. The finite dimensional setting can be achieved by discretisation or approximation of an infinite dimensional problem, and the result applies to the model with increasing dimensions. Likelihood and prior distributions are assumed to be differentiable but without any assumption of finite moments of observations $y$, such as expected value or the variance, thus allowing for possibly non-regular likelihoods. The likelihood belongs to a class of generalised linear inverse problems that includes distributions from exponential family with possibly dependent observations. The prior distribution does not have to be conjugate and may be improper.
We observe quite a surprising phenomenon when applying our result to the spectral approximation framework where it is possible to achieve the parametric rate of convergence, i.e the problem becomes self-regularised.

{\bf Exact problem}. This model is a noisy version of the following exact problem
\begin{equation}\label{eq:exactInDim}
y\exact = G(Ax),
\end{equation}
where  matrix $A \in \mathbb{R}^{n\times p}$ may not be of full rank and function $G: \, A\X \subseteq \mathbb{R}^n \to G(A\X)\subseteq \mathbb{R}^n$ is assumed to be invertible and twice differentiable, and $y\exact = G(Ax\true)$. Due to the stated assumption on function $G$, this inverse problem can be written as a linear problem $G^{-1}(y\exact) = Ax$ with different data. Sets $\X$ and $\Y$ do not have to be open; we do assume however that $x^\star$ defined by \eqref{eq:DefXstar} is an interior point of $\X$ in all sections apart from Section~\ref{sec:boundary}.

Here $A$, $x$ and $y\exact$ can be viewed as a discretised operator and discretised unknown and observed exact functions, respectively, or as the matrix of a finite number of singular values of the infinite dimensional operator $\A$, a vector of $p$ coefficients of the unknown function $\f$ and a vector of $n$ noisy coefficients of the observed function $\g$ in the spectral cutoff estimator, as described in Section~\ref{sec:intro1}.

{\bf Likelihood.} Distribution of noisy observations $y = (y_1,\ldots, y_n)$ is assumed to follow the framework of {\it generalised linear inverse problems} introduced by \textcite{BochkinaGreen}
$$
p_{\tau}(y \mid \eta) = C_{y,\tau} \exp\left\{-\mathring{f}_{y}(\eta)/\tau \right\},
$$
where $\eta = (\eta_1,\ldots, \eta_n)$,    $\tau$ is a dispersion parameter, and the true value of parameter $\eta$ is $y\exact$. The key assumptions about the likelihood are that dispersion $\tau$ can also be interpreted as a level of noise, i.e. that $Y\stackrel{\P}{\to} y\exact$ as $\tau \to 0$, and that function $\mathring{f}_{y}(\eta)$ is three times differentiable and it, together with its derivatives, has a finite limit as $\tau \to 0$. The precise conditions are stated in Sections~\ref{sec:Setup} and \ref{sec:QuadraticApprox}. Note that $\E Y_i$ or $\Var(Y_i)$ do not have to exist for the assumption $Y\stackrel{\P}{\to} y\exact$ as $\tau \to 0$ to hold.

These assumptions hold for a distribution from the exponential family (see Section~\ref{sec:ExpFamily}) with the probability density or mass function
$$
p_{\tau}(y \mid \eta) = \exp\left\{-\sum_{i=1}^n [ y_i b_i(\eta_i) - c_i(\eta_i)]/\tau + d_i(y_i, \tau) \right\}
$$
with $\eta_i =\E Y_i $ which implies $\Var(Y_i) = -\tau [b_i'(\eta_i)]^{-1}$ (see Section~\ref{sec:ExpFamily}).

 Following the exact problem, we use link function $G$: $\eta = G(Ax)$ to introduce dependence of the data on $x$.  Function $G$ can be used to model dependence between observations $y_1,\ldots, y_n$. If matrix $A$ is ill-posed, i.e. matrix $A^T A$ is not of full rank, then the likelihood for $x$
 $$
 p_{\tau}(y  \mid x) = C_{y,\tau} \exp\left\{-\mathring{f}_{y}(G(Ax))/\tau \right\},
 $$
 is not identifiable. In the motivating tomography example \cite{Green}, discretised observations $y_i$ have Poisson distribution with the identity link function $G$, $x\in \X = [0,\infty)^p$, and matrix $A$ is discretisation of the Radon operator.



{\bf Prior.} We assume that the prior distribution with smooth density $p(x)$ is of the form $p(x) = C_{\gamma} \exp(-g(x)/\gamma^2)$, $x\in \X \subseteq \mathbb{R}^{p}$, for some function $g$. This prior distribution can be improper. Full conditions imposed on the prior to derive the main result are given in Section~\ref{sec:SectionSetup}. In the motivating example \cite{Green}, the prior density, that is a non-Gaussian pairwise interaction Markov random field, satisfies this condition and is not identifiable with respect to a constant shift, i.e.   vectors in $\X$ with coordinates $x_i$ and $x_i+c$, $i=1,\ldots,p$ for any constant $c$  have the same density.

{\bf Posterior.} The resulting posterior distribution can be written as
$$
p_{\tau}(x\mid y) = C_{y,\tau,\gamma} \exp\left\{-[\mathring{f}_{y}(G(Ax)) + \frac{\tau}{\gamma^2} g(x) ]/\tau \right\}.
$$
The mode of the posterior distribution can be interpreted as a regularised solution of the original exact problem that takes into account a possibly non-Gaussian  and non-additive observation error that was discussed in Section~\ref{sec:intro1}. Other point summaries of the posterior distribution can be used as an estimator, such as the mean or the median of the posterior distribution that can be easier to compute or more robust to error misspecification in practice. In the paper we suppress index $\tau$ to simplify the notation.

{\bf Ky Fan metric.} To measure the convergence of the posterior distribution, we use the approach considered by \textcite{EHK05}, \textcite{HofingerP:07}, \textcite{HofingerP:09} in the context of linear inverse problems which is to metrise weak convergence of the posterior distribution as a random variable
$\mu_{\rm post}(\omega) = p(x \vert Y(\omega))$
using the Ky Fan metric \cite{Fan}.
In particular, the Ky Fan metric $\varepsilon_{\tau}$ between the posterior distribution $\mu_{\rm post}(\omega)$ and a point mass $\delta_{x^\star}$ satisfies, with probability at least $1 - \rho_{\rm K}(Y , y\exact)$,
$$
\P( d(x, x^\star) \leq \varepsilon_{\tau}  \mid Y) > 1 - \varepsilon_{\tau} \quad \text{on} \quad \{\omega: \, d(Y(\omega), y\exact)\leq \rho_{\rm K}(Y, y\exact)\},
$$
where $\rho_{\rm K}(Y, y\exact)$ is the Ky Fan distance between the data \(Y\) and its small noise limit $y\exact$, for some distance $d$, and $x^\star$ is defined by \eqref{eq:DefXstar}.
This metric metrises convergence in probability which is weaker than the almost sure convergence. Convergence rates in this metric are slower than rates under the mean squared error loss (for instance, for a Gaussian distribution there is  unavoidable extra log  factor).
For comparison, condition on the almost sure concentration rate $\varepsilon$ of the posterior distribution is defined as
$\P( d(x, x^\star) \leq M \varepsilon   \mid Y) \stackrel{\P}{\to} 1$ as $\tau \to 0$ for large enough $M$ (e.g., \textcite{GhosalGV}). Using the Ky Fan metric, we have explicit expressions for $M \varepsilon$ and for the event on which convergence takes place. Also, this allows to make weaker assumptions on the model, e.g. $\E Y_i$ and $\Var(Y_i)$ do not have to be finite. As we shall see, in the considered models, the Ky Fan rate of convergence coincides with the almost sure concentration rate up to a log factor. The Ky Fan metric is discussed in Section~\ref{sec:metrics}.

{\bf Main result.} The main result of the paper (Theorem~\ref{th:KyFanPostInt}) is that for sufficiently small noise level $\tau$ and the prior scale $\gamma$ such that both $\gamma$ and $\tau/\gamma^2$ are small, the posterior distribution concentrates around $x^\star$ defined by
\begin{eqnarray}\label{eq:DefXstar}
x^\star = \arg \min_{x\in \X \, Ax = Ax\true} g(x)
\end{eqnarray}
at the Ky Fan rate of convergence that is bounded (up to an absolute constant) by
$$
c \rho_{\rm K}(Y , y\exact) +   \frac{\tau}{ \gamma^2} ||H_{\nu}^{-1} \nabla g(x^\star)|| + \sqrt{R_{\tau}\log(1/R_{\tau}) }
$$
with $R_{\tau} = \tau \trace (H_{\nu}^{-1})$ and
$$
H_{\nu} =    A^T \tilde{G}^T V \tilde{G} A +  \frac {\tau} {\gamma^2} \nabla^2 g(x^\star).
$$
Here $\tilde{G} =  \left(\frac{\partial }{\partial \eta_i } G_j(A x\true)  \right)$, and  matrix $V = \left(\frac{\partial^2 }{\partial \eta_i \partial \eta_j}\mathring{f}_{y\exact}(y\exact)  \right)$, $i,j=1,\ldots,n$, is an analogue of the Fisher information matrix of $Y$ for  $\eta = G(Ax)$ given the value of $x$ generating the data (``the state of nature'') is $x\true$.
 The first term represents random bias that is explicitly stated in Theorem~\ref{th:KyFanPostInt}, the second term is non-random bias caused by the prior and the last term represents the rate of concentration of the posterior distribution around its mode. As all the constants in the upper bound are given explicitly, this result can be applied to the problem with growing dimensions $n$, $p$ and decreasing eigenvalues of $A$, either as a discretisation on a grid or as a spectral cutoff problem (see Section~\ref{sec:pgrows}).



The crucial difference between the rate of convergence for Gaussian and non-Gaussian distributions is that the variance depends not only on the eigenvalues of matrix $A$ and the prior but also on the unknown function. Therefore, the rate of convergence in this case may be different, an in the direct estimation problems discussed above. The nonlinearity introduced by a nonlinear link function $G$ does not play an important role here as long as $G$ is sufficiently smooth, invertible and the eigenvalues of $\tilde{G}$ are bounded away from 0 and infinity.

{\bf Finite dimensional problem}. Applied to the well-posed problem with fixed dimension $p$ where the smallest eigenvalue separated away from zero by a positive constant independent of $\tau$, condition that $\gamma$ is small is not necessary, and the solution is recovered with the parametric (weak) rate of convergence since $R_{\tau} \asymp \tau^{-1/2}$ and the error rate $\tau$ can be interpreted as $1/n$ for Gaussian independent identically distributed random variables (note that the weak parametric rate of convergence is $\sqrt{n/\log n}$). For an ill-posed problem with fixed dimensions, $R_{\tau} \asymp \gamma$ so the smallest upper bound on the rate is of order $[\tau \log(1/\tau)]^{1/3}$ with $\gamma^2= \tau^{2/3}[\log (1/\tau)]^{-1/6}$ which is slower than the rate for a well-posed problem.

{\bf Discretisation on a grid.}
Application of the result to the two frameworks with growing dimensions $n$, $p$ is discussed in Section~\ref{sec:pgrows}. For a function defined on $[0,1]$  discretised on a regular grid,  we must have $p \geq C \left(\frac{\tau}{\log (1/\tau)}\right)^{-2/3}$ to achieve the rate $[\tau \log(1/\tau)]^{1/3}$.

{\bf Spectral approximation.}
Even though our primary motivation for considering non-Gaussian error models is to study consistency of a Bayesian estimator of $\f$ given discretisation on a grid, our model also applies to the spectral cutoff framework. For the Gaussian noise,  we can recover the minimax rate of convergence (up to a log factor) and the rate of \textcite{vdVaart2011} under a Gaussian prior.

For a non-Gaussian distribution of errors, we discover quite a surprising phenomenon.
 If the variance of the noise 
decreases fast enough as a function of frequency $i$ (faster than the eigenvalues of $\A \A^*$), then the problem can become self-regularised and the Bayesian estimator can achieve the parametric Ky Fan rate of convergence, $[\tau \log(1/\tau)]^{1/2}$. This feature is specific to heteroscedastic models where the information about the unknown function is available not only from the mean, as for the Gaussian white noise model. If the variance of the noise does not decrease faster than the eigenvalues of $\A \A^*$, then the problem remains ill-posed but the ill-posedness is less severe. It is illustrated on an example where the noise has Poisson distribution with the identity link: $Y_i/\tau \sim Pois(a_i x_i/\tau)$ where $a_i>0$ are the eigenvalues of $(\A \A^*)^{1/2}$. The following argument can be used to illustrate why the rate of convergence becomes parametric. Under the specified Poisson model,
$$
\Var(Y_i) = \tau a_i {x\true}_i 
$$
which implies that for the simplest estimator $\hat{x}_i = Y_i/a_i$, $\E \hat{x}_i = {x\true}_i$, $\Var(\hat{x}_i) = \tau  {x\true}_i/a_i$.
If $\frac{{x\true}_i}{a_i} \leq C i^{-1- \delta}$ for some $\delta>0$ and all large $i$, and all coefficients $Y_i$ are available, then we have an unbiased estimator of $\f$ with variance
$$
\sum_{i=1}^{\infty} \Var\left(\frac{Y_i}{a_i} \right) = \tau \sum_{i=1}^{\infty} \frac{{x\true}_i}{a_i} \leq C \tau \sum_{i=1}^{\infty} i^{-1-\delta}  \leq \tilde{C} \tau \to 0\quad \text{ as} \quad\tau \to 0.
$$
Hence, in this case the problem becomes self-regularised, and no regularisation by a prior is necessary. A bias caused by truncation to the first $p$ observations can be made small by an appropriate choice of $p$ as a function of $\tau$ and of smoothness of $\f$. A formal argument for the Bayesian estimator is given in Section~\ref{sec:pgrows}. It would be interesting to see if such cases occur in practice.



 We also show that the rate of convergence may change if the limiting point $x^\star$ lies on the boundary of the parameter space for a constrained inverse problem; see Section~\ref{sec:boundary} (for Gaussian noise and Tikhonov regularisation, this problem was studied by \textcite{Kunisch94}). The remarkable property is that the rate of convergence in this case is faster, as long as every coordinate of $x^\star$ lies on the boundary. This is an example that the inverse problem with non-Gaussian errors can be non-regular; for the study of consistency and the rate of convergence in a more general setting where only some of the coordinates are on the boundary see \textcite{BochkinaGreen}.

The paper is organised as follows. In Section~\ref{sec:SectionSetup} we describe the considered Bayesian model. In Section~\ref{sec:metrics} we define the Ky Fan metric and study it for various distributions. In Section~\ref{sec:ConvergenceKyFan} we state the necessary assumptions on the model and an upper bound on the Ky Fan metric between the posterior distribution and $\delta_{x^\star}$, and conditions for consistency of the posterior distribution. In Section~\ref{sec:pgrows}, we show how this result can be applied to approximations of infinite dimensional inverse problems with growing dimensions, namely to discretisation on a grid and spectral approximation  frameworks. In Section~\ref{sec:boundary} we briefly discuss the case when the point of concentration is on the boundary of the parameter space, and conclude with discussion in Section~\ref{sec:discussion}. All proofs are given in the appendix.


%% file: ModelIP.tex
\section{Model formulation}\label{sec:SectionSetup}


\subsection{Generalised linear inverse problems (GLIP) }\label{sec:Setup}

We assume that the joint density of the observable responses $Y$ (the likelihood) taking values in $\Y\subseteq \mathbb{R}^n$ (with respect
to Lebesgue or counting measure) takes the form
\begin{eqnarray}\label{eq:Lik}
p\,( y \vert\,  x) = F(y, Ax, \tau) =C_{y,\,\tau} \exp\left\{-\frac 1 {\tau } \tilde{f}_y(Ax)\right\}, \quad y\in \Y,
\end{eqnarray}
that is, that the distribution depends on $x \in \X \subseteq \mathbb{R}^p$ only via $Ax$, where $\tau$ is a scalar dispersion parameter; in the Gaussian model,
$\tau$ is the variance $\sigma^2$. The observed data $y$ are generated from this distribution, with $x=x\true$, and we aim to recover $x\true$ as $\tau \to 0$.

We assume that there is a differentiable invertible link function $G: A\X \to \mathbb{R}^n$ such that $ y\exact = G(A x\true)$. 
Function $\tilde{f}_y(Ax)$ can be written as $\tilde{f}_y(Ax) = \mathring{f}_{y}(G(Ax))$, where $\mathring{f}_{y}(\eta)$ was defined in Introduction, that is, it includes the link function implicitly to simplify the notation. 

We will also introduce $f_y(x) = \tilde{f}_y(Ax)$, for convenience of notation. Here $\tilde{f}_y(\eta)$ is negative log likelihood with respect to parameter $\eta = Ax$
and $f_y(x)$ is negative loglikelihood with respect to the parameter of interest $x$  (both up to the constant $1/\tau$).

We make the following assumptions about the error distribution:
\begin{enumerate}

\item If $Y \sim  F(y, G^{-1}(y\exact), \tau)$, then
    $Y\stackrel{\pr_{x\true}}{\to} y\exact$ as $\tau \to 0$.

\item For all $\mu_0 \in  A\X$, $\tilde{f}_{\mu_0}(\ETA)$
    has a unique minimum over $ A \X $ at $\ETA = G(\mu_0)$.


\end{enumerate}

Assumption (i) states that $\tau$ is not only the dispersion parameter in the model but also a scale parameter for the distribution of $Y$.
Assumption (ii) establishes identifiability of the likelihood with respect to the link parameter $\eta = G(Ax)$. 

Assumption (i) is satisfied by generalised linear models
\textcite{NelderW}, an important class of nonlinear
statistical regression problems,
responses $y_t$, $t=1,2,\ldots,n$ are  drawn
    independently from a one-parameter exponential family of
    distributions in canonical form, with density or
    probability function
\begin{equation}\label{eq:ExpFamilyDef}
p(y_t;\ETA_t,\tau) = \exp\left(  - \frac{y_t b(\ETA_t)-c(\ETA_t)}{\tau} +d(y_t,\tau)\right),
\end{equation}
 for  appropriate functions
$b$, $c$ and $d$ characterising the particular distribution
family. The parameter $\tau$ is a common dispersion parameter
shared by all responses. In generalised linear models commonly $G$ has identical component functions.
 The expectation and the variance of this distribution
are
\begin{equation}\label{eq:ExpFamilyMeanVar}
\E(y_t;\ETA_t,\tau)=\ETA_t=c'(\ETA_t)/b'(\ETA_t),\quad \Var(y_t;\ETA_t,\tau)=\tau \frac{  c''(\ETA_t)b'(\ETA_t) - c'(\ETA_t)b''(\ETA_t)}{[b'(\ETA_t)]^3}.
\end{equation}
Both assumptions (i) and (ii) are satisfied for this example.

\subsection{Bayesian formulation of GLIP}

We adopt a Bayesian paradigm, using a prior distribution with density given by
\begin{eqnarray}\label{eq:Prior}
p\,(x) \propto \exp(-g(x)/\gamma^2) , \quad x\in \X\subset \mathbb{R}^p,
\end{eqnarray}
where $\gamma^2$ is a scalar dispersion parameter for the prior that may depend on
$\tau$; we relate this to the data dispersion parameter $\tau$ by $\gamma^2=\tau/\nu$, and express most of our results below in terms of $\tau$ and $\nu$.
Set of possible values of the parameters $\X$ can be any subset of $\mathbb{R}^p$ that contains a nonempty neighbourhood of $x^\star$.

Therefore, the posterior distribution satisfies
\begin{eqnarray}\label{eq:Posterior}
 p\,(x \vert \, y) \propto \exp(-[\tilde{f}_y(Ax) + \nu \, g(x)]/\tau ), \quad x\in \X.
\end{eqnarray}
For convenience of notation, we also introduce function $h_y(x)$ defined by
$$
h_y(x) = f_y(x)  + \nu \, g(x),
$$
so that $p\,( x \vert\,  y) \propto e^{-h_y(x)/\tau }$. 

To summarise the Bayesian model, we have the following likelihood:
$$
p(y \mid x) = C_{y,\,\tau} \exp\left\{-\frac 1 {\tau } f_y(x)\right\}, \quad y\in \Y,
$$
prior density
$$
p(x) \propto \exp\left\{-\frac 1 {\gamma^2 } g(x)\right\}, \quad x \in \X,
$$
and the posterior density
$$
p\,( x \vert\,  y) \propto \exp\left\{-\frac{ \, f_y(x) + \nu \, g(x)}{\tau} \right\} =  \exp\left\{-\frac{h_y(x)}{\tau} \right\}, \quad x \in \X.
$$
Function $\tilde{f}_y(\eta)$ is assumed to be identifiable, and function $f_y(x)$ may be non-identifiable if rank of matrix $A$ is smaller than $p$. 


We will show that in the limit $\tau \to 0$, the posterior distribution concentrates at point $x^\star$ defined by
\begin{equation}\label{eq:xstar}
x^\star= \arg \min_{Ax = Ax\true} g(x).
\end{equation}
Below we make further assumptions on the likelihood and the prior distribution that are necessary for convergence of the posterior distribution.

%% file: KyFanDefs.tex
\section{Types of convergence and corresponding distances}
\label{sec:metrics}

Convergence in distribution (weak convergence) can be metrised by Prokhorov metric \cite{Dudley}.
\begin{definition} The Prokhorov metric between two measures on a metric space
$(\X,d_\X)$  is defined by
$$\rho_{\rm P}(\mu_1,\mu_2)=\inf\{\varepsilon>0: \mu_1(B)\leq
\mu_2(B^\varepsilon)+\varepsilon \:\forall \text{ Borel }B\}$$
where $B^\varepsilon=\{x:\inf_{z\in B} d_\X(x,z)<\varepsilon\}$.
\end{definition}
This metric can be used to study the weak convergence of the posterior distribution $\mu_{\rm
post}(\omega) = \P_{X \mid Y(\omega)}$ as a measure on $\X$ to its limit {\it for a fixed} data set $Y(\omega)$.
We consider the Euclidean metric $d(x,z) = ||x-z||$ on $\X$. 

To study a weak convergence of the posterior distribution to its limit over all $\omega$, we can use Ky Fan metric that metrised convergence in probability \cite{Dudley}.
\begin{definition}  The Ky Fan metric between two random variables $\xi_1$ and $\xi_2$ in a metric space
$(\W, d_\W)$ is defined by
$$
\rho_{\rm K}(\xi_1,\xi_2) =
\inf\{\varepsilon>0:\P(d_\W(\xi_1(\omega),\xi_2(\omega))>\varepsilon)<\varepsilon\}.
$$
\end{definition}
 Hence, weak convergence of the
posterior distribution $\mu_{\rm post}$ (as a random variable) to
$\delta_{x^\star}$, the point mass at $x^\star$, is equivalent to
its convergence in the Ky Fan metric, where the metric space
$(\W, d_\W)$ is a space of probability distributions on $\X$ equipped with the
Prokhorov metric.

Now we give the Ky Fan distance or its upper bound for some distributions.
First we consider a rescaled Poisson distribution.
\begin{lemma}\label{lem:KyFanPoisson}
Consider independent random variables $Y_t/\tau \sim Pois(\mu_t/\tau)$, $t=1,\ldots,n$, $\mu_t >0$. Denote $M=4 \sum_t \mu_t$.

Then, for $\tau$ such that $M \tau < 1/e$,
$$
\rho_{\rm K}(Y, \mu) = \sqrt{-\tau M\log(\tau M)} (1+w_{\tau}),
$$
where $w_{\tau} = o(1)$ as $\tau \to 0$ and $w_{\tau} \leqslant 0$.
\end{lemma}

Now, if we consider the exponential distribution with variance proportional to $\tau$, the order of the Ky Fan distance is different. Let $Y - \mu\sim \Exp(\lambda/\tau)$, then $\E Y = \mu + \tau/\lambda$, $\Var (Y) = \tau^2/\lambda^2$. As $\tau \to 0$, $Y \to \mu$ in probability. The Ky Fan distance is given by
$$
\rho_{\rm K}(Y, \mu) = -\frac{\tau}{\lambda}\log\left(\frac{\tau}{\lambda} \right) (1+w_{\tau}),
$$
where $w_{\tau} \leqslant 0$ and $w_{\tau} = o(1)$ as $\tau\to 0$. This follows from Lemma~\ref{lem:KyFanExp}.


Now we give some general statements on an upper bound on the Ky Fan distance for various distributions.

\begin{proposition}\label{th:KyFanData} Assume that $Y_t$ are independent, $\E Y_t =
\mu_t$ and $\Var(Y_t) = w_t \tau$.

\begin{enumerate}

\item Assume that $\exists  C_t\geqslant 1$ such that
    $\kappa_{t, k}$, the $k$th cumulant of $Y_t$, is bounded
    by $|\kappa_{t, k}| \leqslant C_t w_t \tau^{k-1}$ \,
    $\forall k \geq 2$ and $C_t$ and $ w_t$ are independent of $\tau$. Denote $M = 4\sum_t C_t w_t$.

    Then, for $ \tau\leqslant 1/(eM)$,
    $$
      \rho_{\rm K}(Y, \mu)\leqslant  \sqrt{-\tau M \log(\tau M)}.
    $$

In particular, this case applies to the distributions in the exponential family with dispersion $\tau$ defined by \eqref{eq:ExpFamilyDef}.


\item Assume that $\exists K \geqslant 2$: $\E |Y_t|^K
    <\infty$. Assume that $\E |Y_t -\mu_t|^K \leqslant
    \tau^{m(K)} L_K$ for some $L_K >0$ that may depend on
    $\mu_t$ or $w_t$ but not on $\tau$, for some $m(K)>0$.

Then, for small enough $\tau$,
$$
\rho_{\rm K} (Y,\mu) \leqslant  [n \tau^{m(K)/2 }
L_K]^{1/(K+1)}.
$$

\end{enumerate}
\end{proposition}

Here is an example for the second case.
\begin{example} Suppose $Y_t$ has a $t$ distribution with $\nu$ degrees
of freedom, means $\mu_t$ and scales  $\sqrt{\tau} w_t$, $t=1,\ldots,n$. Then we can
take $K= \nu - 2 - \delta$ for some small $\delta>0$. Then, using the second statement of Proposition~\ref{th:KyFanData},
$$
 \E |Y_t - \mu_t|^K = [\sqrt{\tau} w_t]^K \nu_K,
$$
where $\nu_K$ is the $K$th moment of the standard $t_{\nu}$ distribution, i.e. $m(K)=K/2$ and $L_K = w_t^K
\nu_K$. Hence,
$$
\rho_{\rm K} (Y, \mu)\leqslant \tau^{1/2 - 1/(2(K+1))} [n
w_t^K \nu_K]^{1/(K+1)}.
$$

Note that this bound holds if $Y_t$ can be written as $Y_t = \mu_t
+ \sigma w_t Z_t$ where $Z_t$ are iid and whose
distribution is independent of $\tau$.

\end{example}

Applying Proposition~\ref{th:KyFanData} to the Gaussian distribution, we have the following lemma.
\begin{lemma}\label{lem:KyFanNormal}
 Let $\xi \sim \N_p(\mu,\Sigma)$.
Then, for $\Sigma$ such that $\trace(\Sigma) < 1/(4e)$,
\begin{equation}
\rho_{\rm K}(\xi, \mu) \leq \sqrt{- 4\trace(\Sigma) \log\{ 4\trace(\Sigma)\}}.
\end{equation}
\end{lemma}
This bound is more precise than the one given in Lemma 7 from  \textcite{HofingerP:07}, and is more appropriate to the case of growing dimension $p$ that we study in Section~\ref{sec:pgrows}. It is easy to see that the log factor is unavoidable, even in a simple case $n=1$ and $\Sigma = \sigma^2$. Assume that $\rho_{\rm K}(\xi, \mu) = \sqrt{\tau} M$ for some $M>0$, then $M$ satisfies 
$$
\P(|\xi - \mu| >\sqrt{\tau} M) = 2(1 - \Phi(M)) = \sqrt{\tau} M
$$
since the distribution is continuous. If $\tau$ is small, we have 
$$
\sqrt{\tau} = 2 M^{-1} (1 - \Phi(M))  \leq \sqrt{2/\pi} M^{-2} e^{-M^2/2}
$$
and hence we must have $M\to \infty$ and, moreover, $M^2  = -\log (\tau) (1+o(1)) $ as $\tau \to 0$. Therefore, the upper bound is asymptotically of the correct order.

Note that the Ky Fan distance for the distributions from the exponential family has the same asymptotic order as for the Gaussian distribution with $\Sigma = \tau \Sigma_0$, $\Sigma_0$ is independent of $\tau$, as $\tau \to 0$.

%% file: ConvergenceKyFan_New.tex
\section{Rates of convergence of posterior distribution in Ky
Fan metric}\label{sec:ConvergenceKyFan}


Denote by $\mu_{\rm post} (\omega)$ the
posterior distribution of $X$ given $y=Y(\omega)$. We consider the
metric space $(\X, \ell_2)$ equipped with the Euclidean metric
$||x-z|| = \sqrt{\sum_{i=1}^p (x_i-z_i)^2}$, $\X \subset
\mathbb{R}^p$. Then, the posterior measure $\mu_{\rm post} (\omega)$
can be viewed as a measure on the metric space $(\X, \ell_2)$.
The corresponding metric space for the observations is $(\Y, \ell_2)$,
$\Y \subset \mathbb{R}^n$ equipped with metric generated by
$\ell_2$ norm.

Throughout, we use $\nabla_i = \frac{\partial}{\partial x_i}$ as the differentiating operator, and $\nabla = (\nabla_1, \dots, \nabla_p)^T$ as the gradient. Similarly, $\nabla_{ij}$ and $\nabla_{ijk}$ are operators of the second and third derivatives, with $\nabla^2 = (\nabla_{ij})$ being the matrix of second derivatives. The matrix norm $||\cdot||$ used in the paper is the spectral norm. We will also use $a \asymp b$ for asymptotic equivalence, i.e. that there exist constants $0< c \leq C <\infty$ such that $ c a \leq b \leq Ca$.

In the next section we  evaluate the level of concentration of the posterior distribution $\mu_{\rm post}$ around   $x^\star$. We start with the concentration of the posterior distribution $\mu_{\rm post}(\omega)$ for a fixed $\omega$ (i.e. for a particular data set) in the Prokhorov metric,  and then, using the lifting theorem (Theorem \ref{lem:KyFanDiffBounds}), we use bounds thus obtained to derive a bound on the Ky Fan distance between the posterior distribution and the limit over all $\omega$. 

Throughout this section, we assume that $x^\star$ is an interior point of $\X$.

\subsection{Assumptions on the likelihood and the prior}\label{sec:QuadraticApprox}

In addition to the assumptions on the likelihood stated in Section~\ref{sec:Setup}, we make the assumptions that the posterior distribution is proper, the point of concentration of the posterior distribution $x^\star$ is unique,   the log likelihood and log prior density have bounded third order derivatives and that the derivatives of the log likelihood are continuous with respect to $y$.

{\bf Assumptions on prior distribution.}

We assume that the prior distribution is such that the posterior distribution is proper.
\begin{enumerate}
\item  $\exists \tau_0 >0$: \, $\forall \tau\leqslant \tau_0$, \quad  $\int_{\X} e^{-h_y(x)/\tau} dx < \infty$  for all $y\in
    \Y$.
\item $x^\star = \arg \min_{x\in \X \, Ax = Ax\true } g(x)$ is a unique solution of the minimisation problem.
\end{enumerate}
The first assumption is that the posterior distribution is proper. The second assumption is that the regularisation by the chosen prior leads to a single solution; it is satisfied if $g$ is a convex function and $\X$ is a convex set. 

Define $\Y_{\rm loc}$ as the following neighbourhood of $y\exact$ in $\Y$:
\begin{eqnarray}\label{def:Yloc}
\Y_{\rm loc} = \{y\in \Y: \quad ||y - y\exact|| \leq \rho_{\rm K}(Y, y\exact)\}
\end{eqnarray}
where  $\rho_{\rm K}(Y, y\exact)$ is the Ky Fan distance between $Y$ and $y\exact$.
By the definition of the Ky Fan distance, $\PP(Y \in \Y_{\rm loc}) \geq 1 - \rho_{\rm K}(Y, y\exact)$.  

{\bf Smoothness in $x$}.

There exist $\delta >0$ and positive definite matrices $C_{f} \in \mathbb{R}^{n\times n}$, $ C_g \in \mathbb{R}^{p\times p}$ that may depend on $\delta$ such that for all $x\in B(x^\star, \delta)$ for all $y\in \Y_{\rm loc}$  and all $1 \leqslant i,j \leqslant p$,
\begin{eqnarray}\label{eq:BoundedDerG}
|v^T (\nabla^2 \tilde{f}_y(Ax) - \nabla^2 \tilde{f}_y(A x^\star))v | &\leqslant& \delta ||A|| v^T C_{f} v \quad \forall v\in \mathbb{R}^n, \\
| v^T (\nabla^2 g(x) - \nabla^2 g(x^\star)) v| &\leqslant& v^T C_{g} v\delta \, \quad \forall v\in \mathbb{R}^p,
\end{eqnarray}
where $\Y_{\rm loc}$ is defined by \eqref{def:Yloc}. 

This assumption holds if the third derivatives of these functions are uniformly bounded for $x\in B(x^\star, \delta)$ and $y\in \Y_{\rm loc}$, i.e. if $|\nabla_{ijk} \tilde{f}_y(Ax)|\leq C_{3,f}$ and $|\nabla_{ijk} g(x)|\leq C_{3,g}$ for all $i,j,k$. However, this assumption is not sufficient to obtain optimal rates of convergence in the case of growing dimensions $n$ and $p$ with eigenvalues of $A^T A$ decreasing to 0. See Section~\ref{sec:ExpFamily} for verification of this assumption for independent $Y_1,\ldots, Y_n$.
 


{\bf Convergence in $Y$.}


There exist positive definite matrices $M_{f,\,1}, M_{f,\,2} \in \mathbb{R}^{n\times n}$ such that for all $y\in \Y_{\rm loc}$,
\begin{eqnarray}\label{eq:BoundedDerDiffF2}
|| V (\nabla \tilde{f}_y(Ax^\star) - \nabla  \tilde{f}_{y\exact}(Ax^\star))||  &\leqslant&   || V M_{f,\, 1}(y - y\exact)||, \\
|v^T (\nabla^2 \tilde{f}_y(Ax^\star) - \nabla^2 \tilde{f}_{y\exact}(Ax^\star))v| &\leqslant& v^T M_{f,\, 2}^{1/2} \diag(|y_j - {y\exact}_j|) M_{f,\, 2}^{1/2}v\, \,\notag
\end{eqnarray}
for all $v\in \mathbb{R}^n$, $V \in \mathbb{R}^{p\times n}$. 
Here $\diag(|y_j - {y\exact}_j|) = \diag(|y_1 - {y\exact}_1|,\ldots, |y_n - {y\exact}_n|)$. 

Similarly to the assumption of smoothness in $x$, this assumption holds if 
$$|\nabla_{i} \tilde{f}_y(Ax^\star) - \nabla_{i} \tilde{f}_{y\exact}(Ax^\star)|\leq \tilde{M}_{f,1} ||y-y\exact||,$$ 
$$|\nabla_{ij} \tilde{f}_y(Ax^\star) - \nabla_{ij} \tilde{f}_{y\exact}(Ax^\star)|\leq \tilde{M}_{f,2} ||y-y\exact||$$ for $y \in \Y_{\rm loc}$ and all $i,j$ but again, this assumption is not sufficient to recover the optimal rates of convergence if $n,p$ increase and $A$ is ill-posed.  See Section~\ref{sec:ExpFamily} for verification of this assumption for exponential family.

Note that  $\nabla \tilde{f}_{y\exact}(Ax^\star)=0$ and  $\nabla^2 \tilde{f}_{y\exact}(Ax^\star) $ is positive definite since $Ax^\star$ minimises $\tilde{f}_{y\exact}$, if $x^\star$ is an interior point of $\X$.

The assumptions on smoothness in $x$ and convergence in $y$ are generalisations of the necessary conditions that are satisfied for the exponential family (see Section~\ref{sec:ExpFamily}).



{\bf Assumptions on $\delta$}.

Assume that $\delta >0$ satisfies the following conditions  as $\tau \to 0$:
\begin{enumerate}
\item
\begin{eqnarray}
\delta \to 0,  \quad \frac{\delta }{\sqrt{\tau}} \to \infty, \quad \delta \gg \rho_{\rm K}(Y, y\exact) +\nu, \notag \\
   \frac{ \delta\, [\rho_{\rm K}(Y, y\exact) +\nu]^{2}}{\tau} \to 0, \, \label{eq:CondDeltaLocal} \\
 \frac{  \delta}{\gamma} \to \infty \quad \text{(not necessary if $A^T A$ is of full rank)} \notag .
\end{eqnarray}

\item  For $ y \in \Y_{\rm loc}$,
\begin{eqnarray}\label{eq:LareDevIntegral}
\Delta_0(B(0,\delta))  \to 0 \quad \text{as} \quad \tau \to 0,
\end{eqnarray}
 where
\begin{eqnarray}\label{eq:DefD0}
\Delta_0(B(0,\delta)) &=& \frac{\int_{\X \setminus B(x^\star, \delta)} e^{-[h_y(x) - h_y(x^\star)]/\tau} dx }{\int_{ B(x^\star, \delta)} e^{-[h_y(x) - h_y(x^\star)]/\tau} dx}.
\end{eqnarray}

\end{enumerate}
After the approximation to $e^{-[h_y(x) - h_y(x^\star)]/\tau}$ on $B(x^\star, \delta)$ is derived, condition (\ref{eq:LareDevIntegral}) will be stated in a simplified form in Lemma~\ref{lem:LocalConst}. Throughout this section we use the error  $\Delta_0 = \Delta_0(B(0,\delta))$ defined by (\ref{eq:DefD0}).

\subsection{Rates of convergence} 

The limiting behaviour of the posterior distribution is characterised by
 the matrices of second derivatives:
\begin{eqnarray*}
V_y(x) &=&  \nabla^2 \tilde{f}_y(Ax),\quad \quad 
B(x) = \nabla^2 g(x),\\
H_y(x) &=& \nabla^2 h_y(x) = A^T V_y(x) A + \nu B(x).
\end{eqnarray*}

Define 
$\lambda_{\min, \,  P}(M) = \min_{||v||=1, \, Pv=v} ||M v||$ to be the smallest eigenvalue of a matrix $M$ on the range of a projection matrix $P$.

For a fixed $\omega$, we have the following upper bound on the Prokhorov distance between the posterior distribution and its limit.

\begin{theorem}\label{th:ProkhUpperInt} Suppose we have a Bayesian
model given in Section \ref{sec:SectionSetup}, and let the assumptions  stated
in Section~\ref{sec:QuadraticApprox} hold.
Assume also that $x^\star$ is an interior point of $\X$, and that matrix 
$$
\bar{H}_{Y(\omega)}(x^\star) = A^T V_{Y(\omega)}(x^\star) A +\nu
B(x^\star) - \delta [A^T C_{f} A \, ||A|| + \nu  C_{g}],
$$
 is of full rank.

Then, $\exists \tau_0  >0$ such that for $\forall \,
\tau \in (0, \tau_0]$,
\begin{multline}
\rho_{\rm P}(\mu_{\rm post}(\omega), \delta_{x^\star}) \leqslant \max\left\{  \frac{\Delta_0}{1+\Delta_0}, \quad
  ||[\bar{H}_{Y(\omega)}(x^\star)]^{-1} A^T  M_{f1} \, (Y(\omega)-y\exact)||\right.\\ +  \nu   ||[\bar{H}_{Y(\omega)}(x^\star)]^{-1}   \nabla g(x^\star) ||  \\
+ \left. \sqrt{- 4\tau \lambda(\omega) \log\left(  4\tau \lambda(\omega) \right) }(1+\Delta_{\star}(\delta, Y(\omega)))\right\},
\end{multline}
where $\lambda(\omega) = \trace\left([\bar{H}_{Y(\omega)}(x^\star)]^{-1}\right)$, $\Delta_0$ is defined by (\ref{eq:DefD0}) and  $\Delta_{\star}$ is defined by (\ref{eq:DefDeltaStar1}).
\end{theorem}
The first term in the sum represents the bias of the posterior distribution, and the second term is the Prokhorov distance between $\N(0, \tau H_{Y(\omega)}(x^\star)^{-1})$ and the  point mass at zero. The maximum reflects the fact that there are two ``competing'' tails: Gaussian on the ball  $B(x^\star, \delta)$ and the tail of the posterior distribution outside the ball.

Since $x^\star$ is an interior point of $\X$, under the assumptions of the theorem, then $y\exact$ is an interior point of $G^{-1}(A \X)$ and hence $V_{y\exact}(x^\star)$ is positive definite as the Hessian of the optimisation problem.

This theorem implies that to have convergence of
the posterior distribution to $\delta_{x^\star}$, we must have (a) convergence of the data so that
$||Y- y\exact|| \stackrel{\P_{x\true}}{\to} 0$, (b) $\nu = \tau/\gamma^2\to 0$, i.e. the prior distribution needs to be rescaled
in a way dependent on the scale of the likelihood, and (c) $\tau \trace([\bar{H}_{Y(\omega)}(x^\star)]^{-1})\to 0$. If the matrix $A^T V_{Y(\omega)}(x^\star) A$ is of full rank, then, for small $\tau$, the trace is close to the constant  $\trace[(A^T V_{y\exact}(x^\star) A)^{-1}]$ with high probability, hence the latter condition is satisfied as $\tau \to 0$. However, if $A^T V_{Y(\omega)}(x^\star) A$ is not of full rank, then, for small enough $\nu$ and $\tau$, $\trace([\bar{H}_{Y(\omega)}(x^\star)]^{-1}) \asymp \nu^{-1}$; hence, we must have $\tau/\nu =\gamma^2 \to 0$.

This is summarised in the following corollary.
\begin{corollary}
For weak convergence of the posterior distribution to the point mass at $x^\star$ as $\tau \to 0$ for a fixed $\omega$, we must have  $\nu = \tau/\gamma^2\to 0$.

1. If the matrix $A^T V_{Y(\omega)}(x^\star) A$ is not of full rank, then
we must also have $\gamma \to 0$.

2. If the matrix $A^T V_{Y(\omega)}(x^\star) A$ is of full rank, however, the scale of the prior distribution $\gamma$ may be taken a positive constant.

\end{corollary}

The theorem also implies that the rate of contraction of the posterior distribution (in terms of the Prokhorov distance)  varies between $\PA \X$ and $(I-\PA)\X$ and is determined by the second derivative of the logarithm of the posterior density.

This  theorem gives an upper bound on the Prokhorov distance
between the posterior distribution and the limit for any
particular instance of observed data $Y(\omega)$.
To ``lift'' the
result obtained to a bound on the Ky Fan distance over all $\omega$, we use
the following generalisation of the lifting theorem of \textcite{HofingerP:07} to the case of different bounds for different outcomes $\omega$.

\begin{theorem}\label{lem:KyFanDiffBounds}
Let random variables $X_1$, $X_2$ and $Y_1$, $Y_2$ be defined on
the same probability   space $(\Omega, {\cal F}, \P)$ with values
in metric spaces $(X, d_x)$ and $(Y, d_y)$, respectively, and suppose the sample space $\Omega$ is partitioned into
two parts, $\Omega = \Omega_1 \cup \Omega_2$, $\Omega_1 \cap
\Omega_2 = \emptyset$.

Assume that there exist positive nondecreasing functions $\Phi_1$
and $\Phi_2$:
$$
\forall \omega \in\Omega_k, \quad  d_x (X_1(\omega), X_2(\omega) ) \leqslant
    \Phi_k(d_y (Y_1(\omega), Y_2(\omega))), \quad k=1,2
$$
i.e. we have  different upper bounds on $\Omega_1 $ and $\Omega_2$.

Then, the following inequalities hold:
\begin{eqnarray*}
\rho_{\rm K} (X_1, X_2) &\leqslant& \max\{\rho_{\rm K} (Y_1 , Y_2 ) + \pr(\Omega_2),  \Phi_1(\rho_{\rm K} (Y_1 , Y_2 )) \},\\
\rho_{\rm K} (X_1, X_2) &\leqslant& \max\{\rho_{\rm K} (Y_1 , Y_2 ),  \Phi_1(\rho_{\rm K} (Y_1 , Y_2 )), \Phi_2(\rho_{\rm K} (Y_1 , Y_2 )) \}.
\end{eqnarray*}
\end{theorem}

In our case, $(X, d_x)$ is the space of all distributions equipped with the Prokhorov metric, and $(Y, d_y)$ is the metric space $\Y$ with the $\ell_2$ metric. Theorem~\ref{th:ProkhUpperInt}   provides an upper bound $\Phi_1$  on the event $\Omega_1$ where a random matrix $H_{Y(\omega)}(x^\star)$   is of full rank, and the first statement of the theorem is applied to obtain the Ky Fan rate of convergence.
Note that we do not need an upper bound $\Phi_2$  to bound the Ky Fan distance on $\Omega_2$, as long as $\pr(\Omega_2)$ is vanishingly small as $\tau \to 0$.



Denote
\begin{eqnarray}\label{eq:c1c2}
c_1 &=& ||H_{\nu}^{-1} A^T M_{f1}|| \quad \quad c_2 =  ||H_{\nu}^{-1} \nabla g(x^\star) ||,
\end{eqnarray}
 and, for small enough $\rho_{\rm K}(Y, y\exact)$ and $\delta$,
\begin{eqnarray}\label{eq:tildec1c2}
\bar{c}_k =  c_k \, \left[1 - \tilde\lambda \right]^{-1}, \quad k=1,2,
\end{eqnarray}
where $H_{\nu} = A^T V_{y\exact}(x^\star) A +\nu B(x^\star)$ and 
\begin{eqnarray}\label{eq:defLambda}
\tilde\lambda = \delta ||H_{\nu}^{-1} D || + \rho_{\rm K}(Y, y\exact) ||H_{\nu}^{-1} A^T M_{f2} A|| .
\end{eqnarray}

\begin{theorem}\label{th:KyFanPostInt}
Suppose we have the Bayesian model defined in Section
\ref{sec:SectionSetup}, and that the assumptions on $f_y$, $g$ and $\delta$ stated in Section~\ref{sec:QuadraticApprox} hold.

Assume that
\begin{enumerate}
\item  $x^\star$ is an interior point of $\X$,
\item $H_{\nu} = A^T V_{y\exact}(x^\star) A +\nu B(x^\star)$ is of full rank.
\end{enumerate}

Then, $\exists \tau_0 >0$ such that for $\forall \,
\tau \in (0, \tau_0]$, and small enough $\nu$ and $\tau/\nu$,
\begin{eqnarray}
\rho_{\rm K}(\mu_{\rm post}, \delta_{x^\star}) &\leqslant&
  \max\left\{2\rho_{\rm K}(Y, y\exact), \,\, \frac{\Delta_0}{1+ \Delta_0},\,\, \bar{c}_1 \rho_{\rm K}(Y, y\exact) + \bar{c}_2
 \nu \right. \\
&+& \left. \left[- 4\tau\trace(H_{\nu}^{-1}) \log\left( 4\tau \trace(H_{\nu}^{-1}) \right) \right]^{1/2} (1+\Delta_{\star, K}(\delta))\right\},  \notag
\end{eqnarray}
where $\bar{c}_1$ and $\bar{c}_2$ are defined by (\ref{eq:tildec1c2}), $\Delta_0 = \Delta_0(B(0,\delta))$ is given by (\ref{eq:D0expr}), and $\Delta_{\star, K}(\delta)$ is defined by (\ref{eq:DefDeltaStarK}).


Under the assumptions on $\tau$, $\nu$ and $\delta$, $\Delta_{\star, K}(\delta)=o(1)$  as $\tau \to 0$.
\end{theorem}
In the upper bound we have the maximum of three terms, the rate of convergence of the data, the tail of the posterior distribution and the convergence rate of the local Gaussian approximation of the posterior distribution around the concentration point $x^\star$. In the latter, there are three terms: the first two represent the bias, the first one is due to random error and the second one is the bias caused by the chosen prior, and the third one is the equivalent of the variance term which is characterised by the smallest radius $\lambda_{\rm min}(H_{\nu})$ of the concentration ellipse of the Gaussian approximation.
Recall that in the ill-posed case (if $A^T V_{y\exact}(x^\star) A$ is not of full rank), $\trace(H_{\nu}^{-1}) \asymp \nu^{-1}$, and in the well-posed case $\trace(H_{\nu}^{-1})\asymp const$.

For a well-posed problem, $x^\star= x\true$ and the theorem provides us the rate of convergence of the posterior distribution to the point mass at $x\true$. However, for an ill-posed problem, $x^\star$ may differ from $x\true$ in the null space of $A$, so there is an additional bias term $|| (I-\PA)(x\true - x^\star)||$.

Thus, we have the following corollary.

\begin{corollary}\label{eq:cor:KyFanRate} Suppose that $\rho_{\rm K}(Y, y\exact) \leqslant C\sqrt{-\tau \log \tau}$ for some constant $C$, and that the assumptions of Theorem~\ref{th:KyFanPostInt} are satisfied, and that $\frac{\Delta_0}{1+ \Delta_0}$ is smaller than the other terms in the maximum.

If $A^T V_{y\exact}(x^\star) A$ is of full rank (well-posed problem), the smallest upper bound on the Ky Fan distance  is
$$
\rho_{\rm K}(\mu_{\rm post}, \delta_{x\true}) \leqslant C_1 \left( -\tau \log \tau \right)^{1/2},
$$
with $\gamma^2 \geq \tau^{1/2}[-\log \tau]^{-1/4}$.

If $A^T V_{y\exact}(x^\star) A$ is not of full rank (ill-posed problem), the smallest upper bound on the Ky Fan distance  is
$$
\rho_{\rm K}(\mu_{\rm post}, \delta_{x\true}) \leqslant || (I-\PA)(x\true - x^\star)|| + C_2 \left( -\tau \log \tau \right)^{1/3},
$$
 with $\gamma^2 =\tau^{2/3}[-\log \tau]^{-1/6}$.
\end{corollary}
In particular, in the case of an ill-posed problem, the posterior distribution provides information only about $\PA x\true$. Example of an ill-conditioned problem, where the eigenvalues of $A$ are positive but decrease to zero with growing dimensions $n, p$, is considered in Section~\ref{sec:pgrows}.

The assumption of the corollary $\rho_{\rm K}(Y, y\exact) \leqslant C\sqrt{-\tau \log \tau}$ is satisfied for Gaussian random variables $Y$
as well as  for other distributions from the exponential family (Section~\ref{sec:metrics}).


In the following corollary we list the conditions necessary for the convergence of the posterior distribution.
\begin{corollary}\label{cor:CondConv}
Under the assumptions of Theorem~\ref{th:KyFanPostInt}, conditions necessary for the posterior distribution to converge to $\delta_{x^\star}$ as $\tau \to 0$ are:
\begin{gather}
\rho_{\rm K}(Y, y\exact)\to 0, \quad \nu \to 0,  \quad  \tau \trace(H_{\nu}^{-1})  \to 0, \notag \\
\delta \to 0, \quad \frac{\lambda_{\min}(H_{\nu})\delta^2}{\tau} \to \infty, \quad  \delta > 0.5(\bar{c}_1 \rho_{\rm K}(Y, y\exact) + \bar{c}_2 \nu),\\
\delta \rho_{\rm K}(Y, y\exact)/\sqrt{\tau} \to 0.
\end{gather}
\end{corollary}
It is often of interest to consider the case of growing dimensions $p$ and $n$. The upper bound depends on $p$ and $n$ via the $\ell^2$ vector norms  in $\mathbb{R}^p$ and $\mathbb{R}^n$ that underly the Ky Fan distances, the trace of $p\times p$ matrix $H_{\nu}^{-1}$, and the bias caused by the prior $||H_{\nu}^{-1} \nabla g(x^\star)||$ that is the $\ell^2$ norm of a $p$-dimensional vector (in fact, of $\rank(A^T A)$ -dimensional vector, since $x^\star$ is interior point and hence $(I-\PA) \nabla g(x^\star)=0$). Depending on the interpretation of the finite dimensional model, for instance, a discretisation on a grid or a spectral cutoff framework, the effect of growing dimensions can be different. This problem is discussed in Section~\ref{sec:pgrows}.

\subsection{Choice of $\delta$}
\label{sec:largedev}

Now, we discuss how to choose $\delta$ in such a way that
\begin{eqnarray*}
\int_{ \X} e^{- (h_y(x) - h_y(x^\star))/\tau} dx = [1+o(1)]\int_{B(x^\star, \delta)} e^{- (h_y(x) - h_y(x^\star))/\tau} dx
 \end{eqnarray*}
 with high probability as $\tau \to 0$, i.e. that the condition (\ref{eq:LareDevIntegral}) $ \Delta_0(B(0,\delta))  \to 0 \quad \text{as} \quad \tau \to 0$ is satisfied with high probability.

We introduce the following additional notation. Diagonalise  the projection matrices $\PA$ and $I - \PA$ simultaneously, so that $\PA = U^T \diag(I_{p_0}, 0_{p_1}) U $, $I-\PA = U^T \diag(  0_{p_0}, I_{p_1}) U $ and $U^T U = I_p$, where $p_0 = \rank(A)$ and $p_1 = p-p_0$.
\begin{eqnarray*}
\Omega_{00} &=& \UU_0^T \nabla^2 f_{y\exact} (x^\star) \UU_0,\\
B_{11} &=&  \UU_1^T \nabla^2 g(x^\star) \UU_1,\\
x_0 &=& H^{-1} \nabla h_{y\exact}(x^\star).
\end{eqnarray*}

First we consider the integral of $e^{-h_y(x)/\tau}$ over $B(x^\star, \delta)$.
\begin{lemma}\label{lem:LocalConst} Assume that $\Omega_{00}$ and $B_{11}$ are of full rank.
Under the assumptions on $f_y$, $g$ and assumption (i) on $\delta$ stated in Section~\ref{sec:QuadraticApprox},
\begin{eqnarray*}
\int_{B(x^\star, \delta)} e^{-[h_y(x)- h_y(x^\star)]/\tau} dx = \tau^{p_0/2} \gamma^{p_1} \frac{ (2\pi)^{p/2} e^{ x_0^T H x_0 /(2\tau)} }{ [\det(\Omega_{00}) \det(B_{11})]^{1/2} } [1+o_P(1)].
 \end{eqnarray*}

In particular, this implies that
\begin{eqnarray}\label{eq:D0expr}
\Delta_0(B(0,\delta)) &=& C_H \tau^{-p_0/2} \gamma^{-p_1} \int_{\X \setminus B(x^\star, \delta)} e^{-[h_y(x) - h_y(x^\star)]/\tau} dx \, [1+o_P(1)],
\end{eqnarray}
where $C_H = (2\pi)^{-p/2} [\det(\Omega_{00}) \det(B_{11})]^{1/2}\, e^{ - x_0^T  H  x_0 /(2\tau)}$.
 \end{lemma}
See Proposition \ref{prop:LaplaceApproxAsympt} in the Appendix for further details and the proof.

%% file: Example_ExpFamily.tex
\subsection{Convergence for exponential family}\label{sec:ExpFamily}

In this section we apply Theorem~\ref{th:KyFanPostInt} to the case of the exponential family \eqref{eq:ExpFamilyDef} with
$$
\tilde{f}_y(\ETA) =  \sum_{i=1}^n  [ y_i b(\ETA_i)-c(\ETA_i)]
$$
and identity link function $G$, i.e.  $ \ETA = Ax$. In particular, we have that $y\exact =   A x\true$. 

For all $\ETA$ that is an interior point of $ A \X$, we have
\begin{eqnarray*}
\nabla \tilde{f}_y(\ETA) = ((y_1-\ETA_1) b'(\ETA_1),\ldots, (y_n-\ETA_n) b'(\ETA_n))^T \\
\nabla^2 \tilde{f}_y(\ETA) = \diag((y_i-\ETA_i) b''(\ETA_i) -b'(\ETA_i))
\end{eqnarray*}
due to identity $c'(\ETA_i) = b'(\ETA_i) \, \ETA_i $ which implies $c''(\ETA_i) = b''(\ETA_i) \, \ETA_i + b\,'(\ETA_i)$ and $c'''(\ETA_i) = b'''(\ETA_i) \, \ETA_i + 2b\,''(\ETA_i)$. In particular, we have that for the true value of the parameter $\ETA = y\exact$, $\E Y_i = {y\exact}_i$ and $\Var(Y_i) = -\tau \, [b \,'({y\exact}_i)]^{-1}$.

Also, since $y\exact$ is an interior point of $ A \X$ and it minimises $\tilde{f}_{y\exact}(\ETA)$, it implies that $\nabla^2 \tilde{f}_{y\exact}(y\exact)$ is positive definite, i.e. $b'({y\exact}_i) <0$ for all $i$.

Now we verify the likelihood-specific assumptions.
Assumptions stated in Section~\ref{sec:Setup} have already been verified for the exponential family.

\begin{lemma}\label{lem:XYcondExp} Assume that the distribution of $Y$ belongs to the exponential family \eqref{eq:ExpFamilyDef} 
such that function $\sum_i[{y\exact}_i b(\eta_i)- c(\eta_i)]$ is three times differentiable, and its  third derivatives are bounded on $ A B(x^\star, \delta)$ for small enough $\delta$.

Then, assumptions  \eqref{eq:BoundedDerG} and \eqref{eq:BoundedDerDiffF2} on smoothness in $x$ and convergence in $Y$  are satisfied if with
\begin{eqnarray*}
M_{f,\, 1} = \diag(|b'({y\exact}_i)|),\quad 
M_{f,\, 2} = \diag(|b''({y\exact}_i)|), 
\end{eqnarray*}
and $C_{f} = \diag\left(C_{f,1}, \ldots, C_{f,n}\right)$ where 
$$
C_{f,i} = \max_{x \in B(x^\star, \delta)} |    b'''([  Ax]_i) \max_{y\in \Y_{\rm loc}} y_i - c'''([ Ax]_i)|. 
$$
 
\end{lemma}

Then, for a distribution from the exponential family,
\begin{eqnarray*}
V_{y\exact}(x^\star) &=&   \diag(- b'({y\exact}_i)) ,
\end{eqnarray*}
and hence $H_{\nu} =  A^T\diag(- b'({y\exact}_i))   A + \nu \nabla^2 g(x^\star)$. 
Note that $V_{y\exact}(x^\star) =   M_{f1}$.

Now we state the constants for Gaussian, rescaled Poisson and Gamma distributions.
\begin{enumerate}
\item Gaussian distribution: $Y \sim \N(\ETA, \tau \Sigma)$ with $\Sigma = \diag(\sigma_i^2)$. In this case, $b(\ETA_i)= -\ETA_i/\sigma_i^2$, $c(\ETA_i)=0.5\ETA_i^2/\sigma_i^2$, and hence
$$
M_{f,\, 1} =V_{y\exact}(x^\star)= \Sigma^{-1}, \quad  M_{f 2} =  C_{f} = 0.
$$
Note that all constants are independent of $y\exact$.

\item For the rescaled Poisson distribution: $Y_i/\tau \sim Pois(\ETA_i/\tau)$, we have $b(\ETA_i) = -\log (\ETA_i)$, $c(\ETA_i) = \ETA_i$.
Assume that ${y\exact}_i > 0$ for all $i$. This implies that, for small enough $\delta$,
\begin{eqnarray*}
M_{f,\, 1} =  V_{y\exact}(x^\star)= \diag({y\exact}_i^{-1}), \quad  M_{f 2} = \diag({y\exact}_i^{-2}), \\ 
C_{f} = 2\, \diag\left( ({y\exact}_i + \rho_{\rm K}(Y, y\exact)) [{y\exact}_i-\delta || A||]^{-3}\right).
\end{eqnarray*}

\item For the Gamma distribution with the shape parameter $a/\tau$, we have $b(\ETA_i) = a/\ETA_i$ and $c(\ETA_i) = -a \log (\ETA_i)$. This implies
\begin{eqnarray*}
M_{f,\, 1} =  V_{y\exact}(x^\star)= a\,\diag({y\exact}_i^{-2}), \quad  M_{f 2} = 2a\, \diag({y\exact}_i^{-3}), \\
C_{f} = 2a \,\diag\left(  (4{y\exact}_i+ \rho_{\rm K}(Y, y\exact))\,[{y\exact}_i-\delta ||  A||]^{-4} \right).
\end{eqnarray*}

\end{enumerate}
Note that $C_{f} \leq C M_{f 2}$ where constant $C$ is close to an absolute constant   as $\delta \to 0$; this constant is  1 for Gaussian, 2 for rescaled Poisson and 4 for Gamma.

Therefore, the likelihood-related assumptions of Theorem~\ref{th:KyFanPostInt} are satisfied, and the results of Corollary~\ref{eq:cor:KyFanRate} apply to the exponential family, namely for a well-posed case we can take any $\gamma\geq (\tau /\log(1/\tau))^{1/4}$ and the rate is $\left( - \tau\log  \tau  \right)^{1/2}$, whereas for an ill-posed case, the best possible rate of convergence is $[\tau  \log (1/\tau)]^{1/3}$ with $\gamma \asymp \tau^{1/3}[\log(1/ \tau)]^{-1/6}$.

%% file: NGrows_IP.tex
\section{Consistency in inverse problems with growing dimension}
\label{sec:pgrows}

We consider two cases where dimensions $n$ and $p$ grow. In the first case, the considered finite dimensional inverse problem \eqref{eq:exactInDim} is an equally spaced discretisation of an infinite dimensional inverse problem $\g = \G  (\A \f)$. In the second case the finite dimensional inverse problem \eqref{eq:exactInDim} is formulated for the eigenvalues of operator $\A^* \A$ and the coefficients of functions $f$ and $g$ in the corresponding separable  Hilbert space. 

In this section, we denote by $C$ a generic constant $C \in (0,\infty)$ that can be different even within a single equation.

\subsection{Ill-posed inverse problem, discretisation on a grid}

We assume that functions $\f$ and $\g$ are defined on a finite interval, e.g. $[0,1]$, without loos of generality, and $x$, $y$ and $A$ are discretisations of the functions and the operator on a regular grid as discussed in Section~\ref{sec:intro1}.  The appropriate distances in this case are $\frac 1 {\sqrt{p}} ||x-x^\star||$ and $\frac 1 {\sqrt{n}} ||y-y\exact||$ respectively that are approximations of the distances between functions in $L^2([0,1])$ discretised at $p$ (or $n$) equally spaced points. The distance between functions is the $L^2$ norm of their difference. Denote the Ky Fan metric based on these rescaled distances by $\tilde\rho_{\rm K}$.

Then, Theorem~\ref{th:KyFanPostInt} together with Corollary~\ref{cor:CondConv} can be reformulated as follows.

\begin{theorem}\label{th:KyFanPostInt_NPgrow}
Suppose we have the Bayesian model defined in Section
\ref{sec:Setup}, and that the assumptions on $f_y$, $g$ and $\delta= \tilde\delta \sqrt{p}$ stated in Section~\ref{sec:QuadraticApprox} hold.

Assume that
\begin{enumerate}
\item  $x^\star$ is an interior point of $\X$,
\item $H_{\nu} = A^T V_{y\exact}(x^\star) A +\nu B(x^\star)$ is of full rank,
\end{enumerate}
and that
$$\tilde\rho_{\rm K}(Y, y\exact)\sqrt{n /p}\to 0, \quad    \frac{p  \tilde\delta^2}{\tau \trace(H_{\nu}^{-1}) } \to \infty, \quad
 \frac{\tau \trace(H_{\nu}^{-1})}{p} \to 0, \quad  \frac{ \tilde\delta^2 [\tilde\rho_{\rm K}(Y, y\exact)]^2}{\tau} \to 0.
$$

Then, $\exists \tau_0 >0$ such that for $\forall \,
\tau \in (0, \tau_0]$, and small enough $\nu$ and $\tau/\nu$,
\begin{eqnarray}
\tilde\rho_{\rm K}(\mu_{\rm post}, \delta_{x^\star}) &\leqslant&
  \max\left\{2 \tilde\rho_{\rm K}(Y, y\exact) \sqrt{\frac n p}, \,\, \frac{\Delta_0}{\sqrt{p}(1+ \Delta_0)}, \right.   \\
  && \frac{  ||H_{\nu}^{-1} A^T M_{f1}||}{1 -  \tilde\lambda} \, \tilde\rho_{\rm K}(Y, y\exact) \sqrt{\frac n p}+ \frac{\nu}{ \sqrt{p}} \, \frac{||H_{\nu}^{-1}\nabla g(x^\star) ||}{1 -  \tilde\lambda} \notag\\
&+& \left. \sqrt{-  \frac{4\tau \trace(H_{\nu}^{-1})}{p} \log\left(   4\tau \trace(H_{\nu}^{-1})  \right)  } (1+\Delta_{\star, K}(\delta))\right\},  \notag
\end{eqnarray}
where $\tilde\lambda$ is defined by \eqref{eq:defLambda}, $\Delta_0 = \Delta_0(B(0,\delta))$ is given by (\ref{eq:D0expr}), and $\Delta_{\star, K}(\delta)$ is defined by (\ref{eq:DefDeltaStarK}).


Under the assumptions on $\tau$, $\nu$ and $\delta$, $\Delta_{\star, K}(\delta)=o(1)$  as $\tau \to 0$.
\end{theorem}




For a one-dimensional model with homoscedastic Gaussian errors, the discretisation effect for a one-dimensional function on $\tau$ is $\tau = \sigma^2/n$. We will use this dependence of $\tau$ on $n$ below.

The necessary conditions for the posterior distribution to converge  with $||x-x^\star||\leq \delta = \sqrt{p}\tilde\delta$ in the ill-posed case are
\begin{gather*}
   \tilde\rho_{\rm K}(Y, y\exact) =o(\sqrt{p/n}),  \quad \gamma \gg \sqrt{ p/n}, \quad 
   \tilde\delta = o\left( \sqrt{p/n} \right). 
\end{gather*}
These conditions together with the assumption $\gamma \to 0$ imply that we must have $p/n \to 0$.

If $\rho_{\rm K}(Y, y\exact) \asymp  \sqrt{-n\tau \log(n\tau)} $,  then, as $n\to \infty$,
$$
n^{-1/2}\rho_{\rm K}(Y, y\exact) \asymp n^{-1/2} \sqrt{-n\tau \log(n\tau)} = n^{-1/2}\sigma\sqrt{ -2\log(\sigma)} \to 0.
$$

Denote $r= \rank(A)$ and assume that $||H_{\nu}^{-1}\PA \nabla g(x^\star)|| \asymp \sqrt{\rank(A)}$ since vector $\PA \nabla g(x^\star)$ is in $r$-dimensional subspace of $\mathbb{R}^p$ and $(I-\PA) \nabla g(x^\star)=0$ and $x^\star$ minimises $g(x)$ on $Ax= y\exact$.  
Then, we have $\trace(H_{\nu}^{-1}) \asymp r + \nu^{-1} (p-r)$ which implies
$$
\tau \trace(H_{\nu}^{-1})/p  = (r + n\gamma^2 (p-r))/(np) = r/(np) + \gamma^2 (1-r/p).
$$
If $r=p$ then this expression is of order $1/n$ otherwise the leading order is $\gamma^2$.

Therefore, for a well-posed problem with $r=p$ (assuming $\frac{\Delta_0}{1+ \Delta_0} \leq C$), the rate is
\begin{eqnarray}
\tilde\rho_{\rm K}(\mu_{\rm post}, \delta_{x^\star}) &\leqslant&
  C \left[ p^{-1/2} + [n \gamma^2]^{-1}+  \sqrt{\frac{\log n}{n} }\right],  \notag
\end{eqnarray}
i.e. we can choose $\gamma$ to be a constant and hence we obtain the parametric rate of convergence.

The rate for an ill-posed problem with $r<p$ (assuming $\frac{\Delta_0}{1+ \Delta_0} \leq C $):
\begin{eqnarray}
\tilde\rho_{\rm K}(\mu_{\rm post}, \delta_{x^\star}) &\leqslant&
  C \left[ p^{-1/2} + [n \gamma^2]^{-1}+ \gamma\left[- \log\left( \gamma^2  \right)  \right]^{1/2} \right],  \notag
\end{eqnarray}
The value of $\gamma^2$ that minimises this expression is $\gamma^2 \asymp [ n (\log n)]^{-1/3}$. Condition $\tau/\gamma^2 = \sigma^2/(n\gamma^2) \asymp n^{-2/3} (\log n)^{1/3} \to 0$ is satisfied.
This gives the rate of convergence
\begin{eqnarray}
\tilde\rho_{\rm K}(\mu_{\rm post}, \delta_{x^\star}) &\leqslant& C \left[ p^{-1/2} +  \left(\frac{\log n}{n}\right)^{1/3}\right].
\end{eqnarray}
The smallest  $p$ providing the fastest rate of convergence $(n^{-1} \log n)^{1/3}$ is $p \asymp \left(\frac{n}{\log n}\right)^{2/3}$.

There is also discretisation bias $||\f - \f\true|| - \frac 1 {\sqrt{p}}||x-x\true||$ that depends on smoothness of $\f\true$ \cite{MPdiscr}.


\subsection{Spectral cutoff estimator.}

Assume that we have a linear inverse problem $\g = \A \f$ with operator $\A: \, \HH_{\f} \to \HH_{\g}$ where $\HH_{\f}$ and $\HH_{\g}$ are separable Hilbert spaces  with orthonormal bases $\{\phi_j\}_{j=1}^{\infty}$ and $\{\psi_j\}_{j=1}^{\infty}$, respectively, and $\f\in \HH_{\f}$, $\g \in \HH_{\g}$. Assume that operator $\A$ is compact, then we can use the eigenfunctions of self-adjoint operators $\A^T \A$ and $\A \A^T$ as the bases of $\HH_{\f}$ and $\HH_{\g}$, respectively. Denote the singular values of $\A$ by $a_j $, $j=1,2,\ldots$, and the coefficients of $\f$ and $\g$ in the corresponding Hilbert spaces as $\fk_j$ and $\gk_j$, respectively ($j=1,2,\ldots$). Then, the original inverse problem can be written as
$$
\gk_{j} = a_j  \fk_{j},\quad j=1,2,\ldots.
$$

In practice, instead of observing $(\gk_j)_{j=1}^{\infty}$, we have a finite number of their  noisy observations $y_1,\ldots,y_p$ that we assume to be independent random variables with a distribution from the exponential family \eqref{eq:ExpFamilyDef}:
$$
\tilde{f}_y(\ETA) = \sum_{i=1}^p  [ y_i b(\ETA_i)-c(\ETA_i)],
$$
with the noise level $\tau \to 0$ and the true value of vector $\ETA \in \mathbb{R}^p$ is $\gk = (\gk_1,\ldots, \gk_{p})$. This is an example of the finite-dimensional inverse problem considered in Sections~\ref{sec:SectionSetup} and \ref{sec:ConvergenceKyFan} with $n=p$, identity link function $G$  and diagonal matrix $A = \diag(a_1,\ldots, a_p)$. Here we assume that matrix $A$ is of full rank but the eigenvalues decrease to zero as $p\to \infty$ so that the finite-dimensional problem is ill-conditioned. Thus, here $x^\star = x\true$. To construct an estimators $\f$, we use estimators of $\fk$ as $p\to \infty$. Our motivating examples are Poisson and Gaussian distributions, however, we derive the conditions necessary for consistency for any distribution from the exponential family \eqref{eq:ExpFamilyDef} satisfying stated assumptions.

Given estimates  $\hat{x}_j$, the corresponding estimator of the unknown function $\f\true$ then can be constructed as
\begin{eqnarray}\label{eq:DefHatf}
\hat{\f}(u)  = \sum_{j=1}^p \hat{x}_j \phi_j(u).
\end{eqnarray}
The posterior distribution of $\f$ can be constructed in a similar way using the posterior distribution of $(x_1,\ldots, x_p)$.

Following  \textcite{vdVaart2011}, we consider an example of a mildly ill-posed inverse problem and a smooth unknown function $\f\true$, with some $\alpha, \beta >0$:
$$
|a_j | \asymp j^{-\alpha}, \quad  |\fk_{j}| \asymp j^{-\beta-1/2}, \quad j=1,\ldots, p.
$$
The latter condition on $\fk$ with $p\to \infty$ implies that $\f\true$ belongs to the Sobolev space relative to the basis $(\psi_j)$, $S^{\beta'}$, for all $\beta' \in (0,\beta)$, since in this case
$$
\sum_{j=1}^p j^{-2\beta -1} (1 + j^2)^{\beta'} \leq C \sum_{j=1}^p j^{-2(\beta-\beta') -1} \leq \tilde{C} p^{-2(\beta-\beta')} <\infty
$$
as $p\to \infty$ (see also \textcite{vdVaart2011}). These assumptions imply that $\gk_j \asymp j^{-\alpha -  \beta -1/2}$.

Consider first the prior distribution such that $\nabla^2 g(x^\star)$ is a diagonal matrix, i.e. $\nabla^2 g(x^\star) = \diag(b_j^2)$ with $b_j\geq 0$.
For example, if the prior is $\N(0, \gamma^2 B^{-1})$, then $\nabla^2 g(x^\star) = B = \diag(b_j^2)$, i.e. $b_j^2$ is the prior precision of $x$.
Assume that
$$
\nabla^2_{jj} \tilde{f}_{y\exact}(y\exact) \asymp j^{s},  \quad  b_j \asymp j^{\kappa +1/2},\quad j=1,\ldots,
n.
$$
Parameter $s$ depends on the error distribution, in particular, how the information matrix $\nabla^2\tilde{f}_{y\exact}(y\exact)$ depends on the parameter $y\exact$. As demonstrated in Section~\ref{sec:ExpFamily}, $s = 0$ for the Gaussian  errors, for the scaled Poisson distribution $s = \alpha+ \beta+1/2$ and for the Gamma model $s=2(\alpha+\beta +1/2)$. 
Also, for the exponential family, $\Var(Y_j) \asymp \tau j^{-s}$, since the variance is the inverse of the Fisher information.

Following the argument of \textcite{vdVaart2011}, the  assumption on $b_j$ under the Gaussian prior distribution corresponds to the a priori assumption that the unknown function $\f\true$ belongs to the Sobolev spaces (relative to the chosen basis) $S^{\kappa'}$ with $\kappa' \in (0,\kappa)$, since in this case
$$
\E \left[\sum_{j=1}^p (1+j^{2})^{\kappa'} x_j^2 \right]=  \sum_{j=1}^p (1+j^{2})^{\kappa'} b_j^{-2} <\infty \quad \text{as} \quad p\to \infty
$$
where the expectation is taken with respect to the prior distribution.

The  constants $M_{f,k}$ and $C_{y, f}$ that were derived in Section~\ref{sec:ExpFamily} for the exponential family, satisfy
$$
M_{f,1} \asymp \diag(j^{s}) ,\quad M_{f,2} \asymp C_{y, f} \asymp \diag(j^{s_2})
$$
for $y\in \Y_{\rm loc}$. 
Then, $H_{\nu} \asymp \diag(i^{-2\alpha+s} + \nu i^{2\kappa +1})$.

Assume that $|\nabla_j g(x\true)| \leq C j^{2\kappa+1 - \beta-1/2}$ which is true for the Gaussian prior with zero mean and precision matrix $\diag(b_j^2)$. For non-Gaussian prior distributions, the dependence of $|\nabla_j g(x\true)|$ on $j$ could be different.

The appropriate metric in $\X\subseteq \mathbb{R}^p$ in this case is the $p$-dimensional $\ell_2$ vector norm  $||x-x\true||$ that converge to the  $L^2$ function norm  $||\hat{\f}- \f\true||_2$ as $p \to \infty$. The discrepancy between the finite-dimensional and the infinite dimensional norms for $||x-x\true||$ is of order $[\max(p, \nu^{-1/m})]^{-\beta} =  \min(p^{-\beta}, \nu^{\beta/m})$, i.e.
$$
||\f\true- \hat{\f}||_2 \leq 2\left[\sum_{i=1}^{p} (x_i - [x\true]_i)^2\right]^{1/2} + C \min(p^{-\beta}, \nu^{\beta/m}).
$$
Let $\mu_{\rm post}^*(\omega)$ be the posterior distribution of the unknown function $\f\true$ that is reconstructed from the posterior distribution of its first $p$ coefficients $x_1,\ldots, x_p$ using construction \eqref{eq:DefHatf}, and consider the Ky Fan distance in the metric space $L^2$ with metric $d(\f, \g) = ||\f-\g||_2$. Then, using the triangle inequality,
$$
\rho_{\rm K}(\mu_{\rm post}^*, \delta_{\f\true}) \leq \rho_{\rm K}(\mu_{\rm post}, \delta_{x\true})+ C \min(p^{-\beta}, \nu^{\beta/m}).
$$

The rate of convergence of the posterior distribution under the considered model is summarised in the lemma.
\begin{lemma}\label{lem:SpectralRate}
Assume that $||M_{f,1}^{-1} M_{f,2} M_{f,1}^{-1}||\leq C $ (i.e. that $s_2\leq 2s$),  $m = 2\alpha-s+2\kappa +1 >0$, $\kappa> 0$,   
 and  $ \Delta_0 \leq C \left(\tau\log (1/\tau)\right)^{1/2}$.

 Then, for a Gaussian prior $x\sim \N_p(0, \gamma^2 B^{-1})$, $\nu = \tau/\gamma^2$, and for small enough $\tau$,
\begin{eqnarray}\label{eq:BoundRhoKFExp}
\rho_{\rm K}(\mu_{\rm post}^*, \delta_{\f\true}) &\leq&  C [\max(p, \nu^{-1/m})]^{-\beta} \\
&+& C  \nu \min(\nu^{-1/m},p)^{(m -\beta)_+} [ \log p]^{ I(\beta=m)/2}\notag\\
&+& C \tau^{1/2} \min(\nu^{-1/m},p)^{( \alpha-s/2 +1/2)_+}[ \log (p/\tau) ]^{(1+I(s=2\alpha+1))/2}.\notag
\end{eqnarray}

\end{lemma}
In the assumption $ \Delta_0 \leq C \left(\tau\log (1/\tau)\right)^{1/2}$, the upper   bound can be replaced with the upper bound on $\rho_{\rm K}(\mu_{\rm post}^*, \delta_{\f\true})$ stated in the lemma.

Using a different prior distribution may result in a different upper bound on the nonrandom bias (second term).

Taking $\nu^{-1/m} \leq p$, i.e. $\nu\geq p^{-m}$, the rate becomes
\begin{eqnarray*}
\rho_{\rm K}(\mu_{\rm post}^* , \delta_{\f\true}) &\leq&  C p^{-\beta}+ C  \nu^{\min(1,\beta /m)} [ \log p]^{ I(\beta=m)/2}\\
&+& C \tau^{1/2}  \nu^{-( \alpha-s/2 +1/2)_+/m}[ \log p +\log (1/\tau) ]^{(1+I(s=2\alpha+1))/2}.
\end{eqnarray*}


Now we consider  different cases.
\begin{enumerate}

\item If $s>2\alpha+1$,  the rate of convergence is $\min[p^{-\min(\beta,m)}, \tau^{1/2}]$, up to the log factor. In this case, due to large $s$, the problem becomes well-posed. We would need to use at least $p\geq \tau^{-1/[2\min(\beta,m)]}$ eigenvalues to achieve it. The prior distribution can be chosen to be  non-informative, with the prior density proportional to a constant, by setting $\nu =0$.

\item If $ \alpha-s/2 +1/2 >0$, the upper bound is minimised at $\nu = \left( \tau [ \log (1/\tau)]^{1+ \varkappa}  \right)^{m/[2\min(\beta , m)+2\alpha+ 1-s]} $, and the smallest upper bound is
\begin{eqnarray*}
\rho_{\rm K}(\mu_{\rm post}^*, \delta_{\f\true}) &\leqslant& C\left( \tau \,  \log (1/\tau)]^{1+\varkappa}  \right)^{\frac{\min(\beta, m)}{2\min(\beta , m)+ 2\alpha+ 1-s }} [ \log (1/\tau)]^{0.5 I(\beta = m)},
\end{eqnarray*}
where $\varkappa = I(\alpha =(1-s)/2)-  I(\beta = m)$. Condition $\nu \geq p^{-m}$ is satisfied if $p   \geq \tau^{-1/[2\min(\beta , m)+2\alpha+ 1-s]}$.

\begin{enumerate}

\item If $ \kappa > \beta/2 -\alpha+s/2 -1/2$ (i.e. $\beta <m$), then the smallest upper bound is
 $$
\rho_{\rm K}(\mu_{\rm post}^*, \delta_{\f\true}) \leqslant  C\left( \tau [\log (1/\tau)]^{1+\varkappa} \right)^{\frac{\beta}{2\beta+(2\alpha+ 1-s)_+}}.
 $$

If $s>0$, then the rate of convergence can be faster than the minimax rate for the Gaussian errors. Therefore, for the exponential families, there is an effect of self-regularisation  if $2\alpha +1\leq s$ when the rate becomes parametric.


For the rescaled Poisson distribution, $s=\alpha+\beta+1/2$, and the rate of convergence (up to a log factor) is
 $$
\rho_{\rm K}(\mu_{\rm post}^*, \delta_{\f\true}) \leqslant  C\left( \tau [\log (1/\tau)]^{1+\varkappa}  \right)^{\frac{\beta}{2\beta+(\alpha+ 1/2-\beta)_+}},
 $$
which is $C\left( \tau \log (1/\tau) \right)^{\frac{\beta}{\beta+\alpha+ 1/2 }}$ if $\alpha+ 1/2>\beta$, and it is
$C\left( \tau \log (1/\tau)  \right)^{1/2}$ if $\alpha+ 1/2 < \beta$.

\item If  $s=0$ (Gaussian case), then condition $ \alpha-s/2 +1/2 >0$ is satisfied, where the upper bound is minimised at $\nu = \left( \tau [\log (1/\tau)]^{1+ \varkappa}  \right)^{m/[2\min(\beta , m)+2\alpha+ 1 ]}  $, and the bound is
\begin{eqnarray*}
\rho_{\rm K}(\mu_{\rm post}^*, \delta_{\f\true}) &\leqslant& C\left( \tau [\log (1/\tau)]^{1+\varkappa}  \right)^{\frac{\min(\beta, m)}{2\min(\beta , m)+ 2\alpha+ 1 }} [\log (1/\tau)]^{0.5 I(\beta = m)},
\end{eqnarray*}
where $m = 2\alpha + 2\kappa+1$. Thus, we recover the rate of \textcite{vdVaart2011}. The minimax rate is attained if $2\alpha + 2\kappa+1 >\beta$, otherwise the rate is slower.

\item If 
$\nu \asymp \tau $, i.e. the scale of the prior $\gamma$ is constant, condition $\nu \geq p^{-m}$ is satisfied if  $p\geq \tau^{-1/m}$, and
\begin{eqnarray*}
\rho_{\rm K}(\mu_{\rm post}^*, \delta_{\f\true}) &\leq&  C p^{-\beta}+ C  \tau^{\min(1,\beta /m)} [ \log p]^{ I(\beta=m)/2}\\
&+& C \tau^{ -( \alpha-s/2 +1/2 -m/2)/m}[\log p +\log (1/\tau) ]^{1/2}\\
&\leq& C  \tau^{\min(\beta, \kappa)/(2\alpha -s  + 2\kappa+1)}[\log (1/\tau) ]^{1/2}
\end{eqnarray*}
which is the smallest for $\kappa = \beta$, with the rate
\begin{eqnarray*}
\rho_{\rm K}(\mu_{\rm post}^*, \delta_{\f\true}) &\leq& C  \tau^{ \beta/(2\alpha -s  + 2\beta+1)}[\log (1/\tau) ]^{1/2}
\end{eqnarray*}
which is minimax (up to log factor) for the Gaussian distribution, i.e. if $s=0$.

\item $\beta > m$, i.e $2\alpha+1 > s >2\alpha+2\kappa+1-\beta$ (possible if $2\kappa <\beta$):
\begin{eqnarray*}
\rho_{\rm K}(\mu_{\rm post}^*, \delta_{\f\true}) &\leqslant& C\left( \tau [\log (1/\tau)]^{1+\varkappa}  \right)^{\frac{m}{2m+ 2\alpha+ 1-s}}\\
&\leqslant& C\left( \tau [\log (1/\tau)]^{1+\varkappa}  \right)^{\frac{2\alpha-s+1+2\kappa}{2(2\alpha-s+1) + 2\kappa}}.
\end{eqnarray*}

\end{enumerate}



\end{enumerate}

\begin{example} Consider the case where $y_i$ have the rescaled Poisson  distribution with identity link function: $Y_i/\tau \sim Pois([Ax]_i/\tau)$, $i=1,\ldots,n$. In this case,  $s=\alpha +\beta+1/2 $. Taking $\kappa \geq \beta$, the rate of convergence is
$$
\rho_{\rm K}(\mu_{\rm post}^*, \delta_{\f\true}) \leqslant  C  \tau^{1/2} [\log (1/\tau)]^{1/2+I(\beta = 1/2+\alpha)/2} ,
$$
for $\beta \geq 1/2+\alpha $, i.e. the rate of convergence is parametric; for $\beta < 1/2+\alpha$,
$$
\rho_{\rm K}(\mu_{\rm post}^*, \delta_{\f\true}) \leqslant  C\left( \tau \log (1/\tau) \right)^{\frac{\beta}{\beta+ \alpha +1/2}},
$$
with $\nu = \left( \tau \log (1/\tau)   \right)^{1 +2(\kappa -\beta)/[\beta+\alpha+1/2]}$. These rates are faster than the corresponding minimax rate of convergence under Gaussian errors, $\tau^{\frac{\beta}{2(\beta+\alpha+1/2)}}$ (up to a log factor).

\end{example}

 Theorem~\ref{th:KyFanPostInt} can also be applied for functions $\f$, $\g$ defined on $\mathbb{R}^d$; for $d=2$, the framework considered in \textcite{JSpet} can be applied.

In the considered framework of non-Gaussian errors, we can see two interesting phenomena. The first one is that the rate of convergence in the considered setting can be faster than the rate of convergence under the Gaussian noise. This is the implication for direct as well as for indirect (inverse) problems. This is due to dependence of the variance of the noise on the unknown function. The rate is similar to that of a mildly ill-posed operator with ill-posedness  $\alpha-s/2$ (provided $\alpha >s/2)$ if the variance of the noise increases as $j^s$. The second phenomenon is that it appears to be possible to recover the unknown function with the parametric rate of convergence, for $s$ large enough compared to $\alpha$. A simple non-Bayesian estimator that achieves parametric rate, with $\hat{x}_i = Y_i/a_i$, was discussed in the introduction.

%% file: PureBoundary.tex
\section{Convergence rate when $x^\star$ is on the boundary}\label{sec:boundary}

In this section we consider a special case where the assumption that $x^\star$ is an interior point of \(\X\) does not hold.  This is an example of so called nonregular models that have been considered mostly for a one-dimensional nonregular parameter (see, for instance, \textcite{GGS94} and \textcite{GS95}), and, as far as we are aware, have not been considered in the context of inverse problems. As we shall see, the rate of convergence is different in this case. We shall see that for some probability distributions, it makes it possible to observe exact data under the considered probabilistic model (Section~\ref{sec:SectionSetup}).

In this section we assume that the parameter space is $\X = [0, \infty)^p$, and that each coordinate of $x^\star$ is on the boundary of $\X = [0,\infty)^p$, i.e. $x^\star =0$. This is an important benchmark case where there is no signal. Such setup arises, for example, in image analysis, where $x$ is the vector of the unknown intensities, and we want to test whether there is any image present. We could assume that parameter $x$ is restricted to an arbitrary convex polyhedron; this could be reduced to $[0, \infty)^p$ by a linear change of variables.

\subsection{Assumptions}\label{sec:AssumeExt}

We make the same assumptions on the prior distribution as in Section~\ref{sec:QuadraticApprox}, however, we only need the smoothness and the convergence assumptions for up to the second derivative only, rather than up to the third. Assumptions on $\delta$ -- the radius of approximation -- are also changed.

{\bf Smoothness in $x$}.

There exists $\delta >0$ such that there exist bounded second order derivatives $\exists f_y'', \, \exists g''$ on $B(x^\star, \delta)$ for all $y\in \Y_{\rm loc}$, i.e.
$\exists   C_{\tilde{f},\,2}, \, C_{g,\,2}<\infty$ such that for all $x \in B(x^\star, \delta)$,  for all $y\in \Y_{\rm loc}$,
\begin{eqnarray}\label{eq:BoundedDerG2}
\max_{1 \leqslant i,j \leqslant n}|\nabla_{ij} \tilde{f}_y(x)| \leqslant C_{\tilde{f},\, 2}, \quad \max_{1 \leqslant i,j \leqslant p}|\nabla_{ij} g(x)|  \leqslant  C_{g,\, 2}.
\end{eqnarray}

{\bf Convergence in $Y$.}


 $\exists  M_{\tilde{f},\,1} <\infty$ such that for 
all $1 \leqslant j  \leqslant p$   and for all $y\in \Y_{\rm loc}$,
\begin{eqnarray}\label{eq:BoundedDerDiffF22}
|\nabla_{j } \tilde{f}_y(A x^\star) - \nabla_{j } \tilde{f}_{y\exact}(A x^\star)| \leqslant M_{\tilde{f},\, 1} ||y - y\exact||.
\end{eqnarray}

{\bf Assumptions on $\delta$}.

Assume that $\delta >0$ satisfies the following conditions  as $\tau \to 0$:
\begin{enumerate}
\item
\begin{gather}\label{eq:CondDeltaLocal2}
\notag \delta \to 0,   \quad \frac{\delta }{\tau} \to 0,  \\
 \frac{  \delta}{\gamma^2} \to \infty \quad \text{(not necessary if $A^T A$ is of full rank)}.
\end{gather}

\item  With high probability,
\begin{eqnarray}\label{eq:LareDevIntegral2}
\Delta_0(B(0,\delta))  \to 0 \quad \text{as} \quad \tau \to 0,
\end{eqnarray}
 where $\Delta_0(B(0,\delta))$ is defined by (\ref{eq:DefD0}).

\end{enumerate}

\subsection{Rate of convergence in Ky Fan distance}


Define
$$
b(\omega) = \nabla h_{Y(\omega)}(x^\star).
$$

For a given data set (i.e. fixed $\omega$), the Prokhorov distance between the posterior distribution and the point mass at $x^\star$ can be bounded as follows. 
\begin{theorem}\label{th:ProkhUpperExt}
 Suppose we have the Bayesian model defined in Section \ref{sec:Setup}, and let the assumptions on $f_y$, $g$ and $\delta$ stated
in Section~\ref{sec:AssumeExt} hold.

Assume that $x^\star =0$ and that $b_i(\omega)>0$ for all $i$, and denote $b_{\min}(\omega) = \min_i b_i(\omega)$.

Then,  $\exists \tau_0  >0$ such that for $\forall \,
\tau \in (0, \tau_0]$ and small enough $\gamma$,
\begin{eqnarray*}
\rho_{\rm P}(\mu_{\rm post}(\omega), \delta_{x^\star})\leq
 \max\left\{   \frac{\Delta_0}{1+\Delta_0 }, \quad
- \frac{\tau \sqrt{p}}{\bar{b}_{\min}(\omega) }  \log\left( \frac{\tau}{ \sqrt{p}\, \bar{b}_{\min}(\omega)}\right)  (1+ \Delta_{4})\right\},
\end{eqnarray*}
where $\Delta_0=\Delta_0(B(0,\delta))$ is defined by (\ref{eq:DefD0}) and   $\Delta_4(\delta, Y(\omega))$ is defined by (\ref{eq:DefDeltaStar}).
\end{theorem}

Recall that $b(\omega) = A^T \nabla \tilde{f}_{Y(\omega)}(x^\star) + \nu \nabla g(x^\star)$. Thus, if the image of $A^T$ includes the whole set $\X$ (well-posed case), the leading term  of $b(\omega)$ for each coordinate is a constant, then the rate of convergence is determined by $-\tau \log \tau$. However, if $\rank(A)<p$ (ill-posed case), then for some coordinates the leading term of $b(\omega)$ is $\nu \, const\to 0$, then the rate of convergence is determined by $-\gamma^2 \log \gamma$.

To have consistency in the ill-posed case, we must have $\tau/\nu = \gamma^2 \to 0$. Hence, in this case to have the convergence we must assume that $\nu = \tau/\gamma^2\to 0$ and $\gamma \to 0$ as $\tau \to 0$.


Now we apply Theorems~\ref{lem:KyFanDiffBounds} and \ref{th:ProkhUpperExt} to obtain an upper bound on the Ky Fan distance.
Define
$$
b^\star = \nabla h_{y\exact}(x^\star).
$$

\begin{theorem}\label{th:KyFanPostExt}
Consider the Bayesian model defined in Section
\ref{sec:Setup}, and suppose that the assumptions on $f_y$ and $g$  stated in Section~\ref{sec:AssumeExt} hold.

Assume that $\X = [0, \infty)^p$, $x^\star =0$, $\nabla_{i} \tilde{f}_{y\exact}(G(y\exact)) > 0$ and $b^\star_i>0$ for all $i$. Denote $b_{\min}^\star = \min_i b^\star_i$.
If \(\rank \,(A^T A) < p\), assume also that \(\gamma \to 0\) and \(\tau/\gamma^2\to 0\) as \(\tau \to 0\).


Then, for small enough $\tau$, $\gamma$ and \(\nu\),
\begin{eqnarray*}
\rho_{\rm K}(\mu_{\rm post}, \delta_{x^\star}) &\leq& \max\left\{2\rho_{\rm K}(Y, y\exact),\,\, \Delta_0^\star, \,\, - \frac{\tau \sqrt{p}}{ b_{\min}^\star } \log\left( \frac{\tau  }{  \sqrt{p} \, b_{\min}^\star }\right)(1 + \Delta_5^\star) \right\},
 \end{eqnarray*}
where $\Delta_0^\star$ is defined by (\ref{eq:D0expr}), and
\begin{eqnarray*}
\Delta_5^\star &=& -1+\frac{ 1+\Delta_4^\star }{1 - \Delta_{11}} \left(1 -    \frac{ \log (1 - \Delta_{11}) }{\log\left( \frac{\tau }{  \sqrt{p}\,  b_{\min}^\star   }\right)} \right),\\
\Delta_{11} &=& \frac{ M_{f1}}{b_{\min}^\star} \rho_{\rm K}(Y, y\exact) + \delta \frac{p [C_{f2}  +\nu C_{g2}/2]}{b_{\min}^\star},\\
\Delta_4^\star &=& \frac{\log\left( (1+\Delta_1^\star)/(1+ \Delta_0^\star)\right)}{\log\left(  \sqrt{p} \,b^\star_{\min}[1 - \Delta_{11}] /\tau \right)},\\
 \Delta_1^\star &=& -1 +  \left(\frac{ 1 - \Delta_{11}}{1 + \Delta_{11}}\right)^p \,  \left[   1- e^{-  \max_i b_{i}^\star(1+\Delta_{11}) \delta/(\sqrt{p}\tau)}  \right]^{p}.
\end{eqnarray*}

Under the assumptions on $\tau$, $\gamma$ and $\delta$ given in Section~\ref{sec:AssumeExt},   $\Delta_{5}^\star(\delta)=o(1)$ as $\tau \to 0$.
\end{theorem}
Hence, in the case that the solution is on the boundary, we have a different rate of convergence of the posterior distribution that is faster than the corresponding rate in the case the solution is an interior point. This fits with other studies of the rate of convergence of the posterior distribution for the error densities with jump \cite{GGS94,GS95}.

{\bf Examples.}
\begin{enumerate}
\item[1.] Rescaled Poisson distribution $Y_t/\tau \sim Pois(A_t x/\tau)$, independent. For $x^\star =0$, we have $\pr(Y_t =0) =1$ for all $t$. The Ky Fan distance between the data and its limit is zero, so we observe exact data.

If $A^T A$ is of full rank, the Ky Fan distance 0 and we recover $x\true$ exactly.
If $A^T A$ is not of full rank, then we can recover $\PA x\true$ exactly, and the upper bound for recovering $\PA x^\star$  is of order $-\gamma^2\log\left(  \gamma^2\right)$ and can be arbitrarily small.
This rate is faster than the rate in the case $x^\star$ is an interior point.

\item[2.] Exponential error distribution: $Y_t - A_t x \sim \Exp(\lambda_t/\tau)$, independent. For $x^\star =0$, we have $Y_t\sim \Exp(\lambda_t/\tau)$. In the well-posed case, the Ky Fan distance between the data and its limit is  $-\Lambda_E \tau \log \tau$, i.e. is of the same order as the rate of contraction of the posterior distribution to  its maximum, where \(\Lambda_E\) is a function of \(\lambda_1, \ldots, \lambda_n\).
In the ill-posed case, the dominating rate is of order $-\gamma^2\log\left(  \gamma^2\right)$ which is faster than the corresponding rate when $x^\star \in int(\X)$.

\end{enumerate}


A comprehensive study of the rate of convergence of inverse problems under a more general setting (when $x^\star$ is an arbitrary point on the boundary) was conducted by \cite{BochkinaGreen}.


%% file: Discussion_IP.tex
\section{Discussion}\label{sec:discussion}

This paper provides a tool to study consistency of the posterior distribution in inverse problems where the likelihood is non-Gaussian and the prior distribution is non-conjugate, without assuming the existence of any moments of the data. It can also be used to apply to infinite dimensional models with other non-standard features. For example, dependence of the coefficients could be modelled by introducing a linear function $G$ that rotates the observed data, i.e. $G(y) = G y$ for a rotation matrix $G$. Non-Gaussian and non-conjugate prior distributions can be considered.  The rate of convergence for severely ill-posed and multivariate inverse problems can also be easily derived from our main result. Asymptotic coverage of Bayesian credible sets under the considered model can be derived from the results of \textcite{BochkinaGreen}.
 
 Another issue is adaptation. The examples given in Section~\ref{sec:pgrows} result in non-adaptive estimators since the dispersion of the prior $\gamma^2$ depends on the smoothness of the unknown function. The number of coefficients $p$ in the model can also be taken as dependent on the data, e.g. via putting a prior distribution on it (e.g. \textcite{Ray_Inv})  which can lead to adaptive estimation, i.e. where the prior distribution does not depend on the  smoothness of the unknown function. Other hierarchical prior distributions can be considered, for instance, where some parameter of the prior distribution, such as scale, is also estimated which can also potentially result in adaptive estimators (see \textcite{vdVaartAdaptive} for hierarchical Bayesian models for nonparametric regression).

In the considered framework of non-Gaussian errors, we can see two interesting phenomena. The first one is that the rate of convergence in the considered setting can be faster than the rate of convergence under the Gaussian noise. This is due to dependence of the variance of the noise on the unknown function. The rate is similar to that of an operator with ill-posedness coefficient $\alpha-s/2$ (provided $\alpha >s/2)$ where the variance of the noise increases as $j^s$. For direct problems, e.g. estimating the intensity of the Poisson process, this phenomenon was studied by \textcite{Bouret}.  The second phenomenon is that it appears to be possible to recover the unknown function with the parametric rate of convergence, for $s$ large enough compared to $\alpha$. This is the effect of self-regularisation, since the information about the unknown function comes not only from the mean, as under homoscedastic Gaussian errors, but also from the variance.

These results raise a number of questions. An important question is about the minimax rate of convergence for such problems, particularly as it can be faster than the minimax rate of convergence for recovering functions under the Gaussian noise. Another interesting question is to identify  random processes whose orthogonal decomposition has coefficients with a distribution from the exponential family. Effect on the posterior inference of possible model misspecification under growing dimensions is another interesting topic of investigation.

%% file: Proof_KyFanConvIP_New.tex

\subsection{Proofs of the results in Section \ref{sec:ConvergenceKyFan}}



\begin{lemma}\label{lem:Approx}
Assume that $H$ defined by \eqref{def:H} is invertible, and that $x^\star$ is an interior point of $\X$.

Let $x\in B_{\delta} = \{ x\in \X: \,||x-x^\star|| \leq \delta \}$, and denote $v = (x-x^\star)/\sqrt{\tau}$.

{\bf 1. Upper bound.} Then, for small enough $\delta$ and $\nu$, we have the following upper bound:
\begin{eqnarray*}
 [h_y(x) - h_y(x^\star)]/\tau &\leq& \frac 1 2 ||\widetilde{H}^{1/2}(v - \widetilde{H}^{-1} H x_{0}/\sqrt{\tau}) ||^2  -\frac 1 {2\tau} x_0^T H \widetilde{H}^{-1} H x_0,
\end{eqnarray*}
where $D = A^T C_{f} A \, ||A|| + \nu C_{g}$ and $\widetilde{H}  = H  + \delta D$.

{\bf 2. Lower bound.}
For small enough $\delta$ and $\nu$, we have the following lower bound:
\begin{eqnarray*}
 [h_y(x) - h_y(x^\star)]/\tau &\geq& \frac 1 2  ||\bar{H}^{1/2}(v - \bar{H}^{-1} H x_{0}/\sqrt{\tau}) ||^2 -   \frac 1 {2\tau} x_0^T H \bar{H}^{-1} H x_0,
\end{eqnarray*}
where $\bar{H} = H -  \delta D$.

\end{lemma}
\begin{proof}

For $x\in B_{\delta}$, approximate $ h_y(x )$ by a quadratic function using Taylor decomposition in a neighbourhood of $x^\star$:
\begin{eqnarray*}
 h_y(x ) &=& h_y(x^\star) + [\nabla h_y(x^\star)]^T (x -x^\star) + \frac 1 2 (x -x^\star)^T \nabla^2  h_y(x_c) (x -x^\star),
\end{eqnarray*}
for some $x_c \in B(x^\star,\delta)$. Using Assumption \eqref{eq:BoundedDerG} stated in Section~\ref{sec:QuadraticApprox}, $\forall w \in B(0,\delta)$, we have
\begin{eqnarray*}
w^T \nabla^2  h_y(x_c) w &\leq&  w^T (H + A^T C_{f} A \delta ||A|| + \nu C_g \delta) w,\\
w^T \nabla^2  h_y(x_c) w &\geq& w^T (H - A^T C_{f} A \delta ||A|| - \nu C_g \delta) w.
\end{eqnarray*}
Dividing by $\tau$ and making the change of variables $v =  (x-x^\star)/\sqrt{\tau}$, we have
 \begin{eqnarray*}
 [h_y(x )-h_y(x^\star)]/\tau &\leq& \frac 1 2  v^T \widetilde{H} v  + \tau^{-1/2} [\nabla h_y(x^\star)]^T v,
\end{eqnarray*}
 \begin{eqnarray*}
 [h_y(x )- h_y(x^\star)]/\tau &\geq&  \frac 1 2  v^T \bar{H} v  + \tau^{-1/2} [\nabla h_y(x^\star)]^T v.
\end{eqnarray*}
Completing the square for each expression and recalling that $x_0 = H^{-1} \nabla h_y(x^\star)$ gives the statement of the lemma.

\end{proof}


\begin{proposition}\label{prop:LaplaceApproxAsympt}
Let assumptions on $f_y, \, g$ in Section \ref{sec:QuadraticApprox} and assumptions (\ref{eq:CondDeltaLocal}) on $\delta$ hold.
Assume that $ H=A^T V_{y}(x^\star) A +\nu B(x^\star) $ is of full rank, and that $\gamma\to 0$ and $\nu \to 0$ as $\tau \to 0$.


Then, for any $\varepsilon \in \left(||\bar{H}^{-1} H x_0||, \delta\right)$ such that
$\varepsilon \sqrt{||H||^{-1}} \to \infty$,
\begin{eqnarray*}
    \frac{\int_{\X \setminus B(x^\star, \varepsilon)} e^{-  h_y(x)/\tau} dx}{\int_{\X} e^{-  h_y(x)/\tau} dx} &\leqslant&
\left[  1- \Phix(B(0, \varepsilon/\sqrt{\tau}), \bar{H}) \right] \frac{1+\Delta_2}{1+\Delta_0}+\frac{\Delta_0}{1+\Delta_0},
\end{eqnarray*}
\begin{eqnarray*}
\int_{B(x^\star, \delta)} e^{- [ h_y(x)- h_y(x^\star)]/\tau} dx \geq \frac{\tau^{p/2} [2\pi]^{p/2}}{[\det({\tilde{H}})]^{1/2}}  \exp\left\{\frac {x_0^T  H \tilde{H}^{-1}{H} x_0 }{2\tau}  \right\} [1+\Delta_3],
\end{eqnarray*}
where $\Delta_3(\delta, y)$ and $\Delta_2$ are defined by
\begin{eqnarray}\label{eq:defD3}
  \Delta_3(\delta, y) &=&  -1 + \Phi(B(0, R(\delta))),
 \end{eqnarray}
\begin{eqnarray}\label{eq:DefDelta12}
\Delta_2(\delta, y) &=& \frac{\exp\left\{ \delta ||D||\, ||\bar{H}^{-1} \nabla h_y(x^\star)||\, ||\widetilde{H}^{-1} \nabla h_y(x^\star)|| /\tau \right\}}{ \Phi(B(0, R(\delta))) }   \left[ \frac{\det(\widetilde{H})}{\det(\bar{H})} \right]^{1/2}   
  \end{eqnarray}
  where $R(\delta) =\sqrt{ \lambda_{\min}(\tilde{H})/\tau} [ \delta-  ||\tilde{H}^{-1} H x_0||]$.

\end{proposition}


\begin{proof}[Proof of Proposition \ref{prop:LaplaceApproxAsympt}.]


Making the change of variables $v =  (x-x^\star)/\sqrt{\tau}$ with Jacobian $J = \tau^{p/2}$ and applying Lemmas~\ref{lem:Approx} and \ref{lem:UpperBoundXG_gen}, we have the following bound 
\begin{eqnarray*}
\hspace{-2cm}&&\int_{B(x^\star, \delta)} e^{-  [h_y(x) - h_{y\exact}(x)]/\tau} dx \geq
\tau^{p/2} \exp\left\{ || \widetilde{H}^{-1/2} H x_0||^2/(2\tau) \right\}\\
\hspace{-2cm}  && \times
\int_{B(0, \delta/\sqrt{\tau})}  \exp\left\{-  \left(v  - \frac{\widetilde{H}^{-1} H x_{0}}{\sqrt{\tau}} \right)^T \widetilde{H} \left(v  - \frac{\widetilde{H}^{-1} H x_{0}}{\sqrt{\tau}} \right)/2\right\} dv \\
\hspace{-2cm} &=&  \tau^{p/2} \exp\left\{ || \widetilde{H}^{-1/2} H x_0||^2/(2\tau) \right\} \, [2\pi]^{p/2}  [\det(\widetilde{H})]^{-1/2} \Phix(B(0, \delta/\sqrt{\tau}), \widetilde{H})\\
\hspace{-2cm} &\geq&  \tau^{p/2} \exp\left\{ || \widetilde{H}^{-1/2} H x_0||^2/(2\tau) \right\} \, [2\pi]^{p/2}  [\det(\widetilde{H})]^{-1/2} \Phi(B(0, R(\delta))),
\end{eqnarray*}
where $R(\delta)$ is defined in the statement of the proposition. As $\tau \to 0$, by the assumptions on $\delta$ we have $\rho_{\rm K}(Y, y\exact) \to 0$, $\nu \to 0$ and hence $||H \bar{H}^{-1}|| \to 1$ and $\det(\widetilde{H}) \to \det(\Omega_{00}) \det(B_{11})$ in the notation of Lemma~\ref{lem:LocalConst}. By Lemma~\ref{lem:UpperBoundXG_gen} and  assumptions \eqref{eq:BoundedDerDiffF2}, for $y\in \Y_{\rm loc}$
$$
||x_0|| \leq ||H^{-1} A^T M_{f1}(y-y\exact)|| + \nu ||H^{-1} \nabla g(x^\star)||  \leq c_1 \rho_{\rm K}(Y, y\exact) + c_2 \nu,     
$$
and, since $\delta (c_1 \rho_{\rm K}(Y, y\exact) + c_2 \nu)^2/\tau \to 0$ and $\delta/\sqrt{\tau} \to \infty$ ($\delta/\gamma \to 0$ for an ill-posed problem), we have  $R(\delta) \to \infty$ and hence $\Delta_3(\delta) \to 0$. Thus, since $\P(Y \in \Y_{\rm loc}) \to 1$, we have the statement of Lemma~\ref{lem:LocalConst}.


Similarly, 
we can obtain an upper bound on this integral:
\begin{eqnarray*}
&&\int_{B(x^\star, \delta) \setminus B(x^\star, \varepsilon)} e^{-  [h_y(x)-h_y(x^\star)]/\tau} dx \leq
\tau^{p/2} \exp\left\{  || \bar{H}^{1/2} H x_{0}||^2/(2\tau) \right\}\\
&\times&\int_{ \varepsilon/\sqrt{\tau} \leq  ||v||  \leq \delta/\sqrt{\tau} }
  \exp\left\{  - \frac 1 {2 } || \bar{H}^{1/2}(v - \tau^{-1/2} \bar{H}^{-1}{H} x_0) ||^2   \right\} dv\\
& \leq& \tau^{p/2} [\det(\bar{H})]^{-1/2} (2\pi)^{p/2} \exp\left\{  || \bar{H}^{1/2} H x_{0}||^2/(2\tau) \right\}
(1- \Phix(B(0, \varepsilon/\sqrt{\tau}), \bar{H})).
 \end{eqnarray*}
Assume that $\delta$ is small enough so that $\bar{H}$ is positive definite.

Combining these results together, we have that for $\varepsilon > ||\bar{H}^{-1} {H} x_0 ||$,
\begin{eqnarray*}
&&\frac{\int_{B(x^\star, \delta) \setminus B(x^\star, \varepsilon)} e^{-  h_y(x)/\tau} dx}{\int_{B(x^\star, \delta)} e^{-  h_y(x)/\tau} dx}
\leq  \frac{  1- \Phix(B(0, \varepsilon/\sqrt{\tau}), \bar{H})}{  \Phi(B(0, R(\delta)))}   \left[ \frac{\det(\widetilde{H})}{\det(\bar{H})} \right]^{1/2}  \\
&&\quad \quad \times  \exp\left\{ \delta \, (  \nabla h_y(x^\star))^T  \bar{H}^{-1} D \widetilde{H}^{-1} \nabla h_y(x^\star) /\tau \right\},
 \end{eqnarray*}
since
\begin{eqnarray*}
\bar{H}^{-1} - \widetilde{H}^{-1} &=& \widetilde{H}^{-1} ( \widetilde{H} - \bar{H}) \bar{H}^{-1} = 2 \delta \widetilde{H}^{-1} D \bar{H}^{-1}.
 \end{eqnarray*}

Now we take into account the error of approximating the integral over $\X$ by the integral over $B(x^\star, \varepsilon)$:
\begin{eqnarray*}
\frac{\int_{\X \setminus B(x^\star, \varepsilon)} e^{-  h_y(x)/\tau} dx}{\int_{\X} e^{-  h_y(x)/\tau} dx}& =&
\frac{\int_{B(x^\star, \delta) \setminus B(x^\star, \varepsilon)} e^{-  h_y(x)/\tau} dx + \int_{\X \setminus B(x^\star, \delta) } e^{-  h_y(x)/\tau} dx}{\int_{B(x^\star, \delta)  } e^{-  h_y(x)/\tau} dx + \int_{\X \setminus B(x^\star, \delta) } e^{-  h_y(x)/\tau} dx}\\
&=&
\frac{\int_{B(x^\star, \delta) \setminus B(x^\star, \varepsilon)} e^{-  h_y(x)/\tau} dx  } {(1+\Delta_0)\int_{B(x^\star, \delta)  } e^{-  h_y(x)/\tau} dx } + \frac{\Delta_0}{1+\Delta_0}\\
&\leq& \frac{\Delta_0}{1+\Delta_0} + \frac{  1- \Phix(B(0, \varepsilon/\sqrt{\tau}), \bar{H})}{  (1+\Delta_0)\, \Phi(B(0, R(\delta))) }   \left[ \frac{\det(\widetilde{H})}{\det(\bar{H})} \right]^{1/2}   \\
&& \times  \exp\left\{ \delta ||D||\, ||\bar{H}^{-1} \nabla h_y(x^\star)||\, ||\widetilde{H}^{-1} \nabla h_y(x^\star)|| /\tau \right\},
\end{eqnarray*}
which gives the statement of the proposition. 
\end{proof}


\begin{proof}[Proof of Theorem \ref{th:ProkhUpperInt}.]

By Strassen's theorem \cite{Dudley}, for any $x$, $\rho_{\rm P} (\mu_{\rm post}(\omega), \delta_{x }) =  \rho_{\rm K} (\xi,
 x)$ where $\xi \sim \mu_{\rm post}(\omega)$. Hence, it is sufficient to find an upper bound on the Ky Fan distance between $\xi$ and $x^\star$.
 We consider $\omega$ such that  $Y(\omega) \in \Y_{\rm loc}$ and small $\tau$ and $\delta$ such that  matrix $\bar{H}$ is of full rank.

Take $\varepsilon >  ||\bar{H}^{-1} H x_0|| $. Using Proposition~\ref{prop:LaplaceApproxAsympt}, we have an upper bound on
$\varepsilon$ satisfies
\begin{eqnarray*}
  \frac{\int_{B(x^\star, \delta) \setminus B(x^\star, \varepsilon)} e^{-  h_y(x)/\tau} dx}{\int_{B(x^\star, \delta)} e^{-  h_y(x)/\tau} dx}
  &\leqslant& \left[ 1 - \Phix(B(0, \varepsilon/\sqrt{\tau}), \bar{H})\right]
(1+\widetilde\Delta_2) + \widetilde\Delta_0 \leq \varepsilon,
\end{eqnarray*}
where $\widetilde\Delta_0 = \Delta_0/(1+ \Delta_0)$ and $\widetilde\Delta_2 = (1+ \Delta_2)/(1+\Delta_0) -1$.
The last  inequality implies that as $\tau/\lambda_{\rm min} ({H}) \to 0$, $\varepsilon \to 0$ and $\varepsilon^2  \lambda_{\rm min} ({H})/\tau  \to \infty$.
 Hence, using Lemmas \ref{lem:KyFanNormal} and \ref{lem:KyFanIneq},
 for small enough $\tau$, the Ky Fan distance between $\xi \sim \mu_{\rm post}(\omega)$ and $x^\star$ is bounded by
\begin{eqnarray*}
\rho_{\rm K}( \xi, x^\star) &\leq&  ||\bar{H}^{-1} H x_0||  + \sqrt{-  4 \tau \trace(\bar{H}^{-1})  \log\left( 4 \tau \trace(\bar{H}^{-1})(1+\widetilde\Delta_2)^2\right) },
\end{eqnarray*}
which is also an upper bound on the Prokhorov distance between $ \mu_{\rm post}(\omega)$ and $\delta_{x^\star}$.

Now we bound the bias using assumption \eqref{eq:BoundedDerDiffF2} and Lemma \ref{lem:UpperBoundXG_gen}:
\begin{eqnarray*}
||\bar{H}^{-1} \nabla h_y(x^\star)|| &\leq& ||\bar{H}^{-1} A^T \nabla \tilde{f}_y(Ax^\star)|| + \nu ||\bar{H}^{-1} \nabla g(x^\star)||\\
 &\leq&   ||\bar{H}^{-1} A^T M_{f1} (y - y\exact)|| + \nu ||\bar{H}^{-1} \nabla g(x^\star)||.
\end{eqnarray*}


Hence an upper bound on the Ky Fan distance is the smallest   $\varepsilon \geq \widetilde\Delta_0 >0$ that satisfies the obtained upper bound. Therefore, the Ky Fan distance (and thus, the corresponding Prokhorov distance) is bounded from above by
\begin{eqnarray*}
\hspace{-1.2cm}\rho_{\rm P}(\mu_{\rm post}(\omega), \delta_{x^\star})
 &\leqslant&   \max\left\{   \widetilde\Delta_0, \quad
  ||\bar{H}^{-1}  A^T  M_{f1}(Y(\omega)-y\exact)|| +  \nu   ||\bar{H}^{-1}   \, \nabla g(x^\star) ||  \right.  \notag\\
&+& \left. \sqrt{- 4 \tau \trace(\bar{H}^{-1}) \log \left( 4 \tau \trace(\bar{H}^{-1})\right) }(1+\Delta_{\star}(\delta,   Y(\omega)))  \right\},
\end{eqnarray*}
where  $\Delta_0=\Delta_0(B(0,\delta))$ defined by (\ref{eq:DefD0}) and $\Delta_{\star}(\delta, y)$ is the following bound on the error:
\begin{eqnarray}\label{eq:DefDeltaStar1}
&&\frac{1}{1+\tilde\Delta_2} \left(1 + \frac{\log[(1+\tilde\Delta_2)^2]}{-\log(4 \tau \trace(\bar{H}^{-1}))} \right)^{1/2} \notag\\
&=& \frac{1+ \Delta_0}{(1+ \Delta_2)}\left[ 1 +2\frac{\log(1+ \Delta_2)- \log(1+ \Delta_0)}{ -\log\left( 4 \tau \trace(\bar{H}^{-1})\right) } \right]^{1/2} -1 \notag\\
&\leq& \left( 1 + \frac{\Delta_0- \Delta_2}{(1+ \Delta_2)}\right)\left[ 1 - \frac{\Delta_0- \Delta_2}{-(1+ \Delta_2) \log\left( 4 \tau \trace(\bar{H}^{-1}) \right) } \right] -1 \notag\\
&\leq& \frac{\Delta_0 - \Delta_2}{(1+ \Delta_2)}\left( 1 -  \frac{1}{-\log\left( 4 \tau \trace(\bar{H}^{-1}) \right) } \right)  \notag\\
&\leq& \Delta_0 + \Delta_2 \stackrel{def}{=} \Delta_{\star}(\delta, y),
\end{eqnarray}
using the inequality $(1 + c\log(1-x))^{1/2}\leq 1-cx/2$ for $x \in (0, \min(1, 2/c))$, $c>0$, $4 \tau \trace(\bar{H}^{-1}) < 1/e$ and that $\Delta_0$, $\Delta_2$ are nonnegative.

\end{proof}

\input{Proof_Lift}

\begin{proof}[Proof of Theorem \ref{th:KyFanPostInt}. ]
Now we prove Theorem \ref{th:KyFanPostInt}   in the notation defined in the proof of Theorem~\ref{th:ProkhUpperInt}.
 We apply Theorem \ref{lem:KyFanDiffBounds} with  $\Omega_1
=\{\omega: \, ||Y(\omega)-y\exact||\leq \rho_{\rm K}(Y,
y\exact)\}$ and $\Omega_2 = \Omega \setminus \Omega_1$ with
$\P(\Omega_2)\leq \rho_{\rm K}(Y, y\exact)$ by the definition of
Ky Fan distance, with the bound $\Phi_1$ given in Theorem
\ref{th:ProkhUpperInt}. For small enough $\tau$,
the assumption of Theorem~\ref{th:ProkhUpperInt} that  $\bar{H}$
is of full rank holds on $\Omega_1$, as we
shall show below. The bound stated in Theorem~\ref{th:ProkhUpperInt} depends on $y$ via $y-y\exact$ and $H$, so we bound it above so that it
depends on $y$ only via $||y - y\exact||$. 



Since $ V_y(x)  = \nabla^2 \tilde{f}_y(Ax)$, using Assumption \eqref{eq:BoundedDerDiffF2},
 for any $v\in \mathbb{R}^p$,
\begin{eqnarray*}
&& |v^T [H_{y}(x^\star) - H_{y\exact}(x^\star)] v| \leq  v^T A^T M_{f2}^{1/2} \diag(|y_i-{y\exact}_i|) M_{f2}^{1/2} A v \notag\\
&& \quad \quad \leq \max_i |y_{i} -{y\exact}_{i}| \, ||M_{f2}^{1/2} A v||^2 \leq ||y -{y\exact}||  ||M_{f2}^{1/2} A v||^2. 
\end{eqnarray*}

Hence, on $\Omega_1$,
$$
||\bar{H}^{-1} H_{\nu}||\leq ||[I  - \delta H_{\nu}^{-1} D  - H_{\nu}^{-1} A^T M_{f2} A \,\, \rho_{\rm K}(Y, y\exact)  ]^{-1}  || \leq [1 - \tilde\lambda]^{-1}
$$
where $\tilde\lambda = \delta ||H_{\nu}^{-1} D || + \rho_{\rm K}(Y, y\exact) ||H_{\nu}^{-1} A^T M_{f2} A|| \, $, and
$$
[1+\tilde\lambda]^{-1} \leq  ||[I  + \delta H_{\nu}^{-1} D   + H_{\nu}^{-1} A^T M_{f2} A \,\, \rho_{\rm K}(Y, y\exact)  ]^{-1}  || \leq ||\tilde{H}^{-1} H_{\nu}||. 
$$
This implies that
$$
||\tilde{H}^{-1} H x_0||  \leq ||\bar{H}^{-1} H x_0||   \leq    \frac{||H_{\nu}^{-1} A^T M_{f1} (Y-y\exact)|| + \nu ||H_{\nu}^{-1} \nabla g(x^\star)|| }{1-\tilde\lambda} \stackrel{def}{=} bias.
$$


The $k$-th smallest eigenvalues of matrices $H_{\nu}^{-1}$ and $\bar{H}^{-1}$ satisfy
$
\lambda_k(\bar{H}^{-1}) \leq  \frac{\lambda_k(H_{\nu}^{-1})}{1-\tilde\lambda},
$
that implies
$$
\trace(\bar{H}^{-1}) = \sum_{i=1}^{p} \lambda_k(\bar{H}^{-1})\leq \frac{\trace(H_{\nu}^{-1}) }{1 - \tilde\lambda}.
$$


The ratio of the determinants is bounded by
\begin{eqnarray*}
\frac{\det(\widetilde{H})}{\det(\bar{H})} =\frac{\det(H_{\nu}^{-1} \widetilde{H})}{\det(H_{\nu}^{-1}\bar{H})} \leq \left(\frac{1+\tilde\lambda}{1-\tilde\lambda}\right)^{p}.
\end{eqnarray*}


Also, we have
\begin{eqnarray*}
R(\delta) \geq \sqrt{ \frac{\lambda_{\min}(H_{\nu})}{\tau}} \sqrt{1 - \tilde\lambda} [ \delta-  bias] \stackrel{def}{=} R^\star(\delta).
\end{eqnarray*}

Therefore, on $\Omega_1$, we have
\begin{eqnarray*}
\Delta_3(\delta, y) &\geq& \Delta_3^\star(\delta) =  -1 + \Phi(B(0, R^\star(\delta))),
  \end{eqnarray*}
\begin{eqnarray*}
\Delta_2(\delta, y) &\leq&  \exp\left\{ \delta  \, ||D ||  \frac{[ ||H_{\nu}^{-1} A^T M_{f1}|| \rho_{\rm K}(Y, y\exact) + \nu ||H_{\nu}^{-1} \nabla g(x^\star)||]^2}{
\tau[  1  -  \widetilde\lambda]^2 }  \right\}\notag \\
&\times&  \left[1+ \Delta_3^\star(\delta)\right]^{-1} \left(\frac{1+\tilde\lambda}{1-\tilde\lambda}\right)^{p/2} -1 \stackrel{def}{=} \Delta_2^\star(\delta).
  \end{eqnarray*}

Under the assumptions on $\delta$, $\rho_{\rm K}(Y, y\exact)$ and $\tau$, we have $R(\delta) \to \infty$ and hence $\Delta_3^\star \to 0$ as $\tau \to 0$; also, we have $\widetilde\lambda \to 0$ and $\Delta_2^\star \to 0$.


Hence, on $\Omega_1$  $\Delta_0$ is bounded from above  by
\begin{eqnarray}\label{eq:DefDeltaStar0}
\Delta_0^\star(B(0,\delta)) &=& \frac{\tau^{p/2} \int_{\X \setminus B(x^\star, \delta)} \exp\left\{- \tau^{-1}[ h_{Y(\omega)}(x) - h_{Y(\omega)}(x^\star)]\right\} dx}
{ \exp\left\{ [2\tau]^{-1} [\bar{c}_1\nu - \bar{c_2}\rho_{\rm K}(Y, y\exact)]^2   \right\}}\\
 &\times&    \frac{[1 +\widetilde\lambda]^{p/2}}{[1+\Delta_3^\star]} \frac{[\det(H_{\nu})]^{1/2}}{ [2\pi]^{p/2} }, \notag
\end{eqnarray}
where $\bar{c}_1 = ||H_{\nu}^{-1} A^T M_{f1}||/(1-\widetilde\lambda)$ and $\bar{c}_2 = ||H_{\nu}^{-1} \nabla g(x^\star)||/(1-\widetilde\lambda)$.

Therefore, on $\Omega_1$,
\begin{eqnarray*}
\rho_{\rm P}(\mu_{\rm post}(\omega), \delta_{x^\star}) &\leq&  
\bar{c}_1 ||Y-y\exact|| + \bar{c}_2 \nu  \\
&+&  \left. \frac{\sqrt{  4\tau \trace(H_\nu^{-1})}}{\sqrt{ 1 - \tilde\lambda}} \sqrt{-\log  \left(\frac{4\tau \trace(H_\nu^{-1})}{1 - \tilde\lambda}\right) } (1+ \Delta_{\star}(\delta, y))\right\},
\end{eqnarray*}
since the function $ - x \log x$ increases for $x\in(0, 1/e)$.

The bound  on $\rho_{\rm P}(\mu_{\rm post}(\omega), \delta_{x^\star})$ increases as a function of $||Y-y\exact||$.
 Using the lifting Theorem \ref{lem:KyFanDiffBounds}, we have that, for small enough $\tau, \nu$, $\delta$,
\begin{eqnarray*}
\hspace{-1.5cm}\rho_{\rm K}(\mu_{\rm post}, \delta_{x^\star}) &\leq& \max\left\{2\rho_{\rm K}(Y, y\exact), \, \Delta_0^\star, \,\bar{c}_1 \rho_{\rm K}(Y, y\exact) + \bar{c}_2 \nu  \right.\\
&+&  \left.  \sqrt{ - 4\tau \trace(H_\nu^{-1}) \log  \left( 4\tau \trace(H_\nu^{-1}) \right) } (1+ \Delta_{\star, K}(\delta))\right\},
\end{eqnarray*}
with $\Delta_{\star, K}(\delta) =\Delta_0^\star  + \Delta_2^\star + \tilde\lambda$, i.e. 
\begin{eqnarray}\label{eq:DefDeltaStarK}
\hspace{-1.5cm}\Delta_{\star, K}(\delta) &=& \frac{\exp\left\{ \delta  \, ||D ||  [ \bar{c}_1 \rho_{\rm K}(Y, y\exact) + \bar{c}_2 \nu ]^2 /\tau  \right\}}{ \Phi(B(0, R^\star(\delta)))}
 \left(\frac{1+\tilde\lambda}{1-\tilde\lambda}\right)^{p/2} -1 + \Delta_0^\star+ \tilde\lambda  
\end{eqnarray}
where $\Delta_0^\star$ is defined by \eqref{eq:DefDeltaStar0} and here we collect other constants: 
\begin{eqnarray*}
\widetilde\lambda &=& \delta ||H_{\nu}^{-1} D || + \rho_{\rm K}(Y, y\exact) ||H_{\nu}^{-1} A^T M_{f2} A||\\
D &=& A^T C_{f} A \, ||A|| + \nu C_{g},\\
\bar{c}_1 &=& ||H_{\nu}^{-1} A^T M_{f1}||/(1-\widetilde\lambda)\\ 
\bar{c}_2 &=& ||H_{\nu}^{-1} \nabla g(x^\star)||/(1-\widetilde\lambda),\\
R^\star(\delta) &=& \sqrt{ \frac{\lambda_{\min}(H_{\nu})}{\tau}} \sqrt{1 - \tilde\lambda} [ \delta-  \bar{c}_1 \rho_{\rm K}(Y, y\exact) - \bar{c}_2 \nu].
\end{eqnarray*}

 Thus, we have the statement of Theorem \ref{th:KyFanPostInt}.


To have a more accurate upper bound in terms of the distance between $\tilde{Y} = H_{\nu}^{-1} A^T M_{f1} Y$ and $\tilde{y}\exact = H_{\nu}^{-1} A^T M_{f1} y\exact$ that is necessary for the case of growing dimension, we have
\begin{eqnarray*}
|v^T [H_{Y}(x^\star) - H_{y\exact}(x^\star)] v| &\leq&  ||\tilde{Y}-\tilde{y}\exact|| \,    v^T  M_{f1}^{-1} M_{f2} M_{f1}^{-1} v,
\end{eqnarray*}
and hence we have different values of $\widetilde\lambda$ and $c_1$:
\begin{eqnarray*}
\widetilde\lambda &=& \delta ||H_{\nu}^{-1} D || + \rho_{\rm K}(\tilde{Y}, \tilde{y}\exact) ||M_{f1}^{-1} M_{f2} M_{f1}^{-1}||,\\
\bar{c}_1 &=& 1/(1-\widetilde\lambda),
\end{eqnarray*}
and in the upper bound on $\rho_{\rm K}(\mu_{\rm post}, \delta_{x^\star})$, $\rho_{\rm K}(Y, y\exact)$ is replaced by $\rho_{\rm K}(\tilde{Y}, \tilde{y}\exact)$.

\end{proof}


\subsection{Proofs of the results in Sections \ref{sec:ExpFamily} and \ref{sec:pgrows}}

\begin{proof}[Proof of Lemma \ref{lem:XYcondExp}.]

For $\ETA =   Ax$ and $x \in B(x^\star, \delta)$, the first condition in \eqref{eq:BoundedDerG} is satisfied due to
\begin{eqnarray*}
|v^T (\nabla^2 \tilde{f}_y( Ax) - \nabla^2 \tilde{f}_y(y\exact))v| &\leq&  \delta \, ||A||\, v^T  C_{f,\, y} v \quad \forall v\in \mathbb{R}^n
\end{eqnarray*}
with $C_{f,\, y} = \diag(\max_{x \in B(x^\star, \delta)} | y_i b'''([ Ax]_i) - c'''([ Ax]_i)| $.

On $\Y_{\rm loc} = \{y\in \Y: \, ||y-y\exact|| \leq \rho_{\rm K}(Y, y\exact)\}$,
$$
C_{f,\, y} \leq C_{f} = \diag\left(\max_{x \in B(x^\star, \delta)} [ (|{y\exact}_i| + \rho_{\rm K}(Y, y\exact))| b'''([ Ax]_i)| + |c'''([ Ax]_i)|]\right).
$$


Conditions on convergence in $Y$ \eqref{eq:BoundedDerDiffF2} are satisfied with
\begin{eqnarray*}
|v^T( \nabla \tilde{f}_y(y\exact) - \nabla  \tilde{f}_{y\exact}(y\exact))| &=& | \sum_i v_i (y_i-{y\exact}_i) b'({y\exact}_i)|\notag\\ &=& |v^T M_{f,\, 1}(y - y\exact)|,\\
|v^T (\nabla^2 \tilde{f}_y(y\exact) - \nabla^2 \tilde{f}_{y\exact}(y\exact))v| &=& |v^T \diag((y_i-{y\exact}_i) b''({y\exact}_i)) v|
\end{eqnarray*}
with $M_{f,\, 1} = \diag(|b'({y\exact}_i)|)$ and $M_{f,\, 2} = \diag(|b''({y\exact}_i)|)$.
\end{proof}

\begin{proof}[Proof of Lemma \ref{lem:SpectralRate}.] To obtain an appropriate rate for random bias, we consider the response random variable $\tilde{Y} = H_{\nu}^{-1} A^T M_{f1} Y$ and $\tilde{y}\exact = H_{\nu}^{-1} A^T M_{f1}y\exact$ (see the end of the proof of Theorem~\ref{th:KyFanPostInt} for the modified constants $\bar{c}_1$ and $\tilde\lambda$). The random vector $\tilde{Y}$ has independent components, with the variance
$$
\Var(\tilde{Y}_i)  \asymp \Var\left( \frac{i^{-\alpha+s} }{i^{-2\alpha+s} + \nu i^{2\kappa+1}} Y_i\right) \asymp  \frac{\tau  \, i^{-2\alpha+s} }{[i^{-2\alpha+s} + \nu i^{2\kappa+1}]^2}.
$$


Denote $m = 2\alpha-s+2\kappa +1$. Assume that $m>0$ and $2\alpha-s + 4\kappa+1 >0$. Then,  by Lemma~\ref{lem:Knapik},
\begin{eqnarray*}
M &=& \sum_{i=1}^p \Var(\tilde{Y}_i)/\tau \asymp \sum_{i=1}^p \frac{i^{-2\alpha+s} }{[i^{-2\alpha+s} + \nu i^{2\kappa+1}]^2}
 =  \sum_{i=1}^p \frac{\nu^{-2} \, i^{-2\alpha+s - 4\kappa-2} }{[ \nu^{-1} i^{-2\alpha+s- 2\kappa-1} + 1 ]^2}\\
   &\leq& C  \min(\nu^{-1/m}, p)^{(2\alpha-s +1)_+} [\log p]^{I(s=2\alpha+1)}.
   \end{eqnarray*}
According to Proposition~\ref{th:KyFanData}, the Ky Fan distance between $\tilde{Y}$ and $\tilde{y}\exact$ is
$$
\rho_{\rm K}(\tilde{Y}, \tilde{y}\exact) \leq C \tau^{1/2} \min(\nu^{-1/m}, p)^{( \alpha-s/2 +1/2)_+}[ \log (p/\tau) ]^{(1+I(s=2\alpha+1))/2}.
$$
The rate is parametric ($\tau^{1/2}$ up to a log factor) if  $s >2\alpha  +1 $. If  $s <2\alpha  +1 $, the rate is slower than  parametric; it tends to 0 if
$ \min(\nu^{-1/m}, p) \ll \tau^{-1/(2\alpha-s +1)}$.
The Ky Fan rate of convergence of the posterior distribution  cannot be faster than this rate (Theorem~\ref{th:KyFanPostInt}).

We have two more terms that we bound using Lemma~\ref{lem:Knapik}.

Now we assess  nonrandom bias under the Gaussian prior with zero mean, so that $\nabla g(x\true) = B x\true$:
\begin{eqnarray*}
\nu ||H_{\nu}^{-1} \nabla g(x\true)|| &\asymp& \nu\left[\sum_{i=1}^p \frac{i^{2(2\kappa+1) - 2\beta-1}}{[i^{-2\alpha+s}  + \nu i^{2\kappa +1}]^2}\right]^{1/2}\\
&=& \left[\sum_{i=1}^p \frac{i^{ - 2\beta-1}}{[i^{-2\alpha+s -2\kappa -1} \nu^{-1} +1]^2}\right]^{1/2}\\
&\leq& C  \nu \min(\nu^{-1/m},p)^{(m -\beta)_+} [ \log p]^{ I(\beta=m)/2}.
\end{eqnarray*}
A different prior distribution would give a different bound.

%
According to Proposition~\ref{th:KyFanData}, the variance term is determined by
\begin{eqnarray*}
\tau \trace(H_{\nu}^{-1}) &\asymp&  \tau \sum_{i=1}^p [i^{-2\alpha+s}  + \nu i^{2\kappa +1}]^{-1}=  \tau \nu^{-1} \sum_{i=1}^p \frac{i^{-2\kappa -1}}{i^{-2\alpha+s-2\kappa -1}\nu^{-1}  +  1 }\\
  &\leq&  C\tau \min(\nu^{-1/m},p)^{(2\alpha-s +1)_+} [\log p]^{I(s=2\alpha  +1)}
\end{eqnarray*}
using Lemma~\ref{lem:Knapik}.
Note that the variance term is of the same order as the upper bound on $\rho_{\rm K}(\tilde{Y}, \tilde{y}\exact)$.

Combining together all the terms, we have the statement of the lemma. 

\end{proof}

%% file: Proof_Lift.tex
\begin{proof}[Proof of Theorem~\ref{lem:KyFanDiffBounds}.]
First we note that
\begin{eqnarray*}
&& \P\left\{\, d_x (X_1(\omega), X_2(\omega) ) \leqslant  \Phi_1(\rho_{\rm K} (Y_1, Y_2)) \cap \Omega_1 \right\}\\
&& +\P\left\{\, d_x (X_1(\omega), X_2(\omega) ) \leqslant  \Phi_2(\rho_{\rm K} (Y_1, Y_2)) \cap \Omega_2 \right\}\\
&\geq& \P\left\{\, \Phi_1(d_y (Y_1(\omega), Y_2(\omega) )) \leqslant  \Phi_1(\rho_{\rm K} (Y_1, Y_2)) \cap \Omega_1 \right\}\\
&&+ \P\left\{\, \Phi_2(d_y (Y_1(\omega), Y_2(\omega) )) \leqslant  \Phi_2(\rho_{\rm K} (Y_1, Y_2)) \cap \Omega_2 \right\}\\
&=& \P\left\{\,  d_y (Y_1(\omega), Y_2(\omega) )  \leqslant   \rho_{\rm K} (Y_1, Y_2) \cap \Omega_1 \right\}\\
&& +\P\left\{\,  d_y (Y_1(\omega), Y_2(\omega) ) \leqslant  \rho_{\rm K} (Y_1, Y_2) \cap \Omega_2 \right\}\\
&=& \P\left\{\,  d_y (Y_1(\omega), Y_2(\omega) )  \leqslant   \rho_{\rm K} (Y_1, Y_2) \right\}\\
&\geq&  1 -  \rho_{\rm K} (Y_1, Y_2).
\end{eqnarray*}

On the other hand,
\begin{eqnarray*}
&& \P\left\{\, d_x (X_1(\omega), X_2(\omega) ) \leqslant  \Phi_1(\rho_{\rm K} (Y_1, Y_2)) \cap \Omega_1 \right\}\\
&& + \P\left\{\, d_x (X_1(\omega), X_2(\omega) ) \leqslant  \Phi_2(\rho_{\rm K} (Y_1, Y_2)) \cap \Omega_2 \right\}\\
&\leq& \P\left\{\, d_x (X_1(\omega), X_2(\omega) ) \leqslant
\Phi_1(\rho_{\rm K} (Y_1, Y_2)) \right\} +  \P\left\{\,  \Omega_2
\right\}.
\end{eqnarray*}
Putting these together implies
\begin{eqnarray*}
  \P\left\{\, d_x (X_1(\omega), X_2(\omega) ) >
\Phi_1(\rho_{\rm K} (Y_1, Y_2)) \right\} \leq \rho_{\rm K} (Y_1, Y_2) + \P\left\{\,  \Omega_2
\right\},
\end{eqnarray*}
hence, using Lemma~\ref{lem:KyFanIneq}, we have
\begin{eqnarray*}
\rho_{\rm K} (X_1 , X_2) \leqslant \max\left\{
\Phi_1(\rho_{\rm K} (Y_1, Y_2)), \rho_{\rm K} (Y_1, Y_2) + \P( \Omega_2
)\right\},
\end{eqnarray*}
and we have the first statement. The second statement follows from the first inequality
and
\begin{eqnarray*}
&& \P\left\{\, d_x (X_1(\omega), X_2(\omega) ) \leqslant  \Phi_1(\rho_{\rm K} (Y_1, Y_2)) \cap \Omega_1 \right\}\\
&&+ \P\left\{\, d_x (X_1(\omega), X_2(\omega) ) \leqslant  \Phi_2(\rho_{\rm K} (Y_1, Y_2)) \cap \Omega_2 \right\}\\
&\leq& \P\left\{\, d_x (X_1(\omega), X_2(\omega) ) \leqslant
\max[\Phi_1(\rho_{\rm K} (Y_1, Y_2)) ,
\Phi_2(\rho_{\rm K} (Y_1, Y_2)) ]\right\}.
\end{eqnarray*}

\end{proof}

%% file: Aux_KyFanDist.tex

\subsection{Ky Fan distance inequalities}\label{sec:KyFanGauss}


\begin{lemma}\label{lem:KyFanExp}
Assume that $A \to 0$ and $A\in (0, e^{-1}]$. Then the solution
of
$$
\exp\{ - z/A \} = z
$$
satisfies
$$
z  =  -A \log(A)(1 +w_{A}),
$$
where $w_{A} \leqslant 0$ and $w_{A} = o(1)$ as $A\to 0$.

\end{lemma}

\begin{proof}[Proof of Lemma~\ref{lem:KyFanExp}.]
Taking the logarithm of the given expression, we have
$$
 - z/A  = \log z
$$
Since $A \to 0$, we must have $ z/\log z \to 0$ which implies $z
\to 0$. Denote $ f = z/A $, i.e. $z=Af$. Hence, the equation above
can be rewritten as
$$
-f = \log A + \log f
$$
implying that $f\to \infty$ as $A\to 0$ at the rate $f = - \log A
(1+o(1))$. Hence, the solution is $ z= -A\log A(1+o(1))$.

To show that $z  \leqslant z_* =  -A \log(A)$, we note that for
$A\leqslant e^{-1}$,
$$
\exp\{ z_*/A \} z_* = \exp\{ - \log(A)\} (-A \log(A)) = -\log(A) \geqslant 1 = \exp\{ z/A \} z
$$
implying the desired inequality.

\end{proof}


The following lemma  follows obviously  from the definition of Ky
Fan distance.

\begin{lemma}\label{lem:KyFanIneq}
If $\P( d(X,Y) > \varepsilon_1) \leqslant \varepsilon_2$ for some
$\varepsilon_1, \varepsilon_2 \in (0,1)$, then $\rho_{\rm K} (X,
Y) \leqslant \max(\varepsilon_1, \varepsilon_2)$.
\end{lemma}


%% file: Proof_DataKyFan.tex

\begin{proof}[Proof of Lemma \ref{lem:KyFanPoisson}.]
Apply the Chernoff-Cramer bound to obtain that for all $t$ and
all $x,\varepsilon >0$,
$$
\P(||Y  - \mu  || > \varepsilon) \leqslant e^{-\varepsilon x} \E e^{x||Y  - \mu  ||}  \leqslant e^{-\varepsilon x} \E e^{x||Y  - \mu  ||_1} = e^{-\varepsilon x} \prod_t \E e^{x|Y_t  - \mu_t|}
$$
Now, $\E e^{x|Y_t  - \mu_t|} \leqslant \E e^{x(Y_t  - \mu_t)} + \E
e^{-x(Y_t  - \mu_t)}$. The cumulant function of a Poisson
random variable $Z$ with parameter $\lambda$ is $\log \E e^{ \varepsilon
Z} = \lambda [e^{\varepsilon} - 1]$; hence, for $Y_t= \tau Z$
and $\lambda =
 \mu_t/\tau$, the cumulant function of $Y_t - \mu_t$ is
\begin{eqnarray*}
c_{t}(x) &=& \log \E e^{x(Y_t  - \mu_t)} = \log \E e^{x \tau Z} - x  \mu_t =
 \frac{\mu_t}{\tau} [e^{x\tau} - 1 - x  \tau].
\end{eqnarray*}
Hence, the cumulants of the rescaled Poisson distribution are
$\kappa_{k} = \mu_t \tau^{k-1}$.
Similarly,
$$
\log \E e^{-x(Y_t  - \mu_t)} =  \frac{\mu_t}{\tau} [e^{-x\tau} - 1 + x  \tau] \leqslant c_t(x)\quad \forall x>0.
$$
 Hence, denoting $M=2 \sum_t \mu_t$, we have
$$
\P(||Y  - \mu  || > \varepsilon) \leqslant   e^{-\varepsilon x} e^{2 \sum_t c_t(x)} = \exp\{-\varepsilon x + M  [e^{x\tau} - 1 - x  \tau]/\tau\}.
$$
Since $x>0$ is arbitrary, we can take $x$ corresponding to the
minimum of the upper bound, which is achieved at $x  =\tau^{-1}\log(1 + \varepsilon /M)$, implying
$$\P(||Y  - \mu  || > \varepsilon) \leqslant
  \exp\left\{ - \frac{ \varepsilon+M}{\tau} \log\left(1 + \frac{\varepsilon}{M}\right) + \frac \varepsilon {\tau} \right\}
\leqslant  \exp\left\{ - \frac{ \varepsilon^2}{2M\tau} \left(1 - \frac{\varepsilon}{3M}\right) \right\},
$$
due to the inequality $(1+x) \log(1+x) - x \geqslant - \frac{x^2} 2 (1
-   \frac x 3)$ for small enough $x>0$. For $\varepsilon\leqslant 3M/2$ we have
\begin{eqnarray*}
\P(||Y  - \mu  || > \varepsilon) \leqslant \exp\left\{ - \frac{ \varepsilon^2}{4M\tau} \right\}.
\end{eqnarray*}
Using Lemma \ref{lem:KyFanExp}, for $ \tau\leqslant
1/(2eM)$, the solution of $\exp\{ -
\varepsilon^2/(2M\tau) \} = \varepsilon^2$ satisfies
$$
\varepsilon =  \sqrt{-2\tau M\log(2\tau M)} (1+\omega),
$$
where $\omega = o(1)$ as $\sigma \to 0$ and $\omega \leqslant 0$.

\end{proof}

\begin{proof}[Proof of Proposition \ref{th:KyFanData}.]

1. Following the rescaled Poisson example, we have that the
cumulant function for $Y_t$ is bounded by
\begin{eqnarray*}
c_{t}(x) &=& \log \E e^{x Y_t} = x\mu_t + \frac {x^2} 2  w_t \tau +  \sum_{i=3}^{\infty} \frac{x^k}{k!} \kappa_k\\
&\leqslant& x\mu_t + \frac {x^2} 2  w_t \tau +  \frac 1 {\tau }\sum_{i=3}^{\infty} \frac{(x  \tau)^k}{k!} C_t w_t\\
&=&  x\mu_t + \frac {x^2} 2  w_t \tau +  \frac {C_t w_t} {\tau }[e^{ x  \tau} - 1 -  x \tau - (x  \tau)^2/2]\\
&\leqslant& x\mu_t +  \frac {C_t w_t} {\tau }[e^{ x \tau} - 1 -  x \tau],
\end{eqnarray*}
since $C_t \geqslant 1$.
Similarly, $ \log \E e^{x Y_t} $ can be bounded in the same way.
 Hence, we have
$$
\P(||Y  - \mu  || > \varepsilon) \leqslant  e^{-\varepsilon x} e^{2 \sum_t c_t(x)}= \exp\{ -\varepsilon x + \frac {M} {\tau }[e^{ x \tau} - 1 -  x \tau] \}.
$$
where $M = 2\sum_t C_t w_t$. Now, this is the same upper bound as
for the rescaled Poisson distribution. Hence, we have the same
inequality for the Ky Fan distance.

2. Apply the Markov inequality to the random variable  $||Y - \mu
||^K$:
$$
\P(||Y - \mu || > z) \leqslant \frac{\E ||Y  -
\mu||^K}{z^K} \leqslant \frac{\E ||Y  -
\mu||_K^K }{z^K} \leqslant \frac{n \tau^{m(K)/2} L_K}{z^K}.
$$
Hence, an upper bound on the Ky Fan distance satisfies $n \tau
L_K/z^K= z$, i.e. $z = [n \tau^{m(K)/2} L_K]^{1/(K+1)} $.

\end{proof}

\begin{proof}[Proof of Lemma~\ref{lem:KyFanNormal}.] Check conditions of Proposition~\ref{th:KyFanData}. Let $\Sigma = U \Lambda U^T$ where $U$ is a rotation matrix and $\Lambda =\tau \diag(\sigma_1^2,\ldots, \sigma_p^2)$. Introduce random variable $Z = U \xi$; then it is sufficient to bound the Ky Fan distance between $Z$ and its mean $\mu_z = U\mu$ since $\rho_{\rm K}(\xi, \mu) = \rho_{\rm K}(Z, U\mu)$ due to $||\xi-\mu|| = ||Z- U \mu||$.

The cumulants of $Z_i$ are $\mu_{Z,i}$ and $\tau \sigma_i^2$, and the remaining cumulants are zero, hence the bound on the cumulants of order 2 and higher is $\tau\sigma_i^2$. Then, the conditions of the proposition are satisfied with $w_i = \sigma_i^2$ and $C_t=1$ and thus, with $M = 4\sum_i \sigma_i^2= 4 \trace(\Sigma)/\tau$, we have
$$
\rho_{\rm K}(\xi, \mu) \leq \left(- \tau M\log\{\tau M\}\right)^{1/2} = \left(- 4 \trace(\Sigma)\log\{4 \trace(\Sigma)\}\right)^{1/2},
$$
provided $\trace(\Sigma) < 1/(4e)$, which proves the lemma.

\end{proof}


%% file: Proof_PureBoundary.tex

\subsection{Proofs of the results in Section \ref{sec:boundary}}


\begin{lemma}\label{lem:ApproxBoundary}
Denote $\delta_b = \frac{\delta p}{2} [C_{\tilde{f},2} A^T A +\nu C_{g2}I] \mathbf{1}$, and assume that $b_i(\omega) >0$ for all $i$.

Let $x\in B_{\delta} = \{ x\in \X: \,||x-x^\star|| \leq \delta \}$, $x^\star =0$.
Then, for small enough $\delta$ and $\nu$, we have the following bounds:
\begin{eqnarray*}
h_y(x) - h_y(x^\star) &\leq& (b(\omega) + \delta_b)^T (x-x^\star),\\
h_y(x) - h_y(x^\star)  &\geq& (b(\omega) - \delta_b)^T (x-x^\star).
\end{eqnarray*}

\end{lemma}
\begin{proof}

Approximate $ h_y(x )$ by a linear function using Taylor decomposition in a neighbourhood of $x^\star$:
\begin{eqnarray*}
 h_y(x ) &=& h_y(x^\star) + [\nabla h_y(x^\star)]^T (x -x^\star) +   \Delta_{00}(x ).
\end{eqnarray*}
Similarly to the proof of Lemma~\ref{lem:Approx}, bound $\Delta_{00}$ for $w = x-x^\star \in B(0,\delta) \cap (\X - x^\star)$ using Taylor decomposition of $h_y(x )$: $\exists x_c \in \langle x , x^\star\rangle$:
\begin{eqnarray*}
| \Delta_{00}(\delta)|  &=& \left|\frac 1 2   \sum_{ij} \nabla_{ij} h_y(x_c) (x_i-x^\star_{i})(x_j-x^\star_{j}) \right| \\
   &=& \left|\frac 1 2   \sum_{ij} [ \sum_{k\ell } A_{k i} A_{\ell j} \nabla_{k\ell }\tilde{f}_y(x_c) + \nabla_{ij} g(x_c)] (x_i-x^\star_{i})(x_j-x^\star_{j}) \right| \\
&\leq& \frac {1} 2   \left[   C_{\tilde{f},2}   ||A (x-x^\star)||_1^2    + \nu   C_{g2}  ||x-x^\star||_1^2 \right]\\
&\leq& \frac { p} 2   (x-x^\star)^T \left( C_{\tilde{f},2}   A^T A   + \nu C_{g2} I\right) (x-x^\star)\\
&\leq& \frac { \delta p} 2   (x-x^\star)^T \left( C_{\tilde{f},2}   A^T A   + \nu C_{g2} I\right) \mathbf{1} = \delta_b^T (x-x^\star),
\end{eqnarray*}
since $x_i-x^\star_i \in [0,\delta]$.

\vspace{0.1cm}

Thus, we obtain an upper bound
\begin{eqnarray*}
h_y(x) - h_y(x^\star)  &\leq& (b+\delta_b)^T (x-x^\star)
\end{eqnarray*}
and the lower bound:
\begin{eqnarray*}
h_y(x) - h_y(x^\star)  &\geq& (b-\delta_b)^T (x-x^\star).
\end{eqnarray*}

\end{proof}

\begin{proposition}\label{prop:LaplaceApproxAsymptExt}
Let assumptions on $f_y, \, g$ and $\delta$ in Section \ref{sec:AssumeExt} hold.

Assume that $x^\star=0$, $b_i = \nabla_i h_y(x^\star) >0$ for all $i$, and that $\gamma\to 0$ and $\nu \to 0$ as $\tau \to 0$.


Then, for any $\varepsilon \in (0, \delta)$,   such that
  $b_{\min} \varepsilon/\tau \to \infty$,
\begin{eqnarray*}
  && \frac{\int_{\X \setminus B(x^\star, \varepsilon)} e^{-  h_y(x)/\tau} dx}{\int_{\X} e^{-  h_y(x)/\tau} dx} \leqslant
     p   e^{- \bar{b}_{\min}  \varepsilon/(\sqrt{p}\tau) } \frac{1+\Delta_1 }{1+\Delta_0} + \frac{\Delta_0}{1+\Delta_0},
\end{eqnarray*}
and, in particular,
\begin{eqnarray*}
\int_{B(x^\star, \delta)} e^{- [ h_y(x)- h_y(x^\star)]/\tau} dx \geq \tau^{p }  \prod_i b_{i}^{-1} [1+\widetilde\Delta_3],
\end{eqnarray*}
where $\Delta_1$ and $\widetilde\Delta_3$ are defined by
\begin{eqnarray}\label{eq:DefDelta13}
 \Delta_1(\delta, b) &=& -1 +  \prod_i \frac{b_i - \delta_{b,i}}{b_i + \delta_{b,i}} \,\,  \left[   1- e^{-  \min_i \widetilde{b}_{i} \delta/(\sqrt{p}\tau)}  \right]^{-p},
\end{eqnarray}
\begin{eqnarray}\label{eq:defD3b}
  \widetilde\Delta_3(\delta, y) &=&  -1 +  \left[1 +  \max_i \delta_{b, \, i}/b_i\right]^{-p} \left[   1- e^{-  \min_i \widetilde{b}_{i} \delta/(\sqrt{p}\tau)}  \right]^{p}.
 \end{eqnarray}

\end{proposition}

\begin{proof}[Proof of Proposition \ref{prop:LaplaceApproxAsymptExt}.]


Making the change of variables $v =  (x-x^\star)/\tau$ with Jacobian $J = \tau^{p}$, we have
\begin{eqnarray*}
&&\int_{B(x^\star, \delta)\cap \X} e^{-  [h_y(x) - h_{y\exact}(x)]/\tau} dx \geq \tau^{p}
\int_{ B(0, \delta/\tau) \cap (\X - x^\star)} \exp\left\{ - (b+ \delta_b)^T v \right\} dv \\
 & \geq&  \tau^{p} \int_{ [0, \delta/(\sqrt{p}\tau)]^p} \exp\left\{ - \tilde{b}^T v \right\} dv =  \tau^{p} \prod_i \tilde{b}_i^{-1}  \prod_i \left[1- \exp\left\{ - \tilde{b}_i \delta/(\sqrt{p}\tau)  \right\} \right]\\
 &\geq& \tau^{p} \prod_i b_{i}^{-1}   [1+\widetilde{\Delta}_3]
 \end{eqnarray*}
with $\widetilde\Delta_3$ defined by (\ref{eq:defD3b}).
The error $\widetilde\Delta_3 \to 0$ as $\tau \to 0$, since $\delta \to 0$ and $b_{\min} \delta/\tau \to \infty$,  with   $\P_{y\exact}$  probability $\to 1$.


Similarly, we obtain an upper bound on the following integral:
\begin{eqnarray*}
&\int_{(\X \cap B(x^\star, \delta)) \setminus B(x^\star, \varepsilon)}& e^{-  [h_y(x)-h_y(x^\star)]/\tau} dx \leq
\tau^p \int_{ \varepsilon/\tau  \leq   ||v || \leq \delta/\tau , \, v_i \geq 0} \exp\left\{  -  \bar{b}^T v  \right\} dv\\
&\leq&  \tau^p \int_{ ||v|| \geq \varepsilon/\tau  } \exp\left\{  -  \bar{b}^T v  \right\} dv\\
&\leq& \sum_i \tau^p \int_{v_i \geq  \varepsilon/(\sqrt{p}\tau), \, v_j\geq 0 \, \forall j } \exp\left\{  -  \bar{b}^T v  \right\} dv\\
&=& \tau^p \prod_i \bar{b}_i^{-1} \sum_i  \exp\{ - \bar{b}_i \varepsilon/(\sqrt{p}\tau) \} \\
&\leq& p \tau^p \prod_i \bar{b}_i^{-1}  \exp\{ - \min_i \bar{b}_i \varepsilon/(\sqrt{p}\tau) \},
\end{eqnarray*}
where $\bar{b} = b - \delta_b$.
Assume that $\delta$ is small enough so that $\bar{b}_i >0$ for all $i$.

Therefore,
\begin{eqnarray*}
&&\frac{\int_{B(x^\star, \delta) \setminus B(x^\star, \varepsilon)} e^{-  h_y(x)/\tau} dx}{\int_{B(x^\star, \delta)} e^{-  h_y(x)/\tau} dx}
\leq    p  e^{- \bar{b}_{\min}   \,\varepsilon/(\sqrt{p}\tau) } ( 1+\Delta_1(\delta , b)) ,
 \end{eqnarray*}
 where
\begin{eqnarray*}
 \Delta_1(\delta, b) &=& -1 +  \prod_i \frac{b_i - \delta_{b,i}}{b_i + \delta_{b,i}} \, \frac{ \left[   1- e^{-  \min_i \widetilde{b}_{i} \delta/(\sqrt{p}\tau)}  \right]^{p}}{p \, \exp\{ - \min_i \bar{b}_i \varepsilon/(\sqrt{p}\tau) \}}.
  \end{eqnarray*}
Hence, $\Delta_1$ is small if $b_{\min} \delta /\tau\to \infty$ as $\tau \to 0$.



Now we take into account the error of approximating the integral over $\X$ by the integral over $B(x^\star, \varepsilon)$:
\begin{eqnarray*}
\frac{\int_{\X \setminus B(x^\star, \varepsilon)} e^{-  h_y(x)/\tau} dx}{\int_{\X} e^{-  h_y(x)/\tau} dx}& =&
\frac{\int_{\X \cap B(x^\star, \delta) \setminus B(x^\star, \varepsilon)} e^{-  h_y(x)/\tau} dx + \int_{\X \setminus B(x^\star, \delta) } e^{-  h_y(x)/\tau} dx}{\int_{B(x^\star, \delta) \cap \X  } e^{-  h_y(x)/\tau} dx + \int_{\X \setminus B(x^\star, \delta) } e^{-  h_y(x)/\tau} dx}\\
&=&
\frac{\int_{\X \cap  B(x^\star, \delta) \setminus B(x^\star, \varepsilon)} e^{-  h_y(x)/\tau} dx  } {(1+\Delta_0)\int_{B(x^\star, \delta)  \cap \X} e^{-  h_y(x)/\tau} dx } + \frac{\Delta_0}{1+\Delta_0}.
\end{eqnarray*}
Thus, we have the required statement.

\end{proof}


\begin{proof}[Proof of Theorem  \ref{th:ProkhUpperExt}.]

We proceed similarly as in the proof of Theorem  \ref{th:ProkhUpperInt}.

By Strassen's theorem, for any $x$, $\rho_{\rm P} (\mu_{\rm post}(\omega), \delta_{x }) =  \rho_{\rm K} (\xi,
 x)$ where $\xi \sim \mu_{\rm post}(\omega)$. Hence, we find an upper bound on the Ky Fan distance between $\xi$ and $x^\star$.

 Using Proposition
\ref{prop:LaplaceApproxAsymptExt}, we have an upper bound $\varepsilon$ on the Ky Fan distance satisfies
\begin{eqnarray*}
  \frac{\int_{B(x^\star, \delta) \setminus B(x^\star, \varepsilon)} e^{-  h_y(x)/\tau} dx}{\int_{B(x^\star, \delta)} e^{-  h_y(x)/\tau} dx}
  &\leqslant&
\left[\widetilde\Delta_0 + p \exp\left\{  - \frac{\varepsilon  \bar{b}_{\min}  }{\sqrt{p}\tau } \right\}\frac{ 1+\Delta_1 }{1 + \Delta_0} \right] \leq  \varepsilon,
\end{eqnarray*}
where $\widetilde\Delta_0 = \Delta_0/(1+ \Delta_0)$. 

An upper bound on the Ky Fan distance is the smallest $\varepsilon >0$ such that
\begin{eqnarray*}
\widetilde\Delta_0  &\leq& \varepsilon, \\
 p \exp\left\{  - \frac{\varepsilon \bar{b}_{\min} }{\sqrt{p}\tau } \right\} \frac{ 1+\Delta_1 }{1 + \Delta_0}  &\leq&  \varepsilon .
\end{eqnarray*}

The last  inequality implies that as $\tau/\bar{b}_{\rm min} \to 0$, $\varepsilon \to 0$.
Hence, using Lemma \ref{lem:KyFanExp},
we have that
\begin{eqnarray*}
\varepsilon &\leq& - \frac{\tau \sqrt{p}}{\bar{b}_{\min} }\left[ \log\left( \frac{\tau}{ \sqrt{p}\, \bar{b}_{\min}}\right) - \log\left( \frac{ 1+\Delta_1 }{1 + \Delta_0}\right) \right].
\end{eqnarray*}

Therefore, the Ky Fan distance is bounded from above by the maximum of the two expressions:
\begin{eqnarray*}
\hspace{-1.2cm}\rho_{\rm P}(\mu_{\rm post}(\omega), \delta_{x^\star})
 &\leqslant&   \max\left\{   \frac{\Delta_0}{1+\Delta_0}, \quad
- \frac{\tau \sqrt{p}}{\bar{b}_{\min} }\left[ \log\left( \frac{\tau}{ \sqrt{p}\, \bar{b}_{\min}}\right) - \log\left(\frac{ 1+\Delta_1 }{1 + \Delta_0}\right) \right]  \right\}\\
&=&  \max\left\{   \frac{\Delta_0}{1+\Delta_0 }, \quad
- \frac{\tau \sqrt{p}}{\bar{b}_{\min}(\omega) }  \log\left( \frac{\tau}{ \sqrt{p}\, \bar{b}_{\min}(\omega)}\right)  (1+ \Delta_{4}(\delta, Y(\omega)))\right\},
\end{eqnarray*}
where  $\Delta_0=\Delta_0(B(0,\delta))$ is defined by (\ref{eq:DefD0}) and  $\Delta_{4}$ is defined by
\begin{eqnarray}\label{eq:DefDeltaStar}
\Delta_{4}(\delta,   y) &=& \frac{\log\left( (1+\Delta_1)/(1+ \Delta_0)\right)}{\log\left(  \sqrt{p} \,\bar{b}_{\min}(\omega) /\tau \right)}.
\end{eqnarray}


\end{proof}


\begin{proof}[Proof of Theorem  \ref{th:KyFanPostExt}. ]
Now we prove Theorem \ref{th:KyFanPostExt} in the notation defined in the proof of Theorem  \ref{th:ProkhUpperExt}.

We apply Theorem \ref{lem:KyFanDiffBounds} with  $\Omega_1
=\{\omega: \, ||Y(\omega)-y\exact||\leq \rho_{\rm K}(Y,
y\exact)\}$ and $\Omega_2 = \Omega \setminus \Omega_1$ with
$\P(\Omega_2)\leq \rho_{\rm K}(Y, y\exact)$ by the definition of
Ky Fan distance, with the bounds given in Theorem~\ref{th:KyFanPostExt} which we modify to
depend on $y$ only via $||y - y\exact||$. For small enough $\tau$, given that $b_i^\star >0$,
the assumption of the theorems that  $b_i>0$  holds on $\Omega_1$ for small enough $\tau$, as we
shall show below.

The upper bound depends on $y$ via $||y - y\exact||$,
$b(\omega)$,  $\Delta_0$ and $\Delta_1$.


We have that, on $\Omega_1$,
\begin{eqnarray*}
b_i &=&\sum_{j} A_{ji} \nabla_j  \tilde{f}_y(A x^\star) +\nu \nabla_i g(x^\star)\\
 &\geq& \sum_{j} A_{ji} [\nabla_j  \tilde{f}_{y\exact}(A x^\star) - M_{\tilde{f},1} \rho_{\rm K}(Y, y\exact)] +\nu \nabla_i g(x^\star)\\
 &=& b_i^\star -  \rho_{\rm K}(Y, y\exact) M_{\tilde{f},1} \sum_{j} A_{ji},
\end{eqnarray*}
and also
\begin{equation}\label{eq:BLow}
b_i - \delta_{b,i}\geq b_{i}^\star -  [\rho_{\rm K}(Y, y\exact) M_{\tilde{f},1} + \delta p C_{\tilde{f},2} ||A||_{1,1}/2] \sum_{j} A_{j,i} - \nu \delta p C_{g2}/2, 
\end{equation}
Note that if $\sum_{j} A_{j,i} =0$, then $b_i - \delta_{b,i} = \nu [\nabla_i g(x^\star) -\delta p C_{g2}/2]$, i.e. the leading term in the lower bound is of order $\nu$. If $\sum_{j} A_{j,i} \neq 0$, then the leading term in the lower bound is a positive constant $\sum_{j} A_{ji} \nabla_j  \tilde{f}_y(A x^\star)$.
Denote $i^\star = \arg \min_{i} b_{i}^\star$ and assume that $\tau$ and $\delta$ are small enough so that the minimum of the lower bound in \eqref{eq:BLow} is also achieved at $i^\star$.
Introduce $\Delta_{11}$ such that
$$
\Delta_{11} = \left[ \rho_{\rm K}(Y, y\exact)   M_{\tilde{f},1} +  0.5\delta p C_{\tilde{f},2} ||A||_{1,1}\right]\frac{ \sum_{j} A_{j i^\star}}{b_{\min}^\star}  + \delta \frac{\nu C_{g,2}}{2 b_{\min}^\star}.
$$
If $\sum_{j} A_{j,i^\star} = 0$.
$$
\Delta_{11} =   \delta \frac{ \nu C_{g,2}}{2 b_{\min}^\star} =  \delta \frac{ C_{g,2}}{2 \nabla_{i^\star} g(x^\star)}.
$$
 Then, an upper bound on the Ky Fan distance is given by
\begin{eqnarray*}
\varepsilon &\leq& - \frac{\tau \sqrt{p}}{\bar{b}_{\min} } \log\left( \frac{\tau }{  \sqrt{p}  \, \bar{ b}_{\min}}\right) \left(1+ \Delta_4 \right)\\
 &\leq& - \frac{\tau \sqrt{p}}{b_{\min}^\star[1 - \Delta_{11}] }\log\left( \frac{\tau }{  \sqrt{p}\,  b_{\min}^\star [1 - \Delta_{11}] }\right)\left[ 1+  \Delta_4^\star  \right],
\end{eqnarray*}
since the function $ - x \log x$ increases for $x<1/e$.
This bound  on $\varepsilon $ on $\Omega_1$ is independent of $y$.
The error term $\Delta_4^\star$ is given by
$$
\Delta_4^\star = \frac{\log\left( (1+\Delta_1^\star)/(1+ \Delta_0^\star)\right)}{\log\left(  \sqrt{p} \,b^\star_{\min}[1 - \Delta_{11}] /\tau \right)}
$$
Using the lifting Theorem \ref{lem:KyFanDiffBounds}, we have that, for small enough $\tau, \nu$,
\begin{eqnarray*}
\rho_{\rm K}(\mu_{\rm post}, \delta_{x^\star}) &\leq& \max\left\{2\rho_{\rm K}(Y, y\exact),\,\, \Delta_0^\star, \,\, - \frac{\tau \sqrt{p}}{ b_{\min}^\star } \log\left( \frac{\tau  }{  \sqrt{p}  b_{\min}^\star }\right)(1 + \Delta_5^\star) \right\},
\end{eqnarray*}
where
\begin{eqnarray*}
\Delta_5^\star &=& -1+\frac{ 1+\Delta_4^\star }{1 - \Delta_{11}} \left(1 -    \frac{ \log (1 - \Delta_{11}) }{\log\left( \frac{\tau }{  \sqrt{p}\,  b_{\min}^\star   }\right)} \right) .
\end{eqnarray*}
Thus, we have the statement of Theorem \ref{th:KyFanPostExt}.

\end{proof}

%% file: Aux_LemmasIP.tex

\subsection{Auxiliary results}


\begin{lemma}\label{lem:UpperBoundXG_gen}
Under the setup of Section~\ref{sec:SectionSetup}, under assumptions that $[A^T V_y(x) A : B(x)]$ is of full rank and $x^\star$ is an interior point of $\X$,
\begin{eqnarray*}
\hspace{-1.5cm}||H^{-1}_y(x)|| &=& [\min ( \lambda_{\min,\, \PA}(A^T V_y(x ) A + \nu B(x))  , \nu \lambda_{\min, \, I-\PA}(B(x))  )]^{-1}\\
\hspace{-1.5cm} ||H_y(x^\star)^{-1}\nabla h_y(x^\star)|| &\leqslant& \frac{ ||\PA \,
\nabla f_y(x^\star)|| + \nu  ||\PA \, \nabla g(x^\star) ||}{ \lambda_{\rm min, \PA}(A^T V_y(x^\star) A  +  \nu B(x^\star))},
\end{eqnarray*}
where $\lambda_{\min, \,  P}(B(x )) = \min_{||v||=1, \, Pv=v} ||B(x ) v||$ is the smallest eigenvalue of $B(x)$ on the range of $P$.
\end{lemma}

\begin{proof}[Proof of Lemma~\ref{lem:UpperBoundXG_gen}.]

The norm of $H^{-1}$ is given by
\begin{eqnarray*}
\hspace{-1.5cm}||H^{-1}|| &=&   [\lambda_{\min}(A^T V A  +
\nu  B)]^{-1} =   [\min_{||x||=1}
||(A^T V A + \nu  B)x||]^{-1}\\
&=&  [\min_{||x||=1}
||(A^T V A + \nu  B)\PA x +\nu  B )(I-\PA) x ||]^{-1}\\
&=& [\min(\min_{||x||=1, \PA x=x}
||(A^T V A + \nu  B)\PA x||, \min_{||x||=1, (I-\PA)x=x} \nu  ||B (I-\PA) x || )]^{-1}\\
&=& [\min ( \lambda_{\min,\, \PA}(A^T V A +
\nu  B) , \nu
\lambda_{\min, \, I-\PA}(B)  )]^{-1}.
\end{eqnarray*}
Note that since we assumed that $V_{y}(x^\star)$ is of full
rank with high probability, the projection on the range of $A^T$ coincides with the
projection on the range of $A^T V_{y}(x^\star) A$ (Lemma~\ref{lem:Interlace}, (iii)).

Now we find an upper bound on $||H_{y\exact}(x^\star)^{-1}\nabla
h_y(x^\star)||$ using the first statement in
Lemma~\ref{lem:Interlace}:
\begin{eqnarray*}
\hspace{-1.5cm} ||H_{y\exact}(x^\star)^{-1}\nabla
h_y(x^\star)|| &=&  ||H_{y\exact}(x^\star)^{-1} \left(\nabla f_y(x^\star) + \nu  \nabla g(x^\star)
\right)||\\ &\leqslant&    ||H_{y\exact}(x^\star)^{-1}||_{\PA}
||\PA(\nabla f_y(x^\star) + \nu  \nabla g(x^\star))
||\\
&+& ||H_{y\exact}(x^\star)^{-1}||_{I-\PA}
||(I-\PA)[\nabla f_y(x^\star) + \nu
\nabla g(x^\star)]
|| \\
&\leqslant& \frac{||\PA[\nabla f_y(x^\star) +
\nu  \nabla g(x^\star) ]||}{\lambda_{\min, \PA}(A^T V_{y\exact}(x^\star) A +
\nu  B(x^\star)) }\\  &+&
  \frac 1 {\lambda_{min, I-\PA}(B)} \tau
||(I-\PA)[\nabla f_y(x^\star) + \nu
\nabla g(x^\star)] ||\\
&\leqslant& \frac{\left[  ||\PA \, \nabla f_y(x^\star)|| +
 \nu ||\PA \, \nabla g(x^\star) ||\right]}{ \lambda_{\min, \PA}(A^T V_{y\exact}(x^\star) A +\nu  B(x^\star))}\\
 &+& \frac { \nu^{-1}  ||(I-\PA) \, \nabla f_y(x^\star)|| + ||(I-\PA) \, \nabla g(x^\star) ||} {\lambda_{\min, I-\PA}(B(x^\star))} ,
\end{eqnarray*}
and we have the last statement of the lemma due to
$$
(I-\PA) \,\nabla f_y(x^\star) = (I-\PA) \, A^T \nabla \tilde{f}_y(Ax^\star) =0
$$
and $(I-\PA) \, \nabla g(x^\star) = 0$ due to the Karush-Kuhn-Tucker conditions since $x^\star$ is the solution  of the minimisation problem \eqref{eq:xstar} and  $x^\star$ is an interior point of $\X$.

\end{proof}


\begin{lemma}\label{lem:Interlace}
\begin{enumerate}
\item $||(C+\delta I)^{-1} x|| \leqslant (\delta +
    \lambda_k(C))^{-1} || P_C x|| + \delta^{-1} ||(I-P_C)x||$

where $k=\rank(C)$ and $\lambda_k(C)$ is the smallest
positive eigenvalue of $C$, and $P_C = C^\dag C$ is the
corresponding projection matrix.

\item Cauchy's interlacing theorem \cite{Bhatia}: let
    $C=C^T$ be a $n\times n$ matrix, $L$ any $n-k$ dimensional
    linear subspace, and $C_L = P_L C P_L$. Then, for any
    $j=1,\dots,n-k$,
$$
\lambda_j(C) \geqslant \lambda_j(C_L) \geqslant \lambda_{j+k} (C),
$$
where $\lambda_1(M) \geq \lambda_2(M) \geq \ldots \geq \lambda_{n}(M)$ are ordered eigenvalues of a matrix $M$. 

\item $\lambda_{\rm min\, pos}(A^T D A) \geqslant \min_{D_i>0}
    D_i \lambda_{\rm min\, pos} (A^T A)$ where $D$ is a
    diagonal matrix with non-negative entries $D_i$ and $\lambda_{\rm min\, pos}(M)$ is the minimum positive eigenvalue of matrix $M$.

\end{enumerate}

\end{lemma}

\begin{proof}[Proof of Lemma~\ref{lem:Interlace}.]

(i) Follows from (ii).

(iii). $\lambda_j(A^T D A) =\lambda_j(D^{1/2} A A^T D^{1/2} )$, and
for $j \geqslant \rank(A^T D A)= \rank(D^{1/2} A A^T D^{1/2} )$,
$$\lambda_j(D^{1/2} A A^T D^{1/2} ) \geqslant \min_{D_i>0} D_i
\lambda_j(P_D A A^T  P_D ) \geqslant \min_{D_i>0} D_i \lambda_{j +
m}( A A^T ) $$ by Cauchy's interlacing theorem, where $m =
\rank(I-P_D)$, and $D$ is $n\times n$ matrix.  

If $j=r=\rank(\PA P_D)$, $\lambda_r(A^T D A)$ is the smallest
positive eigenvalue of $A^T D A$, and 
$$
j+m = \rank(I-P_D) + \rank(\PA P_D) \geqslant \rank(\PA(I-P_D)) + \rank(\PA P_D) = \rank(\PA).
$$
Hence $\lambda_{r+m}(A^T A) \geqslant \lambda_{\rank(\PA)}(A^T A)$, and the latter is the smallest
positive eigenvalue of $A^T A$.

\end{proof}


\begin{lemma}\label{lem:Knapik}
\begin{enumerate}
\item For $a > 0$, $m>0$, $v>0$, $\nu > 0$,
$$
\sum_{i=1}^n \frac{i^{-a-1}}{(1 + \nu^{-1} i^{-m} )^v} \leq C \nu^v \min(\nu^{-1/m},n)^{(vm -a)_+} [\log n]^{I(a=vm)}.
$$

\item For $a\leq 0$, $m>0$, $v>0$ and $0<  \nu  \leq n^{-m}$,
\begin{eqnarray*}
\sum_{i=1}^n \frac{i^{-a-1}}{(1 + \nu^{-1} i^{-m} )^v} &\leq& C \nu^v  n ^{ vm-a }.
\end{eqnarray*}

\item For $m\leq 0$ and  $\nu > 0$,
\begin{eqnarray*}
\sum_{i=1}^n \frac{i^{-a-1}}{(1 + \nu^{-1} i^{-m} )^v} &\leq& C \nu^v  n^{(vm-a)_+} [\log n]^{ I(a = vm )}.
\end{eqnarray*}

\end{enumerate}

\end{lemma}

\begin{proof}
Assume $m>0$ and denote $N= \lceil \nu^{-1/m} \rceil$. Then, for $v>0$,
\begin{eqnarray*}
\sum_{i=1}^n \frac{i^{-a-1}}{(1 + \nu^{-1} i^{-m} )^v} &\leq&   2^{-v} \nu^{v} \sum_{i=1}^{\min(N,n)}  i^{vm -a -1} + 2^{-v} \sum_{i=\min(N,n)}^n  i^{-a-1}\\
&\leq& C \nu^v \min(\nu^{-1/m},n)^{(vm-a)_+} [\log n]^{I(vm = a)}\\&& + C  \nu^{(a)_+/m} [\log n]^{I(a=0)} I(\nu^{-1/m}<n).
\end{eqnarray*}

If $a>0$ and $\nu^{-1/m} \leq n$, then
\begin{eqnarray*}
\sum_{i=1}^n \frac{i^{-a-1}}{(1 + \nu^{-1} i^{-m} )^v} &\leq&  C \nu^v  \nu^{- (vm -a)_+/m} [\log n]^{I(vm = a)} + C  \nu^{a/m}\\
&=&  C \nu^{\min(v, a/m) }  [\log n]^{I(vm = a)}.
\end{eqnarray*}



If $\nu^{-1/m} > n$, then the sum is bounded by $C  \nu^v  n^{(vm -a)_+} [\log n]^{I(vm = a)}$.

If $m>0$, $v>0$ and $a\leq 0$, then 
\begin{eqnarray*}
\sum_{i=1}^n \frac{i^{-a-1}}{(1 + \nu^{-1} i^{-m} )^v} &\leq&  C \nu^v \min(\nu^{-1/m},n)^{ vm-a }  + C [\log n]^{I(a=0)} I(\nu^{-1/m}<n)\\
&=& C \nu^v n ^{ vm-a }I(\nu^{-1/m}\geq n)  + C [ \nu^{a} + [\log n]^{I(a=0)}] I(\nu^{-1/m}<n).
\end{eqnarray*}
This expression tends to 0 if $  \nu^{-1/m}\geq n$ and $ n = o(\nu^{-v/(vm-a)})$.

If $m\leq 0$ and $v>0$,
\begin{eqnarray*}
\sum_{i=1}^n \frac{i^{-a-1}}{(1 + \nu^{-1} i^{-m} )^v} &\leq& C \nu^v \sum_{i=1}^n  i^{vm-a-1} \\
&\leq& C \nu^v [ I( a > vm) + n^{vm-a} I(a < vm) +\log n I(a = vm ) ]\\
&\leq& C \nu^v  n^{(vm-a)_+} [\log n]^{ I(a = vm )}.
\end{eqnarray*}

\end{proof}



%% file: ArxivUpdated_IP.bbl
\begin{thebibliography}{}

\bibitem[\protect\citeAY{Agapiou {\it et~al}.}{2012a}]{Stuart_Mild}
Agapiou, S., Larsson, S., and Stuart, A.~M. (2012a).
\newblock Posterior contraction rates for the {B}ayesian approach to linear
  ill-posed inverse problems.
\newblock {\em arXiv:1203.5753v2}.

\bibitem[\protect\citeAY{Agapiou {\it et~al}.}{2012b}]{Stuart_Severe}
Agapiou, S., Stuart, A.~M., and Zhang, Y.-X. (2012b).
\newblock Bayesian posterior contraction rates for linear severely ill-posed
  inverse problems.
\newblock {\em arXiv:1210.1563}.

\bibitem[\protect\citeAY{Auranen {\it et~al}.}{2005}]{BIP9}
Auranen, T., Nummenmaa, A., H\"am\"al\"ainen, M., J\"a\"askel\"ainen, I.,
  Lampinen, J., Vehtari, A., and Sams, M. (2005).
\newblock Bayesian analysis of the neuromagnetic inverse problem with $\ell_p$
  norm priors.
\newblock {\em NeuroImage}, {\bf 26}, 870--84.

\bibitem[\protect\citeAY{Bhatia}{1997}]{Bhatia}
Bhatia, R. (1997).
\newblock {\em Matrix analysis}. Springer, New York.

\bibitem[\protect\citeAY{Bissanz {\it et~al}.}{2007}]{Munketal}
Bissanz, N., Hohage, T., Munk, A., and Ruymgaart, F. (2007).
\newblock Convergence rates of general regularization methods for statistical
  inverse problems and applications.
\newblock {\em SIAM J. Numer. Anal.}, {\bf 45}, 2610--36.

\bibitem[\protect\citeAY{Bochkina and Green}{2012}]{BochkinaGreen}
Bochkina, N. and Green, P.~J. (2012).
\newblock The {B}ernstein-von {M}ises theorem for non-regular generalised
  linear inverse problems.
\newblock {\em arXiv:1211.3434}.

\bibitem[\protect\citeAY{Cavalier}{2008}]{Cavalier}
Cavalier, L. (2008).
\newblock Nonparametric statistical inverse problems.
\newblock {\em Inverse Problems}, {\bf 24}.

\bibitem[\protect\citeAY{Cavalier and Koo}{2002}]{CavalierPET}
Cavalier, L. and Koo, J.-Y. (2002).
\newblock Poisson intensity estimation for tomographic data using a wavelet
  shrinkage approach.
\newblock {\em IEEE Trans. on Info. Theory}, {\bf 48}, 2794--802.

\bibitem[\protect\citeAY{Chavent and Kunisch}{1994}]{Kunisch94}
Chavent, G. and Kunisch, K. (1994).
\newblock Convergence of {T}ikhonov regularization for constrained ill-posed
  inverse problems.
\newblock {\em Inverse Problems}, {\bf 10}, 63--76.

\bibitem[\protect\citeAY{Chichignoud}{2012}]{Chichignoud}
Chichignoud, M. (2012).
\newblock Minimax and minimax adaptive estimation in multiplicative regression:
  locally {B}ayesian approach.
\newblock {\em Probability Theory and Related Fields}, {\bf 153}, 543--86.

\bibitem[\protect\citeAY{Cotter {\it et~al}.}{2009}]{BIP7}
Cotter, S., Dashti, M., Robinson, J., and Stuart, A. (2009).
\newblock Bayesian inverse problems for functions and applications to fluid
  mechanics.
\newblock {\em Inverse Problems}, {\bf 25}.

\bibitem[\protect\citeAY{Dashti {\it et~al}.}{2012}]{BIP8}
Dashti, M., Harris, S., and Stuart, A. (2012).
\newblock Besov priors for {B}ayesian inverse problems.
\newblock {\em Inverse Problems and Imaging}, {\bf 6}, 183--200.

\bibitem[\protect\citeAY{Diaconis and Freedman}{1986}]{DiaconisFriedman}
Diaconis, P. and Freedman, D. (1986).
\newblock On the consistency of {B}ayes estimates.
\newblock {\em The Annals of Statistics}, {\bf 14}, 1--67.

\bibitem[\protect\citeAY{Diggle and Hall}{1993}]{DiggleHall}
Diggle, P. and Hall, P. (1993).
\newblock A {F}ourier approach to nonparametric decomvolution of a density
  estimate.
\newblock {\em Jornal of the Royal Statistical Society Series B}, {\bf 55},
  523--31.

\bibitem[\protect\citeAY{Dudley}{2003}]{Dudley}
Dudley, R.~M. (2003).
\newblock {\em Real analysis and probability}. Cambridge University Press,
  Cambridge.

\bibitem[\protect\citeAY{Efendiev {\it et~al}.}{2008}]{BIP3}
Efendiev, Y., Gupta, A.~D., Hwang, K., Ma, X., and Mallick, B. (2008).
\newblock Bayesian partition models for subsurface characterization.
\newblock In {\em Large-scale inverse problems and quantification of
  uncertainty}. Wiley.

\bibitem[\protect\citeAY{Engl {\it et~al}.}{1996}]{IP96}
Engl, H.~W., Hanke, M., and Neubauer, A. (1996).
\newblock {\em Regularization of Inverse Problems}. Kluwer, Dordrecht.

\bibitem[\protect\citeAY{Engl {\it et~al}.}{2005}]{EHK05}
Engl, H.~W., Hofinger, A., and Kindermann, S. (2005).
\newblock Convergence rates in the {P}rokhorov metric for assessing uncertainty
  in ill-posed problems.
\newblock {\em Inverse Problems}, {\bf 21}, 399--412.

\bibitem[\protect\citeAY{Fan}{1944}]{Fan}
Fan, K. (1944).
\newblock Entfernung zweier zuf{\"a}lliger {G}r{\"o}\ss en und die {K}onvergenz
  nach {W}ahrscheinlichkeit.
\newblock {\em Mathematische Zeitschrift}, {\bf 49}, 681--3.

\bibitem[\protect\citeAY{Ghosal {\it et~al}.}{2000}]{GhosalGV}
Ghosal, S., Ghosh, J.~K., and van~der Vaart, A.~W. (2000).
\newblock Convergence rates of posterior distributions.
\newblock {\em The Annals of Statistics}, {\bf 28}, (2), 500--31.

\bibitem[\protect\citeAY{Ghosal and Samanta}{1995}]{GS95}
Ghosal, S. and Samanta, T. (1995).
\newblock Asymptotic behaviour of {B}ayes estimates and posterior distributions
  in multiparameter nonregular cases.
\newblock {\em Math. Methods Statist.}, {\bf 4}, 361--88.

\bibitem[\protect\citeAY{Ghosh {\it et~al}.}{1994}]{GGS94}
Ghosh, J.~K., Ghosal, S., and Samanta, T. (1994).
\newblock Stability and convergence of posterior in non-regular problems.
\newblock In {\em Statistical Decision Theory and Related Topics}, pp.~183--99.
  Springer.

\bibitem[\protect\citeAY{Green}{1990}]{Green}
Green, P.~J. (1990).
\newblock Bayesian reconstructions from emission tomography data using a
  modified {EM} algorithm.
\newblock {\em IEEE Transactions on Medical Imaging}, {\bf 9}, 84--93.

\bibitem[\protect\citeAY{Hofinger and Pikkarainen}{2007}]{HofingerP:07}
Hofinger, A. and Pikkarainen, H.~K. (2007).
\newblock Convergence rate for the {B}ayesian approach to linear inverse
  problems.
\newblock {\em Inverse Problems}, {\bf 23}, (6), 2469--84.

\bibitem[\protect\citeAY{Hofinger and Pikkarainen}{2009}]{HofingerP:09}
Hofinger, A. and Pikkarainen, H.~K. (2009).
\newblock Convergence rates for linear inverse problems in the presence of an
  additive normal noise.
\newblock {\em Stochastic Analysis and Applications}, {\bf 27}, (2), 240--57.

\bibitem[\protect\citeAY{Johnstone}{1999}]{JohnstoneInv}
Johnstone, I.~M. (1999).
\newblock Wavelet shrinkage for correlated data and inverse problems:
  adaptivity results.
\newblock {\em Statistica Sinica}, {\bf 9}, 51--83.

\bibitem[\protect\citeAY{Johnstone and Silverman}{1990}]{JSpet}
Johnstone, I.~M. and Silverman, B.~W. (1990).
\newblock Speed of estimation in positron emission tomography and related
  inverse problems.
\newblock {\em The Annals of Statistics}, {\bf 18}, 251--80.

\bibitem[\protect\citeAY{Johnstone and Silverman}{1991}]{JSDiscr}
Johnstone, I.~M. and Silverman, B.~W. (1991).
\newblock Discretization effects in statistical inverse problems.
\newblock {\em Journal of Complexity}, {\bf 7}, 1--34.

\bibitem[\protect\citeAY{Kaipio and Fox}{2010}]{BIP6}
Kaipio, J. and Fox, C. (2010).
\newblock The {B}ayesian framework for inverse problems in heat transfer.
\newblock {\em Heat Transfer Eng}, {\bf 32}, 718--53.

\bibitem[\protect\citeAY{Kaipio {\it et~al}.}{1999}]{BIP4}
Kaipio, J., Kolehmainen, V., Vauhkonen, M., and Somersalo, E. (1999).
\newblock Inverse problems with structural prior information.
\newblock {\em Inverse Problems}, {\bf 15}, 713--29.

\bibitem[\protect\citeAY{Kaipio and Somersalo}{2007}]{KaipioDiscr}
Kaipio, J. and Somersalo, E. (2007).
\newblock Statistical inverse problems: Discretization, model reduction and
  inverse crimes.
\newblock {\em Journal of Computational and Applied Mathematics}, {\bf 198},
  493--504.

\bibitem[\protect\citeAY{Kaipio and Somersalo}{2004}]{KaipioSomersalo}
Kaipio, J.~P. and Somersalo, E. (2004).
\newblock {\em Statistical and computational inverse problems}. Springer, New
  York.

\bibitem[\protect\citeAY{Knapik {\it et~al}.}{2011}]{vdVaart2011}
Knapik, B.~T., van~der Vaart, A.~W., and van Zanten, J.~H. (2011).
\newblock {B}ayesian inverse problems with {G}aussian priors.
\newblock {\em The Annals of Statistics}, {\bf 39}, (5), 2626--57.

\bibitem[\protect\citeAY{Mair and Ruymgaart}{1996}]{MairRuym}
Mair, A. and Ruymgaart, F. (1996).
\newblock Statistical inverse estimation in {H}ilbert scales.
\newblock {\em SIAM J. Appl. Math.}, {\bf 56}, 1424--44.

\bibitem[\protect\citeAY{Math\'e and Pereverzev}{2001}]{MathePerDiscr}
Math\'e, P. and Pereverzev, S.~V. (2001).
\newblock Optimal discretization of inverse problems in {H}ilbert scales.
  regularization and self-regularization of projection methods.
\newblock {\em SIAM J. Numer. Anal.}, {\bf 38}, 1999--2021.

\bibitem[\protect\citeAY{Math\'e and Pereverzev}{2006}]{MPdiscr}
Math\'e, P. and Pereverzev, S.~V. (2006).
\newblock Regularization of some linear ill-posed problems with discretized
  random noisy data.
\newblock {\em Math. Comput.}, {\bf 75}, 1913--29.

\bibitem[\protect\citeAY{Nelder and Wedderburn}{1972}]{NelderW}
Nelder, J.~A. and Wedderburn, R. W.~M. (1972).
\newblock Generalized linear models.
\newblock {\em Journal of the Royal Statistical Society. Series A (General)},
  {\bf 135}, (3), 370--84.

\bibitem[\protect\citeAY{Nychka and Cox}{1989}]{NychkaCox}
Nychka, D.~W. and Cox, D. (1989).
\newblock Convergence rates for regularized solutions of integral equtions from
  discrete noisy data.
\newblock {\em The Annals of Statistics}, {\bf 17}, 556--72.

\bibitem[\protect\citeAY{Ray}{2013}]{Ray_Inv}
Ray, K. (2013).
\newblock Bayesian inverse problems with non-conjugate priors.
\newblock {\em arxiv}.

\bibitem[\protect\citeAY{Reynaud-Bouret and Rivoirard}{2010}]{Bouret}
Reynaud-Bouret, P. and Rivoirard, V. (2010).
\newblock Near optimal thresholding estimation of a {P}oisson intensity on the
  real line.
\newblock {\em Electronic Journal of Statistics}, {\bf 4}, 172--238.

\bibitem[\protect\citeAY{R\"over {\it et~al}.}{2007}]{BIP2}
R\"over, C., Meyer, R., Guidi, G.~M., Vicer\'e, A., and Christensen, N. (2007).
\newblock Coherent {B}ayesian analysis of inspiral signals.
\newblock {\em Class. Quantum Grav.}, {\bf 24}, 607--15.

\bibitem[\protect\citeAY{Stuart}{2010}]{StuartReview}
Stuart, A.~M. (2010).
\newblock Inverse problems: a {B}ayesian perspective.
\newblock {\em Acta Numer.}, {\bf 19}, 451--559.

\bibitem[\protect\citeAY{Tarantola}{2006}]{Tarantola}
Tarantola, A. (2006).
\newblock Popper, {B}ayes and the inverse problem.
\newblock {\em Nature Physics}, {\bf 2}, 492--4.

\bibitem[\protect\citeAY{Tikhonov}{1963}]{Tikh63}
Tikhonov, A.~N. (1963).
\newblock Solution of incorrectly formulated problems and the regularization
  method.
\newblock {\em Sov. Dokl.}, {\bf 4}, 1035--8.

\bibitem[\protect\citeAY{van~der Vaart and van Zanten}{2009}]{vdVaartAdaptive}
van~der Vaart, A.~W. and van Zanten, J.~H. (2009).
\newblock Adaptive {B}ayesian estimation using a {G}aussian random field with
  inverse {G}amma bandwidth.
\newblock {\em The Annals of Statistics}, {\bf 37}, 2655--75.

\bibitem[\protect\citeAY{Voutilainen and Kaipio}{2009}]{BIP5}
Voutilainen, A. and Kaipio, J. (2009).
\newblock Model reduction and pollution source identification from remote
  sensing data.
\newblock {\em Inverse Probl Imaging}, {\bf 3}, 711--30.

\bibitem[\protect\citeAY{Wolpert and Ickstadt}{2004}]{BIP1}
Wolpert, R.~L. and Ickstadt, K. (2004).
\newblock Reflecting uncertainty in inverse problems: a {B}ayesian solution
  using {L}\'evy processes.
\newblock {\em Inverse Problems}, {\bf 20}, 1759--71.

\end{thebibliography}
